\documentclass[12pt]{amsart}

\usepackage[mathscr]{eucal}
\usepackage{amsmath,amssymb,amscd,amsthm}
\usepackage{color}
\usepackage{epsfig}
\newtheorem{theorem}{Theorem}

\newtheorem{definition}{Definition}

\newtheorem{criterion}{Criterion}
\newtheorem{proposition}{Proposition}
\theoremstyle{definition}
\newtheorem{remark}{Remark}

\theoremstyle{plane}

\def \beq{ \begin{equation} }
\def \eeq{\end{equation}}

\title{Relative equilibria in the 3-dimensional curved $n$-body problem}
\begin{document}
\maketitle
\markboth{Florin Diacu}{Relative equilibria in the 3-dimensional curved $n$-body problem}
\vspace{-0.2cm}
\author{\begin{center}
Florin Diacu\\
\smallskip
{\footnotesize Pacific Institute for the Mathematical Sciences\\
and\\
Department of Mathematics and Statistics\\
University of Victoria\\
P.O.~Box 3060 STN CSC\\
Victoria, BC, Canada, V8W 3R4\\
diacu@math.uvic.ca\\
}\end{center}

}


\begin{center}
\today
\end{center}

\begin{abstract}
We consider the $3$-dimensional gravitational $n$-body problem, $n\ge 2$, in spaces of constant Gaussian curvature $\kappa\ne 0$, i.e.\ on spheres ${\mathbb S}_\kappa^3$, for $\kappa>0$, and on hyperbolic manifolds ${\mathbb H}_\kappa^3$, for $\kappa<0$. Our goal is to define and study relative equilibria, which are orbits whose mutual distances remain constant in time. We also briefly discuss the issue of singularities in order to avoid impossible configurations. We derive the equations of motion and define six classes of relative equilibria, which follow naturally from the geometric properties of ${\mathbb S}_\kappa^3$ and ${\mathbb H}_\kappa^3$. Then we prove several criteria, each expressing the conditions for the existence of a certain class of relative equilibria, some of which have a simple rotation, whereas others perform a double rotation, and we describe their qualitative behaviour. In particular, we show that in ${\mathbb S}_\kappa^3$ the bodies move either on circles or on Clifford tori, whereas in ${\mathbb H}_\kappa^3$ they move either on circles or on hyperbolic cylinders. Then we construct concrete examples for each class of relative equilibria previously described, thus proving that these classes are not empty. We put into the evidence some surprising orbits, such as those for which a group of bodies stays fixed on a great circle of a great sphere of ${\mathbb S}_\kappa^3$, while the other bodies rotate uniformly on a complementary great circle of another great sphere, as well as a large class of quasiperiodic relative equilibria, the first such non-periodic orbits ever found in a 3-dimensional $n$-body problem. Finally, we briefly discuss other research directions and the future perspectives in the light of the results we present here.
\end{abstract}

\noindent --------------------------------------------------------------------------------------------

{\footnotesize\noindent {\it 2000 Mathematics Subject Classification}. Primary 70F10; Secondaries 34C25, 37J45. }

\newpage


\tableofcontents

\newpage

\part{\rm INTRODUCTION}

In this introductory part, we will define the problem, explain its importance, tell its history, and outline the main results we obtained.

\section{The problem}

We consider in this paper a natural extension of the Newtonian $n$-body problem to spaces of non-zero constant curvature. Since there is no unique way of generalizing the classical equations of motion such that they can be recovered when the space in which the bodies move flattens out, we seek a potential that satisfies the same basic properties as the Newtonian potential in its simplest possible setting, that of one body moving around a fixed centre, which is a basic problem in celestial mechanics known as the Kepler problem. Two basic properties stick out in this case: the Newtonian potential is a harmonic function in 3-dimensional space, i.e.\ it satisfies Laplace's equation, and it generates a central field in which all bounded orbits are closed, a result proved by Joseph Louis Bertrand in 1873, \cite{Ber}. 

On one hand, the curved potential we define in Section \ref{equationsofmotion} approaches the classical Newtonian potential when the curvature tends to zero, whether through positive or negative values. In the case of the Kepler problem, on the other hand, this potential satisfies Bertrand's property and is a solution of the Laplace-Beltrami equation, \cite{Kozlov}, the natural generalization of Laplace's equation to Riemannian and pseudo-Riemannian manifolds, which include the spaces of constant curvature $\kappa\ne 0$ we are interested in, namely the spheres ${\mathbb S}_\kappa^3$, for $\kappa>0$, and the hyperbolic manifolds ${\mathbb H}_\kappa^3$, for $\kappa<0$. For simplicity, from now on we will refer to the $n$-body problem defined in these spaces of constant curvature as {\it the curved $n$-body problem}.

In the Euclidean case, the Kepler problem and the 2-body problem are equivalent. The reason for this overlap is the existence of the integrals of the centre of mass and of the linear momentum. It can be shown with their help that the equations of motion appear identical, whether the origin of the coordinate system is fixed at the centre of mass or fixed at one of the two bodies. For non-zero curvature, however, things change. As we will later see, the equations of motion of the curved $n$-body problem lack the integrals of the centre of mass and of the linear momentum, which prove to characterize only the Euclidean case. Consequently the curved Kepler problem and the curved 2-body problem are not equivalent anymore. It turns out that, as in the Euclidean case, the curved Kepler problem is integrable, but, unlike in the Euclidean case, the curved 2-body problem is not, \cite{Shc1}. As expected, the curved $n$-body problem is not integrable for $n\ge 3$, a property that is well-known to be true in the Euclidean case.

Our main goal is to find solutions that are specific to each of the spaces corresponding to $\kappa<0$, $\kappa=0$, and $\kappa>0$. We already succeeded to do that in some previous papers, but only in the 2-dimensional case, \cite{DiacuPol}, \cite{DiacuIntrinsic}, \cite{Diacu1}, \cite{Diacu001}, \cite{Diacu002}. In this paper, we will see that the 3-dimensional curved $n$-body problem puts into the evidence even more differences between the qualitative behaviour of orbits in each of these spaces.

\section{Importance}

The curved $n$-body problem lays a bridge between the theory of dynamical systems and the geometry and topology of 3-dimensional manifolds of constant curvature, as well as with the theory of regular polytopes. As the results obtained in this paper show, many of which are surprising and nonintuitive, the geometry of the configuration space (i.e.\ the space in which the bodies move) strongly influences their dynamics, including some important qualitative properties, stability among them. Some topological concepts, such as the Hopf fibration and the Hopf link, some geometric objects, such as the pentatope, the Clifford torus,  and the hyperbolic cylinder, or some geometric properties, such as the Hegaard splitting of genus 1, become essential tools for understanding the gravitational motion of the bodies. Since we provide here only a first study in this direction of research, reduced to the simplest orbits the equations of motion have, namely the relative equilibria, we expect that many more connections between dynamics, geometry, and topology will be discovered in the near future through a deeper exploration of the 3-dimensional case.

\section{History}

The first researchers that took the idea of gravitation beyond the Euclidean space were Nikolai Lobachevsky and J\'anos Bolyai, who also laid the foundations of hyperbolic geometry independently of each other. In 1835, Lobachev\-sky proposed a Kepler problem in the 3-dimensional hyperbolic space, ${\mathbb H}^3$, by defining an attractive force proportional to the inverse area of the 2-dimensional sphere of radius equal to the distance between bodies, \cite{Lob}. Independently of him, and at about the same time, Bolyai came up with a similar idea, \cite{Bol}. Both of them understood the intimate connection between geometry and physical laws, a relationship that proved very prolific ever since.

These co-discoverers of the first non-Euclidean geometry had no followers in their pre-relativistic attempts until 1860, when Paul Joseph Serret\footnote{Paul Joseph Serret (1827-1898) should not be confused with another French mathematician, Joseph Alfred Serret (1819-1885), known for the Frenet-Serret formulas of vector calculus.} extended the gravitational force to the sphere ${\mathbb S}^2$ and solved the corresponding Kepler problem, \cite{Ser}. Ten years later, Ernst Schering revisited Lobachevsky's gravitational law for which he obtained an analytic expression given by the
curved cotangent potential we study in this paper, \cite{Sche}. Schering also wrote that Lejeune Dirichlet had told some friends to have dealt with the same problem during his last years in Berlin\footnote{This must have happened around 1852, as claimed by Rudolph Lipschitz, \cite{Lip72}.}, \cite{Sche1}, but Dirichlet never published anything in this direction, and we found no evidence of any manuscripts he would have left behind in which he dealt with this particular topic. In 1873, Rudolph Lipschitz considered the problem in ${\mathbb S}^3$, but defined a potential proportional to $1/\sin\frac{r}{R}$, where $r$ denotes the distance between bodies and $R$ is the curvature radius, \cite{Lip}. He obtained the general solution of this problem only in terms of elliptic functions. His failure to provide an explicit formula, which could have him helped draw some conclusions about the motion of the bodies, showed the advantage of Schering's approach. 

In 1885, Wilhelm Killing adapted Lobachevsky's gravitational law to ${\mathbb S}^3$ and defined an extension of the Newtonian force given by the inverse area of a 2-dimensional sphere (in the spirit of Schering), for which he proved a generalization of Kepler's three laws, \cite{Kil10}. Then an important step forward took place. In 1902, Heinrich Liebmann\footnote{Although he signed his papers and books as Heinrich Liebmann, his full name was Karl Otto Heinrich Liebmann (1874-1939). He did most of his work in Munich and Heidelberg, where he became briefly the university's president, before the Nazis forced him to retire. A remembrance colloquium was held in his honour in Heidelberg in 2008. On this occasion, Liebmann's son, Karl-Otto Liebman, donated to the University of Heidelberg an oil portrait of his father, painted by Adelheid Liebmann, \cite{Sgries}} showed that the orbits of the Kepler problem are
conics in ${\mathbb S}^3$ and ${\mathbb H}^3$ and further generalized Kepler's three laws to $\kappa\ne 0$, \cite{Lie1}. One year later, Liebmann proved ${\mathbb S}^2$- and ${\mathbb H}^2$-analogues of Bertrand's theorem, \cite{Ber}, \cite{Win}, which states that for the Kepler problem there exist only two analytic central potentials in the Euclidean space for which all bounded orbits are closed, \cite{Lie2}. He thus made the most important contributions of the early times to the model, by showing that the approach of Bolyai and Lobachevsky led to a natural extension of Newton's gravitational law to spaces of constant curvature. 

Liebmann also summed up his results in the last chapter of a book on hyperbolic geometry published in 1905, \cite{Lie3}, which saw two more editions, one in 1912 and the other in 1923. Intriguing enough, in the third edition of his book he replaced the constant-curvature approach with relativistic considerations.  

Liebmann's change of mind about the importance of the constant-curvature approach may explain why this direction of research was ignored in the decades immediately following the birth of special and general relativity. The reason for this neglect was probably connected to the idea that general relativity could allow the study of 2-body problems on manifolds with variable Gaussian curvature, so the constant-curvature case appeared to be outdated. Indeed, although the most important subsequent success of relativity was in cosmology and related fields, there were attempts to discretize Einstein's equations and define a gravitational $n$-body problem. Remarkable in this direction were the contributions of Jean Chazy, \cite{Cha}, Tullio Levi-Civita, \cite{Civ}, \cite{Civita}, Arthur Eddington, \cite{Edd}, Albert Einstein, Leopold 
Infeld\footnote{A vivid description of the collaboration between Einstein and Infeld appears in Infeld's autobiographical book \cite{Infeld}.}, Banesh Hoffmann, \cite{Ein}, and Vladimir Fock, \cite{Fock}. Subsequent efforts led to refined post-Newtonian approximations (see, e.g., \cite{Soff1}, \cite{Soff2}, \cite{Soff3}), which prove very useful in practice, from understanding the motion of artificial satellites---a field with applications in geodesy and geophysics---to using the Global Positioning System (GPS), \cite{Soff4}.

But the equations of the $n$-body problem derived from relativity are highly complicated even for $n=2$, and they are not prone to analytical studies similar to the ones done in the classical case. This is probably the reason why the need for some simpler equations revived the research on the motion of 2 bodies in spaces of constant curvature.

Starting with 1940, Erwin Schr\"odinger developed a quantum-mecha\-nical analogue of the Kepler problem in ${\mathbb S}^2$, \cite{Schr}. Schr\"odinger used the same cotangent potential of Schering and Liebmann, which he deemed to be the natural extension of Newton's law to the sphere\footnote{``The correct form of [the] potential (corresponding to $1/r$ of the flat space) is known to be $\cot\chi$,'' \cite{Schr}, p.~14.}. Further results in this direction were obtained by Leopold Infeld, \cite{Inf}, \cite{Ste}. In 1945, Infeld and his student Alfred Schild extended this problem to spaces of constant negative curvature using a potential given by the hyperbolic cotangent of the distance, \cite{InfS}. A comprehensive list of the above-mentioned works also appears in \cite{Sh}, except for Serret's book, \cite{Ser}, which is not mentioned. A bibliography of works about mechanical problems in spaces of constant curvature is given in \cite{Shch2}, \cite{Shc1}.

The Russian school of celestial mechanics led by Valeri Kozlov also studied the curved 2-body problem given by the cotangent potential and considered related problems in spaces of constant curvature starting with the 1990s, \cite{Kozlov}. An important contribution to the case $n=2$ and the Kepler problem belongs to Jos\'e Cari\~nena, Manuel Ra\~nada, and Mariano Santander, who provided a unified approach in the framework of differential geometry with the help of intrinsic coordinates, emphasizing the dynamics of the cotangent potential in ${\mathbb S}^2$ and ${\mathbb H}^2$, \cite{Car} (see also \cite{Car2}, \cite{Gut}). They also proved that, in this unified context, the conic orbits known in Euclidean space extend naturally to spaces of constant curvature, in agreement with the results obtained by Liebmann, \cite{Sh}. Moreover, the authors used the rich geometry of the hyperbolic plane to search for new orbits, whose existence they either proved or conjectured.

Inspired by the work of Cari\~nena, Ra\~nada, and Santander, we proposed a new setting for the problem, which allowed us an easy derivation of the equations of motion for any $n\ge 2$ in terms of extrinsic coordinates, \cite{Diacu1}, \cite{Diacu001}. The combination of two main ideas helped us achieve this goal: the use of Weierstrass's hyperboloid model of hyperbolic geometry and the application of the variational approach of constrained Lagrangian dynamics, \cite{Gel}. Although we obtained the equations of motion for any dimension, we explored so far only the 2-dimensional case. In \cite{Diacu1} and \cite{Diacu001}, we studied relative equilibria and solved Saari's conjecture in the collinear case (see also \cite{Diacu2}, \cite{Diacu3}), in \cite{Diacu002} and \cite{Diacu1bis} we studied the singularities of the curved $n$-body problem, in \cite{Diacu4} we gave a complete classification of the homographic solutions in the 3-body case, and in \cite{DiacuPol} we obtained some results about polygonal  homographic orbits, including a generalization of the Perko-Walter-Elmabsout theorem,
\cite{Elma}, \cite{Perko}. Knowing the correct form of the equations of motion, Ernesto P\'erez-Chavela and J.\ Guadalupe Reyes Victoria succeeded to derive them in intrinsic coordinates in the case of positive curvature and showed them to be equivalent with the extrinsic equations obtained with the help of variational methods, \cite{Perez}. We are currently developing an intrinsic approach for negative curvature, whose analysis appears to be more complicated than the study of the positive curvature case, \cite{DiacuIntrinsic}.

\section{Results}

This paper is a first attempt to study the 3-dimensional curved $n$-body problem in the general context described in the previous section, with the help of the equations of motion written in extrinsic coordinates, as they were obtained in \cite{Diacu1} and \cite{Diacu001}. We are mainly concerned with understanding the motion of the simplest possible orbits, the relative equilibria, which move like rigid bodies, by maintaining constant mutual distances for all time.  

We structured this paper in 6 parts. This is Part 1, which  defines the problem, explains its importance, gives its historical background, describes the structure of the paper, and presents the main results. In Part 2 we set the geometric background and derive the equations of motion in the 3-dimensional case. In Part 3 we discuss the isometric rotations of the spaces of constant curvature in which the bodies move and, based on this development, define several natural classes of relative equilibria. We also analyze the occurrence of fixed points. In Part 4 we prove criteria for the existence of relative equilibria and a give qualitative description of how these orbits behave. In Part 5 we provide concrete examples of relative equilibria for all the classes of orbits previously discussed. Part 6 concludes the paper with a short analysis of the recent achievements obtained in this direction of research and a presentation of the future perspectives our results offer. 

Each part contains several sections, which are numbered increasingly throughout the paper. Sections 1, 2, 3, and 4 form Part 1. In Section 5, which starts Part 2, we lay the background for the subsequent developments that appear in this work. We thus introduce in 5.1 Weierstrass's model of hyperbolic geometry, outline in 5.2 its history, define in 5.3 some basic concepts of geometric topology, describe in 5.4 the metric in curved space, and introduce in 5.5 some functions that unify circular and hyperbolic trigonometry. In Section 6, we start with 6.1 in which we give the analytic expression of the potential, derive in 6.2 Euler's formula, present in 6.3 the theory of constrained Lagrangian dynamics and then use it in 6.4 to obtain the equations of motion of the $n$-body problem on 3-dimensional manifolds of constant Gaussian curvature, $\kappa\ne 0$, i.e.\  spheres, ${\mathbb S}_\kappa^3$, for $\kappa>0$, and hyperbolic manifolds, ${\mathbb H}_\kappa^3$, for $\kappa<0$. Then we prove in 6.5 that the equations of motion are Hamiltonian and show in 6.6, Proposition 1, that the 2-dimensional manifolds of the same curvature, ${\mathbb S}_\kappa^2$ and ${\mathbb H}_\kappa^2$, are invariant for these equations. In Section 7 we prove the existence and derive the expressions of the 7 basic integrals this Hamiltonian system has, namely we come up in 7.1 with 1 integral of energy and in 7.2 with the 6 integrals of the total angular momentum. Unlike in the Euclidean case, there are no integrals of the centre of mass and linear momentum. Although we don't aim to study the singularities of the equations of motion here, we describe them in Section 8, Proposition 2, in order to avoid singular configurations when dealing with relative equilibria.

Section 9, which starts Part 3, describes the isometric rotations of ${\mathbb S}_\kappa^3$ and ${\mathbb H}_\kappa^3$, a task that helps us understand the various types of transformations and how to use them to define natural classes of relative equilibria. It turns out that we can have: (1) simple elliptic rotations in ${\mathbb S}_\kappa^3$; (2) double elliptic rotations in ${\mathbb S}_\kappa^3$; (3) simple elliptic rotations in ${\mathbb H}_\kappa^3$; (4) simple hyperbolic rotations in ${\mathbb H}_\kappa^3$; (5) double elliptic-hyperbolic rotations in ${\mathbb H}_\kappa^3$; and (6) simple parabolic rotations in ${\mathbb H}_\kappa^3$. Section 10 continues this task by exploring when 2-dimensional spheres, in 10.1, and hypeboloids, in 10.2, of curvature $\kappa$ or different from $\kappa$ are preserved by the isometric rotations of ${\mathbb S}_\kappa^3$ and ${\mathbb H}_\kappa^3$, respectively.
Based on the results we obtain in these sections, we define 6 natural types of relative equilibria in Section 11: $\kappa$-positive elliptic, in 11.1; $\kappa$-positive elliptic-elliptic, in 11.2; $\kappa$-negative elliptic, in 11.3; $\kappa$-negative hyperbolic, in 11.4; $\kappa$-negative elliptic-hyperbolic, in 11.5; and $\kappa$-negative parabolic, in 11.6; then we summarize them all in 11.7. Since relative equilibria can also be generated from fixed-point configurations, we study fixed-point solutions in Section 12, which closes Part 3. It turns out that fixed-point solutions occur in ${\mathbb S}_\kappa^3$, as we prove in 12.1, but not in ${\mathbb H}_\kappa^3$ (Proposition 3 of 12.3) or hemispheres of ${\mathbb S}_\kappa^3$ (Proposition 4 of 12.3). Moreover, some of these fixed points are specific to ${\mathbb S}_\kappa^3$, such as when 5 equal masses lie at the vertices of a regular pentatope, because there is no 2-dimensional sphere on which this configuration occurs, as we show in 12.2.

Section 13 opens Part 4 of this paper with proving 7 criteria for the existence of relative equilibria as well as the nonexistence of $\kappa$-negative parabolic relative equilibria,
the latter in 13.6 (Proposition 5). Criteria 1 and 2 of 13.1 are for $\kappa$-positive elliptic relative equilibria, both in the general case and when they can be generated from fixed-point configurations, respectively. Criteria 3 and 4 of 13.2 repeat the same pattern, but for $\kappa$-positive elliptic-elliptic relative equilibria. Finally, Criteria 5, 6, and 7, proved in 13.3, 13.4, and 13.5, respectively, are for $\kappa$-negative elliptic, hyperbolic, and elliptic-hyperbolic relative equilibria, respectively. In Section 14 we prove the main results of this paper after laying some background of geometric topology. In 14.1 we introduce the concept of Clifford torus. In 14.2 we prove Theorem 1, which describes the general qualitative behaviour of relative equilibria in ${\mathbb S}_\kappa^3$, assuming that they exist, and shows that the bodies could move on circles or on Clifford tori. In 14.3 we prove Theorem 2, which becomes more specific about relative equilibria generated from fixed-point configurations; it shows that we could have orbits for which some points are fixed while others rotate on circles, as well as a large class of quasiperiodic orbits, which are the first kind of non-periodic relative equilibria ever hinted at in celestial mechanics. In Section 15, which closes Part 4, we study what happens in $\mathbb H_\kappa^3$. In 15.1 we introduce the concept of hyperbolic cylinder and in 15.2 we prove Theorem 3, which describes the possible motion of relative equilibria in ${\mathbb H}_\kappa^3$; it shows that the bodies could move only on circles, on geodesics, or on hyperbolic cylinders.

Since none of the above results show that such relative equilibria actually exist, it is necessary to construct concrete examples with the help of Criteria 1 to 7. This is the goal of Part 5, which starts with Section 16, in which we construct 4 classes of relative equilibria with a simple rotation in ${\mathbb S}_\kappa^3$. Example 16.1 provides a Lagrangian solution: the bodies, of equal masses, are at the vertices of an equilateral triangle) in the curved 3-body problem in ${\mathbb S}_\kappa^3$ and rotate on a non-necessarily geodesic circle of a great sphere, the frequency of rotation depending on the masses and on the circle's position. Example 16.2 generalizes the previous construction in the particular case of the curved 3-body problem when the bodies rotate on a great circle of a great sphere. Then, for any acute scalene triangle, we can find masses that lie at its vertices while the triangle rotates with any non-zero frequency. In Example 16.3, the concept of complementary great circles, meaning great circles lying in the planes $wx$ and $yz$, plays the essential role. We construct a relative equilibrium in the curved 6-body problem in ${\mathbb S}_\kappa^3$ for which 3 bodies of equal masses are fixed at the vertices of an equilateral triangle on a great circle of a great sphere, while the other 3 bodies, of the same masses as the previous 3, rotate at the vertices of an equilateral triangle along a complementary great circle of another great sphere. We show that these orbits cannot be contained on any 2-dimensional sphere. Example 16.4 generalizes the previous construction when the triangles are acute and scalene and the masses are not necessarily equal.

Section 17 provides examples of relative equilibria with double elliptic rotations in 
${\mathbb S}_\kappa^3$. In Example 17.1 an equilateral triangle with equal masses at its vertices has 2 elliptic rotations generated from a fixed-point configuration in the curved 3-body problem. Example 17.2 is in the curved 4-body problem: a regular tetrahedron with equal masses at its vertices has 2 elliptic rotations of equal frequencies. In Example 17.3 we construct a solution of the curved 5-body problem in which 5 equal masses are at the vertices of a regular pentatope that has two elliptic rotations of equal-size frequencies. As in the previous example, the orbit is generated from a fixed point, the frequencies have equal size, and the motions cannot take place on any 2-dimensional sphere. Example 17.4 is in the curved 6-body problem, for which 3 bodies of equal masses are at the vertices of an equilateral triangle that rotates on a great circle of a great sphere, while the other 3 bodies, of the same masses as the others, are at the vertices of another equilateral triangle that moves along a complementary great circle of another great sphere. The frequencies are distinct, in general, so, except for a negligible set of solutions that are periodic, the orbits are quasiperiodic. This seems to be the first class of solutions in $n$-body problems given by various potentials for which relative equilibria are not periodic orbits. Example 17.5 generalizes the previous construction to acute scalene triangles and non-equal masses. Sections 18, 19, and 20 provide, each, an example of a $\kappa$-negative elliptic relative equilibrium, a $\kappa$-negative hyperbolic relative equilibrium, and a $\kappa$-negative elliptic-hyperbolic relative equilibrium, respectively, in ${\mathbb H}_\kappa^3$. These 3 examples are in the curved 3-body problem. So Part 5 proves that all the criteria we proved in Part 4 can produce relative equilibria, which behave qualitatively exactly as Theorems 1, 2, and 3 predict.

Finally, Part 6 concludes the paper with a short discussion about the current research of solutions of the curved $n$-body problem and their stability (Section 21), and about future perspectives for relative equilibria in general, in 22.1, and their stability, in 22.2.

\newpage

\part{\rm BACKGROUND AND EQUATIONS OF MOTION}

The goal of this part of the paper is to lay the mathematical background for future developments and results and to obtain the equations of motion of the curved $n$-body problem together with their first integrals. We will also identify the singularities of the equations of motion in order to avoid impossible configurations for the relative equilibria we are going to construct in Part 5. 

\section{Preliminary developments}

In this section we introduce some concepts that will be needed for the derivation of the equations of motion of the $n$-body problem in spaces of constant curvature as well as for the study of the relative equilibria we aim to investigate in this paper.

The standard models of 2-dimensional hyperbolic geometry are the Poincar\'e disk and the Poincar\'e upper-half plane, which are conformal, i.e.\ maintain the angles existing in the hyperbolic plane, as well as the Klein-Beltrami disk, which is not conformal. But none of these models will be used in this paper. In the first part of the section, we introduce the less known Weierstrass model, which physicists usually call the Lorentz model, and we extend it to the 3-dimensional case. This model is more natural than the ones previously mentioned in the sense that it is dual to the sphere, and will thus be essential in our endeavours to develop a unified $n$-body problem in spaces of constant positive and negative curvature. We then provide a short history of Weierstrass's model, to justify its name, and introduce some geometry concepts that will be useful later. In the last part of this section, we define the metric that will be used in the model and unify circular and hyperbolic trigonometry, such that we can introduce a single potential for both the positive and the negative curvature case.

\subsection{The Weierstrass model of hyperbolic geometry}\label{Weierstrass}

Since the Weierstrass model of hyperbolic geometry is not widely known among nonlinear analysts or experts in differential equations and dynamical systems, we present it briefly here. We first discuss the 2-dimensional case, which we will then extend to 3 dimensions. In its 2-dimensional form, this model appeals for at least two reasons: it allows an obvious comparison with the 2-dimensional sphere, both from the geometric and from the algebraic point of view; it emphasizes  the difference between the hyperbolic (Bolyai-Lobachevsky) plane and the Euclidean plane as clearly as the well-known difference between the Euclidean plane and the sphere. From the dynamical point of view, the equations of motion of the curved $n$-body problem in ${\mathbb S}_\kappa^3$ resemble the equations of motion in ${\mathbb H}_\kappa^3$, with just a few changes of sign, as we will see later. The dynamical consequences will be, nevertheless, significant, but we will still be able to use the resemblances between the sphere and the hyperboloid in order to understand the dynamics of the problem.

The 2-dimensional Weierstrass model is built on one of the sheets of the hyperboloid of two sheets,
$$
x^2+y^2-z^2=\kappa^{-1},
$$
where $\kappa< 0$ represents the curvature of the surface in the 3-dimensional Minkowski space 
$\mathbb{\mathbb R}^{2,1}:=(\mathbb{R}^3, \boxdot)$, in which  
$$
{\bf a}\boxdot{\bf b}:=a_xb_x+a_yb_y-a_zb_z,
$$
with ${\bf a}=(a_x,a_y,a_z)$ and ${\bf b}=(b_x,b_y,b_z)$, defines the Lorentz inner product.
We choose for our model the $z>0$ sheet of the hyperboloid of two sheets, which we identify with the hyperbolic plane ${\mathbb H}_\kappa^2$. We can think of this surface as being a pseudosphere of imaginary radius $iR$, a case in which the relationship between radius and curvature is given by $(iR)^2=\kappa^{-1}$.

A linear transformation $T\colon{\mathbb R}^{2,1}\to{\mathbb R}^{2,1}$ is called orthogonal if 
$$
T({\bf a})\boxdot T({\bf a})={\bf a}\boxdot{\bf a}
$$ 
for any ${\bf a}\in{\mathbb R}^{2,1}$. The set of these transformations, together with the Lorentz inner product, forms the orthogonal group $O({\mathbb R}^{2,1})$, given by matrices of determinant $\pm1$.
Therefore the group $SO({\mathbb R}^{2,1})$ of orthogonal transformations of determinant 1 is a subgroup of $O({\mathbb R}^{2,1})$. Another subgroup
of $O({\mathbb R}^{2,1})$ is $G({\mathbb R}^{2,1})$, which is formed by the transformations $T$ that leave ${\mathbb H}_\kappa^2$ invariant. Furthermore, $G({\mathbb R}^{2,1})$ has the closed Lorentz subgroup, ${\rm Lor}({\mathbb R}^{2,1}):= G({\mathbb R}^{2,1}) \cap SO({\mathbb R}^{2,1})$. 
 
An important result, with consequences in our paper, is the Principal Axis Theorem for  ${\rm Lor}({\mathbb R}^{2,1})$, \cite{Dillen}, \cite{Nomizu}. Let us define the Lorentzian rotations about an axis as  the 1-parameter subgroups of  ${\rm Lor}({\mathbb R}^{2,1})$ that leave the axis pointwise fixed. Then the Principal Axis Theorem states that every Lorentzian transformation has one of the forms: 
$$
A=P\begin{bmatrix}
\cos\theta & -\sin\theta & 0 \\ 
\sin\theta & \cos\theta & 0 \\ 
0 & 0 & 1 
\end{bmatrix} P^{-1},
$$
$$B=P\begin{bmatrix}
1 & 0 & 0 \\ 
0 & \cosh s & \sinh s \\ 
0 & \sinh s & \cosh s 
\end{bmatrix}P^{-1}, 
$$  
or 
$$C=P\begin{bmatrix}
1 & -t & t \\ 
t & 1-t^2/2 & t^2/2 \\ 
t & -t^2/2 & 1+t^2/2 
\end{bmatrix}P^{-1}, 
$$  
where $\theta\in[0,2\pi)$, $s, t\in\mathbb{R}$, and  $P\in {\rm Lor}({\mathbb R}^{2,1})$. 
These transformations are called elliptic, hyperbolic, and parabolic, respectively. The elliptic transformations are rotations about a {\it timelike} axis---the $z$ axis in our case---and act along a circle, like in the spherical case; the hyperbolic rotations are about a {\it spacelike} axis---the $x$ axis in our context---and act along a hyperbola; and the parabolic transformations are rotations about a {\it lightlike} (or null) axis---represented here by the line  $x=0$, $y=z$---and act along a parabola. This result is the analogue of Euler's Principal Axis Theorem for the sphere, which states that any element of $SO(3)$ can be written, in some orthonormal basis, as a rotation about the $z$ axis.

The geodesics of ${\mathbb H}_\kappa^2$ are the hyperbolas obtained by intersecting the hyperboloid with planes passing through the origin of the coordinate system. For any two distinct points ${\bf a}$ and ${\bf b}$ of ${\mathbb H}_\kappa^2$, there is a unique geodesic that connects them, and the distance between these points is given by 
\begin{equation}
d({\bf a},{\bf b})=(-\kappa)^{-1/2}\cosh^{-1}(\kappa{\bf a}\boxdot{\bf b}).
\label{hyp-distance}
\end{equation}

In the framework of Weierstrass's model, the parallels' postulate of hyperbolic geometry can be translated as follows. Take a geodesic $\gamma$, i.e.~a hyperbola obtained by intersecting a plane through the origin, $O$, of the coordinate system with the upper sheet, $z>0$, of the hyperboloid. This hyperbola has two asymptotes in its plane: the straight lines $a$ and $b$, intersecting at $O$. Take a point, $P$, on the upper sheet of the hyperboloid but not on the chosen hyperbola. The plane $aP$ produces the geodesic hyperbola $\alpha$, whereas $bP$ produces $\beta$. These two hyperbolas intersect at $P$. Then $\alpha$ and $\gamma$ are parallel geodesics meeting at infinity along $a$, while $\beta$ and $\gamma$ are parallel geodesics meeting at infinity along $b$. All the hyperbolas between $\alpha$ and $\beta$ (also obtained from planes through $O$) are non-secant with $\gamma$.

Like the Euclidean plane, the abstract Bolyai-Lobachevsky plane has no privileged points or geodesics. But the Weierstrass model has some convenient points and geodesics, such as the point $(0,0,|\kappa|^{-1/2})$, namely the vertex of the sheet $z>0$ of the hyperboloid, and the geodesics passing through it. The elements of ${\rm Lor}({\mathbb R}^{2,1})$ allow us to move the geodesics of ${\mathbb H}_\kappa^2$ to convenient positions, a property that can be used to simplify certain arguments. 

More detailed introductions to the 2-dimensional Weierstrass model can be found in \cite{Fab} and \cite{Rey}. The Lorentz group is treated in some detail in \cite{Bak} and \cite{Rey}, but the Principal Axis Theorems for the Lorentz group contained in \cite{Bak} fails to include parabolic rotations, and is therefore incomplete.

The generalization of the 2-dimensional Weierstrass model to 3 dimensions is straightforward. We consider first the 4-dimensional Min\-kowski space ${\mathbb R}^{3,1}=({\mathbb R}^4,\boxdot)$, where $\boxdot$ is now defined as the Lorentz inner product
$$
{\bf a}\boxdot{\bf b}=a_wb_w+a_xb_x+a_yb_y-a_zb_z,
$$
with ${\bf a}=(a_w,a_x,a_y,a_z)$ and ${\bf b}=(b_w,b_x,b_y,b_z)$. In the Minkowski space we embed the $z>0$ connected component of the 3-dimensional hyperbolic manifold given by the equation
\begin{equation}
w^2+x^2+y^2-z^2=\kappa^{-1},
\label{hyp-manifold}
\end{equation}
which models the 3-dimensional hyperbolic space ${\mathbb H}_\kappa^3$ of constant curvature $\kappa<0$. The distance is given by the same formula \eqref{hyp-distance}, where $\bf a$ and $\bf b$ are now points in $\mathbb R^4$ that lie in the 3-dimensional hyperbolic manifold \eqref{hyp-manifold} with $z>0$.

The next issue to discuss would be that of Lorentzian transformations in ${\mathbb H}_\kappa^3$. But we postpone this subject, to present it in Section \ref{isometries} together with the isometries of the 3-dimensional sphere.

\subsection{History of the Weirstrass model}

The first mathematician who mentioned Karl Weierstrass in connection with the hyperboloid model of the hyperbolic plane was Wilhelm Killing. In a paper published in 1880, \cite{Kil1}, he used what he called Weierstrass's coordinates to describe the ``exterior hyperbolic plane'' as an ``ideal region'' of the Bolyai-Lobachevsky plane. In 1885, he added that Weierstrass had introduced these coordinates, in combination with ``numerous applications,''  during a seminar held in 1872,  \cite{Kil2}, pp.~258-259. We found no evidence of any written account of the hyperboloid model for the Bolyai-Lobachevsky plane prior to the one Killing gave in a paragraph of \cite{Kil2}, p.~260. His remarks might have inspired Richard Faber to name this model after Weierstrass and to dedicate a chapter to it in \cite{Fab}, pp.~247-278.

\subsection{More geometric background}

Since we are interested in the motion of
point particles in 3-dimensional manifolds, the natural framework for the study of the 3-dimensional curved $n$-body problem is the Euclidean ambient space, $\mathbb R^4$, endowed with a specific inner-product structure, which depends on whether the curvature is positive or negative. We therefore aim to continue to set here the problem's geometric background in $\mathbb R^4$.
For positive constant curvature, $\kappa>0$, the motion takes place in a 3-dimensional sphere embedded in the Euclidean space $\mathbb{R}^4$ endowed with the standard dot product, $\cdot$\ , i.e.\ on the manifold
$$
{\mathbb S}_\kappa^3=\{(w,x,y,z) | w^2+x^2+y^2+z^2=\kappa^{-1}\}.
$$
For negative constant curvature, $\kappa<0$, the motion takes place on the manifold introduced in the previous subsection, the upper connected component of a 3-dimensional hyperboloid of two connected components embedded in the Minkowski space ${\mathbb{R}}^{3,1}$, i.e.\ on the manifold
$$
{\mathbb H}_\kappa^3=\{(w,x,y,z) | w^2+x^2+y^2-z^2=\kappa^{-1}, z>0\},
$$
where ${\mathbb{R}}^{3,1}$ is $\mathbb R^4$ endowed with the Lorentz inner product,
$\boxdot$. Generically, we will denote these manifolds by
$$
{\mathbb M}^3_\kappa=\{(w,x,y,z)\in\mathbb{R}^4\ |\ w^2+x^2+y^2+\sigma z^2=\kappa^{-1}, \ {\rm with}\ z>0 \ {\rm for}\ \kappa<0\},
$$
where $\sigma$ is the signum function,
\begin{equation}
\sigma=
\begin{cases}
+1, \ \ {\rm for} \ \ \kappa>0\cr
-1, \ \ {\rm for} \ \ \kappa<0.\cr
\end{cases}\label{sigma}
\end{equation}
 Given the 4-dimensio\-nal vectors 
$${\bf a}=(a_w, a_x,a_y,a_z)\ \ {\rm and}\ \ {\bf b}=(b_w, b_x,b_y,b_z),$$ 
we define their inner product as
\begin{equation}
{\bf a}\odot{\bf b}:=a_wb_w+a_xb_x+a_yb_y+\sigma a_zb_z,
\label{dotpr}
\end{equation}
so ${\mathbb M}^3_\kappa$ is endowed with the operation $\odot$, meaning $\cdot$ for $\kappa>0$ and $\boxdot$ for $\kappa<0$.

If $R$ is the radius of the sphere ${\mathbb S}_\kappa^3$, then the relationship between $\kappa>0$ and $R$ is $\kappa^{-1}=R^2$. As we already mentioned, to have an analogue interpretation in the case of negative curvature, ${\mathbb H}_\kappa^3$ can be viewed as a 3-dimensional pseudosphere of imaginary radius $iR$, such that the relationship between $\kappa<0$ and $iR$ is $\kappa^{-1}=(iR)^2$.

Let us further define some concepts that will be useful later.

\begin{definition}
A great sphere of ${\mathbb S}_\kappa^3$ is a 2-dimensional sphere of the same
radius as ${\mathbb S}_\kappa^3$.
\label{greatsphere}
\end{definition}

Definition \ref{greatsphere} implies that the curvature of a great sphere is the same as the curvature of ${\mathbb S}_\kappa^3$. Great spheres of ${\mathbb S}_\kappa^3$ are obtained by intersecting ${\mathbb S}_\kappa^3$ with hyperplanes of ${\mathbb R}^4$ that pass through the origin of the coordinate
system. Examples of great spheres are:
\begin{equation}
{\bf S}_{\kappa,w}^2=\{(w,x,y,z)|\ \! x^2+y^2+z^2=\kappa^{-1},\ w=0\},
\label{s2w}
\end{equation}
\begin{equation}
{\bf S}_{\kappa,x}^2=\{(w,x,y,z)|\ \! w^2+y^2+z^2=\kappa^{-1},\ x=0\},
\end{equation}
\begin{equation}
{\bf S}_{\kappa,y}^2=\{(w,x,y,z)|\ \! w^2+x^2+z^2=\kappa^{-1},\ y=0\},
\end{equation}
\begin{equation}
{\bf S}_{\kappa,z}^2=\{(w,x,y,z)|\ \! w^2+x^2+y^2=\kappa^{-1},\ z=0\}.
\label{z-sphere}
\end{equation}

\begin{definition}
A great circle of a great sphere of ${\mathbb S}_\kappa^3$ is a circle (1-dimensional sphere) of the same radius as ${\mathbb S}_\kappa^3$.
\label{greatcircle}
\end{definition}

Definition \ref{greatcircle} implies that, since the curvature of ${\mathbb S}_\kappa^3$ is
$\kappa$, the curvature of a great circle is $\kappa=1/R$, where $R$ is
the radius of ${\mathbb S}_\kappa^3$ and of the great circle. Examples of great circles are:
\begin{equation}
{\bf S}_{\kappa,wx}^1=\{(w,x,y,z)|\ \! y^2+z^2=\kappa^{-1},\ w=x=0\},
\end{equation}
\begin{equation}
{\bf S}_{\kappa,yz}^1=\{(w,x,y,z)|\ \! w^2+x^2=\kappa^{-1},\ y=z=0\},
\end{equation}
\begin{equation}
{\bf S}_{\kappa,wy}^1=\{(w,x,y,z)|\ \! y^2+z^2=\kappa^{-1},\ w=y=0\},
\end{equation}
\begin{equation}
{\bf S}_{\kappa,xz}^1=\{(w,x,y,z)|\ \! w^2+x^2=\kappa^{-1},\ x=z=0\},
\end{equation}
\begin{equation}
{\bf S}_{\kappa,wz}^1=\{(w,x,y,z)|\ \! y^2+z^2=\kappa^{-1},\ w=z=0\},
\end{equation}
\begin{equation}
{\bf S}_{\kappa,xy}^1=\{(w,x,y,z)|\ \! y^2+z^2=\kappa^{-1},\ x=y=0\}.
\end{equation}
Notice that ${\bf S}_{\kappa,wx}^1$ is a great circle for both the great spheres ${\bf S}_{\kappa,w}^2$ and ${\bf S}_{\kappa,x}^2$, whereas ${\bf S}_{\kappa,yz}^1$
is a great circle for both the great spheres ${\bf S}_{\kappa,y}^2$ and ${\bf S}_{\kappa,z}^2$. Similar remarks can be made about any of the above pairs of great circles.

\begin{definition}
Two great circles, $C_1$ and $C_2$,  of two different great spheres of ${\mathbb S}_\kappa^3$ are called complementary if there is a coordinate system $wxyz$ such that
\begin{equation}
C_1={\bf S}_{\kappa,wx}^1\ \ {\rm and}\ \ C_2={\bf S}_{\kappa,yz}^1
\label{1st-rep}
\end{equation}
or
\begin{equation}
C_1={\bf S}_{\kappa,wy}^1\ \ {\rm and}\ \ C_2={\bf S}_{\kappa,xz}^1.
\label{2nd-rep}
\end{equation}
\label{complementary}
\end{definition}
\vspace{-1.2em}

The representations \eqref{1st-rep} and \eqref{2nd-rep} of $C_1$ and $C_2$ give, obviously, all the possibilities we have, because, for instance, the representation 
$$
C_1={\bf S}_{\kappa,wz}^1\ \ {\rm and}\ \ C_2={\bf S}_{\kappa,xy}^1,
$$ 
is the same as \eqref{1st-rep} after we perform a circular permutation of the coordinates $w,x,y,z$.
For simplicity, and without loss of generality, we will always use representation \eqref{1st-rep}.

The pair $C_1$ and $C_2$ of complementary circles of ${\mathbb S}^3$, for instance, forms, in topological terms, a Hopf link in a Hopf fibration, which is the map
$$
{\mathcal H}\colon{\mathbb S}^3\to{\mathbb S}^2, \ h(w,x,y,z)=(w^2+x^2-y^2-z^2,2(wz+xy),2(xz-wy))
$$
that takes circles of $\mathbb S^3$ to points of $\mathbb S^2$,
\cite{Hopf}, \cite{Lyons}. In particular, $\mathcal H$ takes ${\bf S}_{1,wx}^1$ to $(1,0,0)$ and ${\bf S}_{1,yz}^1$ to $(-1,0,0)$. Using the stereographic projection, it can be shown that the circles $C_1$ and $C_2$ are linked (like any adjacent rings in a chain), hence the name of the pair, \cite{Lyons}. Of course, these properties can be expressed in terms of any curvature $\kappa>0$. Hopf fibrations have important physical applications in fields such as rigid body mechanics, \cite{tudor}, quantum information theory, \cite{Mosseri}, and magnetic monopoles, \cite{Nakahara}. As we will see later, they are also useful in celestial mechanics via the curved $n$-body problem.

We will show in the next section that the distance between two points lying on complementary great circles is independent of their position. This remarkable geometric property turns out to be even more surprising from the dynamical point of view. Indeed, given the fact that the distance between 2 complementary great circles is constant, the magnitude of the gravitational interaction (but not the direction of the force) between a body lying on a great circle and a body lying on the complementary great circle is the same, no matter where the bodies are on their respective circles. This simple observation will help us construct some interesting, nonintuitive classes of solutions of the curved $n$-body problem.

In analogy with great spheres of ${\mathbb S}_\kappa^3$, we define great hyperboloids of ${\mathbb H}_\kappa^3$ as follows.

\begin{definition}
A great hyperboloid of ${\mathbb H}_\kappa^3$ is a 2-dimensional hyperboloid of the same
curvature as ${\mathbb H}_\kappa^3$.
\label{greathyperboloid}
\end{definition}

Great hyperboloids of ${\mathbb H}_\kappa^3$ are obtained by intersecting ${\mathbb H}_\kappa^3$ with hyperplanes of ${\mathbb R}^4$ that pass through the origin of the coordinate system. Examples of great hyperboloids are:
\begin{equation}
{\bf H}_{\kappa,w}^2=\{(w,x,y,z)|\ \! x^2+y^2-z^2=\kappa^{-1},\ w=0\},
\label{h2w}
\end{equation}
\begin{equation}
{\bf H}_{\kappa,x}^2=\{(w,x,y,z)|\ \! w^2+y^2-z^2=\kappa^{-1},\ x=0\},
\end{equation}
\begin{equation}
{\bf H}_{\kappa,y}^2=\{(w,x,y,z)|\ \! w^2+x^2-z^2=\kappa^{-1},\ y=0\}.
\label{great-Y}
\end{equation}

It is interesting to remark that great spheres and great hyperboloids are totally geodesic surfaces in ${\mathbb S}_\kappa^3$ and ${\mathbb H}_\kappa^3$, respectively. Recall that if through a given point of a Riemannian manifold (such as ${\mathbb S}_\kappa^3$ and ${\mathbb H}_\kappa^3$) we take the various geodesics of that manifold tangent at this point to the same plane element, we obtain a geodesic surface. A surface that is geodesic at each of its points is called totally geodesic.

\subsection{The metric}

An important preparatory issue lies with introducing the metric used on the manifolds ${\mathbb S}_\kappa^3$ and ${\mathbb H}_\kappa^3$, which, given the definition of the inner products, we naturally take as 
\begin{equation}
d_{\kappa}({\bf a},{\bf b}):=
\begin{cases}
\kappa^{-1/2}\cos^{-1}(\kappa{\bf a}\cdot{\bf b}),\ \ \ \ \ \ \ \ \  \kappa >0\cr
|{\bf a}-{\bf b}|, \ \ \ \ \ \ \ \ \ \ \ \ \ \ \ \ \ \ \ \ \ \ \  \! \ \hspace{2 pt} \kappa=0\cr
({-\kappa})^{-1/2}\cosh^{-1}(\kappa{\bf a}\boxdot{\bf b}),\hspace{4 pt} \kappa<0,\cr
\end{cases}
\label{distance}
\end{equation}
where the vertical bars denote the standard Euclidean norm.

When $\kappa\to 0$, with either $\kappa>0$ or $\kappa<0$, then $R\to\infty$, where $R$ represents the radius of the sphere ${\mathbb S}_\kappa^3$ or the real factor in the expression $iR$ of the imaginary radius of the pseudosphere ${\mathbb H}_\kappa^3$. As $R\to\infty$, both ${\mathbb S}_\kappa^3$ and ${\mathbb H}_\kappa^3$ become $\mathbb R^3$, and the vectors $\bf a$ and $\bf b$ become parallel, so the distance between them is given by the Euclidean distance, as indicated in \eqref{distance}. Therefore, in a way, $d$ is a continuous function of $\kappa$ when the manifolds ${\mathbb S}_\kappa^3$ and ${\mathbb H}_\kappa^3$ are pushed to infinity.

To get more insight into the fact that the metric in ${\mathbb S}_\kappa^3$ and ${\mathbb H}_\kappa^3$ becomes the Euclidean metric in $\mathbb R^3$ when $\kappa\to 0$, let us use the stereographic projection, which takes the points of coordinates $(w,x,y,z)\in{\mathbb M}_\kappa^3$, where
$$
{\mathbb M}^3_\kappa=\{(w,x,y,z)\in\mathbb{R}^4\ |\ w^2+x^2+y^2+\sigma z^2=\kappa^{-1}, \ {\rm with}\ z>0 \ {\rm for}\ \kappa<0\},
$$
to the points of coordinates $(W,X,Y)$ of the 3-dimensional hyperplane $z=0$ through the bijective transformation 
\begin{equation}
W=\frac{Rw}{R-\sigma z},\ \ X=\frac{Rx}{R-\sigma z},\ \ Y=\frac{Ry}{R-\sigma z},
\label{Gprojection}
\end{equation}
which has the inverse
$$
w=\frac{2R^2W}{R^2+\sigma W^2+\sigma X^2+\sigma Y^2}, \ \
x=\frac{2R^2X}{R^2+\sigma W^2+\sigma X^2+\sigma Y^2},
$$
$$
 y=\frac{2R^2Y}{R^2+\sigma W^2+\sigma X^2+\sigma Y^2},\ \ z=\frac{R(W^2+X^2+Y^2-\sigma R^2)}{R^2+\sigma W^2+\sigma X^2+\sigma Y^2}.
$$
From the geometric point of view, the correspondence between a point of ${\mathbb M}_\kappa^3$ and
a point of the hyperplane $z=0$ is made via a straight line through the point $(0,0,0,\sigma R)$, called the north pole, for both $\kappa>0$ and $\kappa<0$. 

For $\kappa>0$, the projection is the Euclidean space $\mathbb R^3$, whereas for $\kappa<0$ it is the 3-dimensional solid Poincar\'e ball of radius $\kappa^{-1/2}$. 
The metric in coordinates $(W,X,Y)$ is given by 
$$
ds^2=\frac{4R^4(dW^2+dX^2+dY^2)}{(R^2+\sigma W^2+\sigma X^2+\sigma Y^2)^2},
$$
which can be obtained by substituting the inverse of the stereographic projection into the metric
$$
ds^2=dw^2+dx^2+dy^2+\sigma dz^2.
$$
The stereographic projection is conformal, but it's neither isometric nor area preserving. Therefore we cannot expect to recover the exact Euclidean metric when $\kappa\to 0$, i.e.\ when $R\to\infty$, but hope, nevertheless, to obtain some expression that resembles it. Indeed, we can divide the numerator and denominator of the right hand side of the above metric by $R^4$ and write it after simplification as
$$
ds^2=\frac{4(dW^2+dX^2+dY^2)}{(1+\sigma W^2/R^2+\sigma X^2/R^2+\sigma Y^2/R^2)^2}.
$$
When $R\to\infty$, we have
$$
ds^2=4(dW^2+dX^2+dY^2),
$$
which is the Euclidean metric of $\mathbb R^3$ up to a constant factor.

\begin{remark}
From \eqref{distance} we can conclude that if, for $\kappa>0$, $C_1$ and $C_2$ are two complementary great circles, as described in Definition \ref{complementary}, and ${\bf a}\in C_1, {\bf b}\in C_2$, then the distance between
$\bf a$ and $\bf b$ is
$$
d_\kappa({\bf a},{\bf b})=\kappa^{-1/2}\pi/2.
$$
This fact shows that the distance between two complementary circles is constant.
\end{remark}

Since, to derive the equations of motion, we will apply a variational principle, we need to extend the distance from the 3-dimensional manifolds of constant curvature ${\mathbb S}_\kappa^3$ and ${\mathbb H}_\kappa^3$ to the 4-dimensional ambient space. We therefore redefine the distance between ${\bf a}$ and $\bf b$ as
\begin{equation}
d_{\kappa}({\bf a},{\bf b}):=
\begin{cases}
\kappa^{-1/2}\cos^{-1}{\kappa{\bf a}\cdot{\bf b}\over\sqrt{\kappa{\bf a}\cdot{\bf a}}
\sqrt{\kappa{\bf b}\cdot{\bf b}}},\ \ \ \ \ \ \ \ \! \ \ \kappa >0\cr
|{\bf a}-{\bf b}|, \ \ \ \ \ \ \ \ \ \ \ \ \ \ \ \ \ \ \ \ \ \ \! \ \hspace{0.82cm} \kappa=0\cr
({-\kappa})^{-1/2}\cosh^{-1}{\kappa{\bf a}\boxdot{\bf b}\over\sqrt{\kappa{\bf a}\boxdot{\bf a}}\sqrt{\kappa{\bf b}\boxdot{\bf b}}},\hspace{4.5 pt} \kappa<0.\cr
\end{cases}
\label{extendedmetric}
\end{equation}
Notice that on ${\mathbb S}_\kappa^3$ we have $\sqrt{{\kappa{\bf a}\cdot{\bf a}}}=
\sqrt{\kappa{\bf b}\cdot{\bf b}}=1$ and on ${\mathbb H}_\kappa^3$ we have
$\sqrt{\kappa{\bf a}\boxdot{\bf a}}=\sqrt{\kappa{\bf b}\boxdot{\bf b}}=1$, which means that the new distance reduces to the previously defined distance on the corresponding 3-dimensional manifolds of constant curvature.

\subsection{Unified trigonometry}

Following the work of Cari\~nena, Ra\~nada, and Santander, \cite{Car}, we will further define the trigonometric $\kappa$-functions, which unify circular and hyperbolic trigonometry. The reason for this step is to obtain the equations of motion of the curved $n$-body problem in both constant positive and constant negative curvature spaces. We define the 
$\kappa$-sine, ${\rm sn}_\kappa$, as
$$
{\rm sn}_{\kappa}(x):=\left\{
\begin{array}{rl}
{\kappa}^{-1/2}\sin{\kappa}^{1/2}x & {\rm if }\ \ \kappa>0\\
x & {\rm if }\ \ \kappa=0\\
({-\kappa})^{-{1/2}}\sinh({-\kappa})^{1/2}x & {\rm if }\ \ \kappa<0,
\end{array}  \right.
$$
the $\kappa$-cosine, ${\rm csn}_\kappa$, as
$$
{\rm csn}_{\kappa}(x):=\left\{
\begin{array}{rl}
\cos{\kappa}^{1/2}x & {\rm if }\ \ \kappa>0\\
1 & {\rm if }\ \ \kappa=0\\
\cosh({-\kappa})^{1/2}x & {\rm if }\ \ \kappa<0,
\end{array}  \right.
$$
as well as the $\kappa$-tangent, ${\rm tn}_\kappa$, and $\kappa$-cotangent,
${\rm ctn}_\kappa$, as 
$${\rm tn}_{\kappa}(x):={{\rm sn}_{\kappa}(x)\over {\rm csn}_{\kappa}(x)}\ \ \ {\rm and}\ \ \
{\rm ctn}_{\kappa}(x):={{\rm csn}_{\kappa}(x)\over {\rm sn}_{\kappa}(x)},$$
respectively. The entire trigonometry can be rewritten in this unified context,
but the only identity we will further need is the fundamental formula
\begin{equation}
{\kappa}\ {\rm sn}_{\kappa}^2(x)+{\rm csn}_{\kappa}^2(x)=1.
\label{fundamentalformula}
\end{equation}

Notice that all the above trigonometric $\kappa$-functions are continuous with respect to $\kappa$. In the above formulation of the unified trigonometric $\kappa$-functions, we assigned no particular meaning to the real parameter $\kappa$. In what follows, however, $\kappa$ will represent the constant curvature of a 3-dimensional manifold. 
Therefore, with this notation, the distance \eqref{distance} on the manifold ${\mathbb M}_\kappa^3$ can be written as
$$
d_\kappa({\bf a},{\bf b})=(\sigma\kappa)^{-1/2}{\rm csn}_\kappa^{-1}[(\sigma\kappa)^{1/2}{\bf a}\odot{\bf b}]
$$
for any ${\bf a},{\bf b}\in{\mathbb M}_\kappa^3$ and $\kappa\ne 0$.

\section{Equations of motion}\label{equationsofmotion}

The main purpose of this section is to derive the equations of motion of the curved $n$-body problem on the 3-dimensional manifolds ${\mathbb M}_\kappa^3$. To achieve this goal, we will introduce the curved potential function, which also represents the potential energy of the particle system, present and apply Euler's formula for homogeneous functions to the curved potential function, and introduce the variational method of constrained Lagrangian dynamics. After we derive the equations of motion of the curved $n$-body problem, we will show that they can be put in Hamiltonian form.

\subsection{The potential}

Since the classical Newtonian equations of the $n$-body problem are expressed in terms of a potential function, our next goal is to define a potential function that extends to spaces of constant curvature and reduces to the classical potential function in the Euclidean case, i.e.\ when $\kappa=0$.

Consider the point particles (also called point masses or bodies) of masses $m_1, m_2,\dots, m_n>0$ in $\mathbb{R}^4$, for $\kappa>0$,
and in $\mathbb{R}^{3,1}$, for $\kappa<0$, whose positions are given by the vectors
${\bf q}_i=(w_i,x_i,y_i,z_i),\ i=1,2,\dots,n$. Let ${\bf q}=({\bf q}_1, {\bf q}_2,\dots,{\bf q}_n)$ be the configuration of the system
and ${\bf p}=({\bf p}_1, {\bf p}_2,\dots,{\bf p}_n)$, with ${\bf p}_i=m_i\dot{\bf q}_i, \ i=1,2,\dots,n$, be the momentum of the system.
We define the gradient operator with respect to the vector ${\bf q}_i$ as
$$\nabla_{{\bf q}_i}:=(\partial_{w_i},\partial_{x_i},\partial_{y_i},\sigma\partial_{z_i}).$$

From now on we will rescale the units such that the gravitational constant $G$
is $1$. We thus define the potential of the curved $n$-body problem, which we will call the curved potential, as the function $-U_\kappa$, where
\begin{equation*}
U_\kappa({\bf q}):={1\over 2}\sum_{i=1}^n\sum_{j=1,j\ne i}^n{m_im_j{\rm ctn}_\kappa
({d}_\kappa({\bf q}_i,{\bf q}_j))}
\label{defpot}
\end{equation*}
stands for the curved force function, and ${\bf q}=({\bf q}_1,\dots, {\bf q}_n)$ is the configuration of the system.
Notice that, for $\kappa=0,$ we have ${\rm ctn}_0({d}_0({\bf q}_i,{\bf q}_j))=|{\bf q}_i-{\bf q}_j|^{-1}$, which means that the curved potential becomes the classical Newtonian potential in the Euclidean case. Moreover, $U_\kappa\to U_0$ as $\kappa\to 0$, whether through positive or negative values of $\kappa$. Nevertheless, we cannot claim that $U_\kappa$ is continuous with respect to $\kappa$ at $0$ in the usual way, but rather in a degenerate sense. Indeed, when $\kappa\to 0$, the manifold ${\mathbb M}_\kappa^3$ is pushed to infinity, so the above continuity with respect to $\kappa$ must be understood in this restricted way. In fact, as we will further see, the curved force function, $U_\kappa$, is homogeneous of degree $0$, whereas the Newtonian potential, $U_0$, defined in the Euclidean space, is a homogeneous function of degree $-1$.

Now that we defined a potential that satisfies the basic limit condition we required of any extension of the $n$-body problem beyond the Euclidean space, we emphasize that it also satisfies the basic properties the classical Newtonian potential fulfills in the case of the Kepler problem, as mentioned in the Introduction: it obeys Bertrand's property, according to which every bounded orbit is closed, and is a solution of the Laplace-Beltrami equation, \cite{Kozlov}, the natural generalization of Laplace's equation to Riemannian and pseudo-Riemannian manifolds. These properties ensure that the cotangent potential provides us with a natural extension of Newton's gravitational law to spaces of constant curvature.

Let us now focus on the case $\kappa\ne 0$. A straightforward computation, which uses the fundamental formula \eqref{fundamentalformula}, shows that
\begin{equation}
U_\kappa({\bf q})={1\over 2}\sum_{i=1}^n\sum_{j=1,j\ne i}^n{m_im_j
(\sigma\kappa)^{1/2}{\kappa{\bf q}_i\odot{\bf q}_j\over\sqrt{\kappa{\bf q}_i
\odot{\bf q}_i}\sqrt{\kappa{\bf q}_j\odot{\bf q}_j}}\over
\sqrt{\sigma-\sigma\Big({\kappa{\bf q}_i\odot{\bf q}_j\over
\sqrt{\kappa{\bf q}_i
\odot{\bf q}_i}\sqrt{\kappa{\bf q}_j\odot{\bf q}_j}}\Big)^2}}, \ \ \kappa\ne 0,
\label{pothom}
\end{equation}
an expression that is equivalent to
\begin{equation}
U_\kappa({\bf q})=\sum_{1\le i<j\le n}{m_im_j
|\kappa|^{1/2}{\kappa{\bf q}_i\odot{\bf q}_j}\over
[\sigma(\kappa{\bf q}_i
\odot{\bf q}_i)(\kappa{\bf q}_j\odot{\bf q}_j)-\sigma({\kappa{\bf q}_i\odot{\bf q}_j
})^2]^{1/2}},\ \kappa\ne 0.
\label{forcef}
\end{equation}
In fact, we could simplify $U_\kappa$ even more by recalling that $\kappa{\bf q}_i\odot{\bf q}_i=1,\ i=1,2,\dots, n$. But since we still need to compute $\nabla U_\kappa$, which means differentiating $U_\kappa$, we we will not make that simplification here.


\subsection{Euler's formula for homogeneous functions}

In 1755, Leonhard Euler proved a beautiful formula related to homogeneous functions, \cite{Euler}. We will further present it and show how it applies to the curved potential.

\begin{definition}
A function $F\colon {\mathbb R}^m\to\mathbb R$ is called homogeneous of degree $\alpha\in\mathbb R$ if for all $\eta\ne 0$ and ${\bf q}\in \mathbb R^m$, we have 
$$
F(\eta{\bf q})=\eta^\alpha F({\bf q}).
$$
\end{definition}
Euler's formula shows that, for any homogeneous function of degree $\alpha\in\mathbb R$, we have
$$
{\bf q}\cdot\nabla F({\bf q})=\alpha F({\bf q})
$$
for all ${\bf q}\in{\mathbb R}^m$.

Notice that $U_{\kappa}(\eta{\bf q})=U_\kappa({\bf q})=\eta^0U_{\kappa}({\bf q})$
for any $\eta\ne 0$, which means that the curved potential is a homogeneous function of degree zero. With our notations, we have $m=3n$, therefore Euler's formula
can be written as 
$$
{\bf q}\odot\nabla{F({\bf q})}=\alpha F({\bf q}).
$$
Since $\alpha=0$ for $U_{\kappa}$ with $\kappa\ne 0$, we conclude that
\begin{equation}
{\bf q}\odot\nabla U_{\kappa}({\bf q})=0.
\end{equation}
We can also write the curved force function as 
$U_\kappa({\bf q})={1\over 2}\sum_{i=1}^nU_\kappa^i({\bf q}_i)$, where
$$
U_\kappa^i({\bf q}_i):=
\sum_{j=1,j\ne i}^n{m_im_j(\sigma\kappa)^{1/2}{\kappa{\bf q}_i\odot{\bf q}_j\over\sqrt{\kappa{\bf q}_i\odot{\bf q}_i}\sqrt{\kappa{\bf q}_j\odot{\bf q}_j}}\over\sqrt{\sigma-\sigma\Big({\kappa{\bf q}_i\odot{\bf q}_j\over\sqrt{\kappa{\bf q}_i\odot{\bf q}_i}\sqrt{\kappa{\bf q}_j\odot{\bf q}_j}}\Big)^2}}, \ \ i=1,\dots,n,
$$ 
are also homogeneous functions of degree $0$. Applying Euler's formula for functions $F:\mathbb{R}^3\to\mathbb{R}$, we obtain that ${\bf q}_i\odot\nabla_{{\bf q}_i} U_\kappa^i({\bf q})=0$. Then using the identity $\nabla_{{\bf q}_i}U_\kappa({\bf q})=\nabla_{{\bf q}_i} U_\kappa^i({\bf q}_i)$, we can conclude that 
\begin{equation}
{\bf q}_i\odot\nabla_{{\bf q}_i} U_\kappa({\bf q})=0, \ \ i=1,\dots,n.
\label{eul}
\end{equation}


\subsection{Constrained Lagrangian dynamics}

To obtain the equations of motion of the curved $n$-body problem, we will use the classical variational theory of constrained Lagrangian dynamics, \cite{Gel}.
According to this theory, let 
$$L=T-V$$
be the Lagrangian of a system of $n$ particles constrained to move on a manifold, where $T$ is the kinetic energy and $V$  is the potential energy of the system. If the positions and the velocities of the particles are given by the vectors ${\bf q}_i, \dot{\bf q}_i,\ i=1,2,\dots,n$, and the constraints are characterized by the equations $f_i=0, \ i=1,2,\dots, n$, respectively, then the motion is described by the Euler-Lagrange equations with constraints,
\begin{equation}
{d\over dt}\Bigg({\partial L\over\partial\dot{\bf q}_i}\Bigg)-{\partial L\over\partial{\bf q}_i}-\lambda^i(t){\partial f_i\over\partial{\bf q}_i}={\bf 0},\ \ \ i=1,\dots,n,\label{EL}
\end{equation}
where $\lambda^i,\ i=1,2,\dots, n$, are the Lagrange multipliers. To obtain this formula we have to assume that the distance is defined in the entire ambient space. Using this classical result, we can now derive the equations of motion of the curved $n$-body problem.


\subsection{Derivation of the equations of motion}

In our case, the potential energy is $V=-U_\kappa$,
given by the curved force function (\ref{pothom}), and
we define the kinetic energy of the system of particles as
$$
T_\kappa({\bf q},\dot{\bf q}):={1\over 2}\sum_{i=1}^nm_i(\dot{\bf q}_i\odot\dot{\bf q}_i)(\kappa{\bf q}_i\odot{\bf q}_i).
$$
The reason for introducing the factors $\kappa{\bf q}_i\odot{\bf q}_i=1,\ i=1,2,\dots, n$, into the definition of the kinetic energy  will become clear in Subsection \ref{Hamil}. Then the Lagrangian of the curved $n$-body system has the form
$$
L_\kappa({\bf q}, \dot{\bf q})=T_\kappa({\bf q}, \dot{\bf q})+U_\kappa({\bf q}).
$$
So, according to the theory of constrained Lagrangian dynamics discussed above, which requires the use of a distance defined in the ambient space---a condition we satisfied when we produced formula \eqref{extendedmetric}, the equations of motion are
\begin{equation}
{d\over dt}\Bigg({\partial L_\kappa\over\partial\dot{\bf q}_i}\Bigg)-{\partial L_\kappa\over\partial{\bf q}_i}-\lambda_\kappa^i(t){\partial f_\kappa^i\over\partial{\bf q}_i}={\bf 0},\ \ \ i=1,\dots,n,\label{eqLagrangianS2}
\end{equation}
where $f_\kappa^i={\bf q}_i\odot{\bf q}_i-{\kappa}^{-1}$ is the function that characterizes the constraints $f_\kappa^i=0,\ i=1,2,\dots, n$. Each constraint keeps the body of mass $m_i$ on the surface of constant curvature $\kappa$, and $\lambda_\kappa^i$ is the Lagrange multiplier corresponding to the same body. Since ${\bf q}_i\odot{\bf q}_i=\kappa^{-1}$ implies that 
$\dot{\bf q}_i\odot{\bf q}_i=0$, it follows that
$$
{d\over dt}\Bigg({\partial L_\kappa\over\partial\dot{\bf q}_i}\Bigg)=
m_i\ddot{\bf q}_i(\kappa{\bf q}_i\odot{\bf q}_i)+2 m_i(\kappa\dot{\bf q}_i\odot{\bf q}_i)=
m_i\ddot{\bf q}_i.
$$
This relationship, together with
$$  {\partial L_\kappa\over\partial{\bf q}_i}=m_i\kappa(\dot{\bf q}_i\odot\dot{\bf q}_i){\bf q}_i+\nabla_{{\bf q}_i} U_{\kappa}({\bf q}),
$$
implies that equations (\ref{eqLagrangianS2}) are equivalent to
\begin{equation}
m_i\ddot{\bf q}_i-m_i\kappa(\dot{\bf q}_i\odot\dot{\bf q}_i){\bf q}_i-\nabla_{{\bf q}_i} U_{\kappa}({\bf q})-2\lambda_\kappa^i(t){\bf q}_i={\bf 0},\ \ \ i=1,\dots, n.\label{equations}
\end{equation}
To determine $\lambda_\kappa^i$, notice that
$0=\ddot{f}_\kappa^i=2\dot{\bf q}_i\odot\dot{\bf q}_i+
2({\bf q}_i\odot\ddot{\bf q}_i),$ so
\begin{equation}
{\bf q}_i\odot\ddot{\bf q}_i=-\dot{\bf q}_i\odot\dot{\bf q}_i.\label{h1}
\end{equation}
Let us also remark that $\odot$-multiplying equations (\ref{equations}) by 
${\bf q}_i$ and using Euler's formula (\ref{eul}), we obtain that
$$
m_i({\bf q}_i\odot\ddot{\bf q}_i)-m_i(\dot{\bf q}_i\odot\dot{\bf q}_i)-{\bf q}_i\odot\nabla_{{\bf q}_i} U_{\kappa}({\bf q})=2\lambda_\kappa^i{\bf q}_i\odot{\bf q}_i=2\kappa^{-1}\lambda_\kappa^i,
$$
which, via (\ref{h1}), implies that $\lambda_\kappa^i=-\kappa m_i(\dot{\bf q}_i\odot\dot{\bf q}_i)$.
Substituting these values of the Lagrange multipliers into equations (\ref{equations}),
the equations of motion and their constraints become
\begin{multline}
m_i \ddot {\bf q}_i=\nabla_{{\bf q}_i} U_{\kappa}({\bf q})-m_i\kappa(\dot{\bf q}_i\odot\dot{\bf q}_i){\bf q}_i, \ \ {\bf q}_i\odot{\bf q}_i=\kappa^{-1}, \ \ \kappa\ne 0,\\
 \ \ i=1,\dots, n.\label{eqmotion}
\end{multline}
The ${\bf q}_i$-gradient of the curved force function, obtained from (\ref{pothom}), has the
form
\begin{equation}
\nabla_{{\bf q}_i} U_\kappa({\bf q})=\sum_{j=1,j\ne i}^n
{{m_im_j(\sigma\kappa)^{1/2}\left(\sigma\kappa{\bf q}_j  -\sigma{\kappa^2{\bf q}_i\odot{\bf q}_j\over\kappa{\bf q}_i\odot{\bf q}_i}{\bf q}_i \right)\over\sqrt{\kappa{\bf q}_i\odot{\bf q}_i}\sqrt{\kappa{\bf q}_j\odot{\bf q}_j}}
\over
\left[\sigma-\sigma\left({\kappa{\bf q}_i\odot{\bf q}_j}\over{\sqrt{\kappa{\bf q}_i\odot{\bf q}_i}\sqrt{\kappa{\bf q}_j\odot{\bf q}_j}}\right)^2\right]^{3/2}},\ \kappa\ne 0,
\label{gra}
\end{equation}
which is equivalent to 
\begin{equation}
{\nabla}_{{\bf q}_i}U_\kappa({\bf q})=\sum_{\substack{j=1\\ j\ne i}}^n{m_im_j|\kappa|^{3/2}(\kappa{\bf q}_j\odot{\bf q}_j)[(\kappa{\bf q}_i\odot{\bf q}_i){\bf q}_j-(\kappa{\bf q}_i\odot{\bf q}_j){\bf q}_i]\over
[\sigma(\kappa{\bf q}_i
\odot{\bf q}_i)(\kappa{\bf q}_j\odot{\bf q}_j)-\sigma({\kappa{\bf q}_i\odot{\bf q}_j
})^2]^{3/2}}.
\label{gr}
\end{equation}
Using the fact that $\kappa{\bf q}_i\odot{\bf q}_i=1$, we can write this gradient as
\begin{equation}
\nabla_{{\bf q}_i}U_\kappa({\bf q})=\sum_{j=1,j\ne i}^n{{m_im_j}|
\kappa|^{3/2}
\left[{\bf q}_j  - (\kappa{\bf q}_i\odot{\bf q}_j){\bf q}_i \right]
\over
\left[\sigma-\sigma\left(\kappa{\bf q}_i\odot{\bf q}_j\right)^2\right]^{3/2}}, \ \kappa\ne 0.
\label{grad}
\end{equation}

Sometimes we can use the simpler form \eqref{grad} of the gradient, but whenever
we need to exploit the homogeneity of the gradient or have to differentiate it, we must
use its original form \eqref{gr}. 
Thus equations (\ref{eqmotion}) and (\ref{gr}) describe
the $n$-body problem on surfaces of constant curvature for $\kappa\ne0$. 
Though more complicated than the equations of motion Newton derived for 
the Euclidean space, system (\ref{eqmotion}) is simple enough to allow an
analytic approach. 


\subsection{Hamiltonian formulation}\label{Hamil}

It is always desirable to place any new problem into a more general theory. In our case, like in the classical $n$-body problem, the theory of Hamiltonian systems turns out to be the suitable framework.

The Hamiltonian function describing the motion of the $n$-body problem in spaces of constant curvature is
$$
H_\kappa({\bf q},{\bf p})=T_\kappa({\bf q},{\bf p})-U_\kappa({\bf q}).
$$ 
Then the Hamiltonian form of the equations of motion is given by the system with constraints
\begin{equation}
\begin{cases}
\dot{\bf q}_i=
\nabla_{{\bf p}_i} H_\kappa({\bf q},{\bf p})=m_i^{-1}{\bf p}_i,\cr
\dot{\bf p}_i=
-\nabla_{{\bf q}_i} H_\kappa({\bf q},{\bf p})=
\nabla_{{\bf q}_i} U_\kappa({\bf q})
-m_i^{-1}{\kappa}({\bf p}_i\odot{\bf p}_i){\bf q}_i,\cr
{\bf q}_i\odot{\bf q}_i={\kappa}^{-1}, 
\ \  {\bf q}_i\odot{\bf p}_i=0,
\ \ \kappa\ne 0,
\ \  i=1,2,\dots,n.\label{Ham}
\end{cases}
\end{equation}

The configuration space is the manifold $({\mathbb M}^3_\kappa)^n$, where, recall, ${\mathbb M}_\kappa^3$ represents ${\mathbb S}_\kappa^3$ or ${\mathbb H}_\kappa^3$. Then\ $({\bf q},{\bf p})\in {\bf T}^*({\mathbb M}_\kappa^3)^n$,  where ${\bf T}^*({\mathbb M}_\kappa^3)^n$ is the cotangent bundle, which represents the phase space. The constraints $\kappa{\bf q}_i\odot{\bf q}_i=1,\ {\bf q}_i\odot{\bf p}_i=0,\  i=1,2,\dots,n,$ which keep the bodies on the manifold and show that the position vectors and the momenta of each body are orthogonal to each other, reduce the $8n$-dimensional system \eqref{Ham} by  $2n$ dimensions. So, before taking into consideration the first integrals of motion, which we will compute in Section \ref{integralssection}, the phase space has dimension $6n$, as it should, given the fact that we are studying the motion of $n$ bodies in a 3-dimensional space.

\subsection{Invariance of great spheres and great hyperboloids}\label{invariance}

In the Euclidean case, planes are invariant sets for the equations of motion. In other
words, if the initial positions and momenta are contained in a plane, the motion takes place in that plane for all time for which the solution is defined. We can now prove the natural analogue of this result for the curved $n$-body problem. More precisely, we will show that, in ${\mathbb S}_\kappa^3$ and ${\mathbb H}_\kappa^3$, the motion can take place on 2-dimensional great spheres and 2-dimensional great hyperboloids, respectively, if we assign suitable initial positions and momenta.

\begin{proposition}
Let $n\ge 2$ and consider the point particles of masses $m_1,m_2,\dots, m_n>0$ in ${\mathbb M}_\kappa^3$. Assume that ${\mathbb M}_\kappa^2$ is any 2-dimensional submanifold of ${\mathbb M}_\kappa^3$ having the same curvature, i.e.\ a great sphere for $\kappa>0$ or a great hyperboloid for $\kappa<0$. Then, for any nonsingular initial conditions $({\bf q}(0),{\bf p}(0))\in({\mathbb M}_\kappa^2)^n\times (T({\mathbb M}_\kappa^2)^n$, where $\times$ denotes the cartesian product of two sets and $T({\mathbb M}_\kappa^2)$ is the tangent space of ${\mathbb M}_\kappa^2$, the motion takes place in ${\mathbb M}_\kappa^2$.
\label{invariance-prop}
\end{proposition}
\begin{proof}
Without loss of generality, it is enough to prove the result for ${\bf M}_{\kappa,w}^2$,
where  
$$
{\bf M}_{\kappa,w}^2:=\{(w,x,y,z)\ \! |\ \! x^2+y^2+\sigma z^2=\kappa^{-1},\ w=0\}
$$
is the great 2-dimensional sphere ${\bf S}_{\kappa,w}^2$, for $\kappa>0$, and the great 2-dimensional hyperboloid ${\bf H}_{\kappa,w}^2$, for $\kappa<0$, both of which we defined in Subsection \ref{invariance} as \eqref{s2w} and \eqref{h2w}, respectively. 
Indeed, we can obviusly restrict to this case since any great sphere or great hyperboloid can be reduced to it by a suitable change of coordinates. 

Let us denote the coordinates and the momenta of the bodies $m_i,\ i=1,2,\dots,n$, by 
$$
{\bf q}_i=(w_i,x_i,y_i,z_i)\ \ {\rm and}\ \ {\bf p}_i=(r_i,s_i,u_i,v_i), \ i=1,2,\dots,n,
$$
which when restricted to ${\bf M}_{\kappa,w}^2$ and $T({\bf M}_{\kappa,w}^2)$, respectively, become
$$
{\bf q}_i=(0,x_i,y_i,z_i)\ \ {\rm and}\ \ {\bf p}_i=(0,s_i,u_i,v_i), \ i=1,2,\dots,n.
$$

Relative to the first component, $w$, the equations of motion \eqref{Ham} have the form
$$
\begin{cases}
\dot{w}_i=m_i^{-1}{r}_i,\cr
\dot{r}_i=\sum_{j=1,j\ne i}^n{{m_im_j}|\kappa|^{3/2}
\left[{w}_j  - (\kappa{\bf q}_i\odot{\bf q}_j){w}_i \right]
\over
\left[\sigma-\sigma\left(\kappa{\bf q}_i\odot{\bf q}_j\right)^2\right]^{3/2}}
-m_i^{-1}{\kappa}({\bf p}_i\odot{\bf p}_i){w}_i,\cr
{\bf q}_i\odot{\bf q}_i={\kappa}^{-1}, 
\ \  {\bf q}_i\odot{\bf p}_i=0,
\ \ \kappa\ne 0,
\ \  i=1,2,\dots,n.
\end{cases}
$$
For our purposes, we can view this first-order system of differential equations as linear in the variables $w_i, r_i,\ i=1,2,\dots,n$. But on ${\bf M}_{\kappa,w}^2$, the initial conditions are $w_i(0)=r_i(0)=0,\ i=1,2,\dots,n$, therefore $w_i(t)=r_i(t)=0,\ i=1,2,\dots,n$, for all $t$ for which the corresponding solutions are defined. Consequently, for initial conditions $({\bf q}(0),{\bf p}(0))\in{\bf M}_\kappa^2\times T({\bf M}_\kappa^2)$, the motion is confined to ${\bf M}_{\kappa,w}^2$, a remark that completes the proof.
\end{proof}

\section{First integrals}\label{integralssection}

In this section we will determine the first integrals of the equations of motion. These integrals lie at the foundation of the reduction method, which played an important role in the theory of differential equations ever since mathematicians discovered the existence of functions that remain constant along solutions. The classical $n$-body problem in $\mathbb R^3$ has 10 first integrals that are algebraic with respect to ${\bf q}$ and ${\bf p}$, known already to Lagrange in the mid 18th century, \cite{Win}. In 1887, Heinrich Bruns proved that there are no other first integrals, algebraic with respect to ${\bf q}$ and $\bf p$, \cite{Bruns}.  Let us now find the first integrals of the curved $n$-body problem. 

\subsection{The integral of energy}

The Hamiltonian function provides the integral of energy,
\begin{equation}
H_\kappa({\bf q},{\bf p})=h,
\label{energy}
\end{equation}
where $h$ is the energy constant. Indeed, $\odot$-multiplying equations (\ref{eqmotion}) by $\dot{\bf q}_i$, we obtain 
$$
\sum_{i=1}^nm_i\ddot{\bf q}_i\odot\dot{\bf q}_i=
\sum_{i=1}^n[\nabla_{{\bf q}_i} U_\kappa({\bf q})]\odot\dot{\bf q}_i-
\sum_{i=1}^n{m_i\kappa}(\dot{\bf q}_i\odot\dot{\bf q}_i){\bf q}_i\odot{\dot{\bf q}_i}=
{d\over dt}U_\kappa({\bf q}(t)).
$$
Then equation (\ref{energy}) follows by integrating the first and last term in the above 
equation. 

Unlike in the classical Euclidean case, there are no integrals of the centre of mass and the linear momentum. This is not surprising, giving the fact that $n$-body problems obtained by discretizing Einstein's field equations lack these integrals too, \cite{Ein}, \cite{Fock}, \cite{Civ}, \cite{Civita}. Their absence, however, complicates the study of our problem since many of the standard methods don't apply anymore.

One could, of course, define some artificial centre of mass for the particle system, but this move would be to no avail. Indeed, forces do not cancel each other at such a point, as it happens in the Euclidean case, so no advantage can be gained from introducing this concept.

\subsection{The integrals of the total angular momentum}

As we will show below, equations \eqref{Ham} have six angular momentum integrals. To prove their existence, we need to introduce the concept of bivector, which generalizes the idea of vector. If a scalar has dimension 0, and a vector has dimension 1, then a bivector has dimension 2. A bivector is constructed with the help of the wedge product ${\bf a}\wedge{\bf b}$, defined below, of two vectors $\bf a$ and $\bf b$. Its magnitude can be intuitively understood as the oriented area of the parallelogram with edges $\bf a$
and $\bf b$. The wedge product lies in a vector space different from that of the vectors it is generated from. The space of bivectors together with the wedge product is called a Grassmann algebra.

To make these concepts precise, let 
$$
{\bf e}_w=(1,0,0,0),\ {\bf e}_x=(0,1,0,0),\ {\bf e}_y=(0,0,1,0),\ {\bf e}_z=(0,0,0,1)
$$ 
denote the elements of the canonical basis of $\mathbb R^4$, and consider the vectors ${\bf u}=(u_w,u_x,u_y,u_z)$ and ${\bf v}=(v_w,v_x,v_y,v_z)$. We define the wedge product (also called outer product or exterior product) of $\bf u$ and $\bf v$ of $\mathbb R^4$ as
$$
{\bf u}\wedge {\bf v}:=(u_wv_x-u_xv_w)e_w\wedge e_x+(u_wv_y-u_yv_w)e_w\wedge e_y+
$$
\vspace{-22pt}
\begin{equation}
\ \ \ \ \ \ \ \ \ \  \ \!(u_wv_z-u_zv_w)e_w\wedge e_z+(u_xv_y-u_yv_x)e_x\wedge e_y+
\end{equation}
\vspace{-21pt}
$$\ \ \ \ \ \ \  \ \!   (u_xv_z-u_zv_x)e_x\wedge e_z+(u_yv_z-u_zv_y)e_y\wedge e_z,$$
where ${\bf e}_w\wedge {\bf e}_x, {\bf e}_w\wedge {\bf e}_y, {\bf e}_w\wedge {\bf e}_z,
{\bf e}_x\wedge {\bf e}_y,  {\bf e}_x\wedge {\bf e}_z, {\bf e}_y\wedge {\bf e}_z$ represent the bivectors that form a canonical basis of the exterior Grassmann algebra over $\mathbb R^4$ (for more details, see, e.g., \cite{Doran}).
In ${\mathbb R}^3$, the exterior product is equivalent with the cross product. 

Let us define $\sum_{i=1}^nm_i{\bf q}_i\wedge\dot
{\bf q}_i$ to be 
the total angular momentum of the particles of masses $m_1, m_2,\dots, m_n>0$ in 
${\mathbb R}^4$. We will further show that the total angular momentum is conserved for the equations of motion, i.e.
\begin{equation}
\sum_{i=1}^nm_i{\bf q}_i\wedge\dot{\bf q}_i={\bf c},
\label{angintegrals}
\end{equation}
where
${\bf c}=c_{wx}{\bf e}_w\wedge {\bf e}_x+c_{wy}{\bf e}_w\wedge {\bf e}_y+c_{wz}{\bf e}_w\wedge {\bf e}_z+
c_{xy}{\bf e}_x\wedge {\bf e}_y+c_{xz}{\bf e}_x\wedge {\bf e}_z+c_{yz}{\bf e}_y\wedge {\bf e}_z,$ with the coefficients 
$c_{wx}, c_{wy}, c_{wz}, c_{xy}, c_{xz}, c_{yz}\in{\mathbb R}$.
Indeed, relations (\ref{angintegrals}) follow by integrating the identity formed by the first
and last term of the equations
\begin{multline}
\sum_{i=1}^nm_i\ddot{\bf q}_i\wedge{\bf q}_i=\sum_{i=1}^n\sum_{j=1,j\ne i}^n{m_im_j|\kappa|^{3/2}{\bf q}_i\wedge{\bf q}_j
\over
[\sigma-\sigma(\kappa{\bf q}_i\odot{\bf q}_j)^2]^{3/2}}\\
-\sum_{i=1}^n\left[\sum_{j=1,j\ne i}^n{m_im_j|\kappa|^{3/2}(\kappa{\bf q}_i\odot{\bf q}_j)
\over
[\sigma-\sigma(\kappa{\bf q}_i\odot{\bf q}_j)^2]^{3/2}}-
m_i{\kappa}(\dot{\bf q}_i\odot\dot{\bf q}_i)\right]{\bf q}_i\wedge{\bf q}_i
={\bf 0},
\end{multline}
obtained after $\wedge$-multiplying the equations of motion (\ref{eqmotion}) 
by ${\bf q}_i$. The last of the above identities follows from the skew-symmetry of the $\wedge$ operation and the fact that  ${\bf q}_i\wedge{\bf q}_i={\bf 0}, \ i=1,\dots,n$.

On components, the 6 integrals in \eqref{angintegrals} can be written as
\begin{align}
\sum_{i=1}^nm_i(w_i\dot{x}_i-\dot{w}_ix_i)&=c_{wx}, & \sum_{i=1}^nm_i(w_i\dot{y}_i-\dot{w}_iy_i)&=c_{wy},\\
\sum_{i=1}^nm_i(w_i\dot{z}_i-\dot{w}_iz_i)&=c_{wz},& \sum_{i=1}^nm_i(x_i\dot{y}_i-\dot{x}_iy_i)&=c_{xy},\\
\sum_{i=1}^nm_i(x_i\dot{z}_i-\dot{x}_iz_i)&=c_{xz},&  \sum_{i=1}^nm_i(y_i\dot{z}_i-\dot{y}_iz_i)&=c_{yz}.
\label{angularmomentum}
\end{align}
The physical interpretation of these six integrals is related to the geometry of $\mathbb R^4$. The coordinate axes $Ow, Ox, Oy,$ and $Oz$ determine six orthogonal planes, $wx, wy, wz, xy, xz,$ and $yz$. We call them basis planes, since they correspond to the exterior products ${\bf e}_w\wedge{\bf e}_x$, ${\bf e}_w\wedge{\bf e}_y$, ${\bf e}_w\wedge{\bf e}_z$, ${\bf e}_x\wedge{\bf e}_y$, ${\bf e}_x\wedge{\bf e}_z$, and ${\bf e}_y\wedge{\bf e}_z$, respectively, of the basis vectors ${\bf e}_w, {\bf e}_x, {\bf e}_y, {\bf e}_z$ of $\mathbb R^4$. Then the constants $c_{wx}, c_{wy}, c_{wz}, c_{xy}, c_{xz}, c_{yz}$ measure the rotation of an orbit relative to a point in the plane their indices define. This point is the same for all 6 basis planes, namely the origin of the coordinate system.

To clarify this interpretation of rotations in $\mathbb R^4$, let us point out that, in $\mathbb R^3$, rotation is understood as a motion around a pointwise invariant axis orthogonal to a basis plane, which the rotation leaves globally (not pointwise) invariant. In $\mathbb R^4$, there are infinitely many axes orthogonal to this plane, and the angular momentum is the same for them all, since each equation of \eqref{angularmomentum} depends only on the 2 coordinates of the plane and the corresponding velocities. It is, therefore, more convenient to think of these rotations in $\mathbb R^4$ as taking place around a point in a plane, in spite of the fact that the rotation moves points outside the plane too.

Whatever sense of rotation a scalar constant of the angular momentum determines, the opposite sign indicates the opposite sense. A zero scalar constant means that there is no rotation relative to the origin in that particular plane.

Using the constraints $\kappa{\bf q}_i\odot{\bf q}_i=1,\ i=1,2,\dots,n$, we can write system
\eqref{Ham} as
\begin{equation}
\ddot{\bf q}_i=\sum_{\substack{j=1\\ j\ne i}}^n{m_j|\kappa|^{3/2}[{\bf q}_j-(\kappa{\bf q}_i\odot{\bf q}_j){\bf q}_i]\over
[\sigma-\sigma({\kappa{\bf q}_i\odot{\bf q}_j
})^2]^{3/2}}-(\kappa\dot{\bf q}_i\odot\dot{\bf q}_i){\bf q}_i, \ i=1,2,\dots,n,
\label{second}
\end{equation}
which is the form of the equations of motion we will mostly use in this paper. The sums on the right hand side of the above equations represent the $i$th gradient of the potential after the constraints are taken into account. 

If we regard these equations as a first order system with constraints, 
\begin{equation}
\begin{cases}
\dot{\bf q}_i=m_i^{-1}{\bf p}_i,\cr
\dot{\bf p}_i=\sum_{\substack{j=1\\ j\ne i}}^n{m_j|\kappa|^{3/2}[{\bf q}_j-(\kappa{\bf q}_i\odot{\bf q}_j){\bf q}_i]\over
[\sigma-\sigma({\kappa{\bf q}_i\odot{\bf q}_j
})^2]^{3/2}}-(\kappa\dot{\bf q}_i\odot\dot{\bf q}_i){\bf q}_i,\cr
{\bf q}_i\odot{\bf q}_i=1, \ {\bf q}_i\odot{\bf p}_i=0, \ i=1,2,\dots,n,\cr
\end{cases}
\end{equation}
the dimension of the phase space, after taking into account the integrals of motion described above, is $6n-7$.

\section{Singularities}

Before we begin the study of relative equilibria, it is important to know whether there exist impossible configurations of the bodies. The answer is positive, and these configurations occur when system \eqref{second} encounters singularities, i.e.\ if at least one denominator in the sum on the right hand sides of the equations occurring in system \eqref{second} vanishes. In other words, a configuration is singular when $({\kappa{\bf q}_i\odot{\bf q}_j})^2=1,\ i,j=1,2,\dots,n,\ i\ne j$, which is the same as saying that ${\kappa{\bf q}_i\odot{\bf q}_j}=1$ or ${\kappa{\bf q}_i\odot{\bf q}_j}=-1,\  i,j=1,2,\dots,n,\ i\ne j$. The following result shows that the former case corresponds to collisions, i.e.\ to configurations for which at least two bodies have identical coordinates, whereas the latter case occurs in ${\mathbb S}_\kappa^3$, but not in ${\mathbb H}_\kappa^3$, and corresponds to antipodal configurations, i.e.\ when at least two bodies have coordinates of opposite signs. These are impossible initial configurations all the solutions we will introduce in this paper must avoid.

\begin{proposition}{\bf (Collision and antipodal configurations)} 
Consider the $3$-dimen\-sional curved $n$-body problem, $n\ge 2$, with masses $m_1,m_2,\dots,m_n>0$. Then, in $\mathbb S_\kappa^3$, if there are $i,j\in\{1,2,\dots,n\},\ i\ne j$, such that $\kappa{\bf q}_i\cdot{\bf q}_j=1$, the bodies $m_i$ and $m_j$ form a collision configuration, and if $\kappa{\bf q}_i\odot{\bf q}_j=-1$, they form an antipodal configuration. In $\mathbb H_\kappa^3$, if there are $i,j\in\{1,2,\dots,n\},\ i\ne j$, such that $\kappa{\bf q}_i\boxdot{\bf q}_j=1$, the bodies $m_i$ and $m_j$ form a collision configuration, whereas configurations with $\kappa{\bf q}_i\boxdot{\bf q}_j=-1$ don't exist. 
\end{proposition}
\begin{proof}
Let us first prove the implication related to collision configurations for $\kappa>0$. Assume that there exist $i,j\in\{1,2,\dots, n\},\ i\ne j$, such that $\kappa{\bf q}_i\odot{\bf q}_j=1$, relationship that can be written as 
\begin{equation}
\kappa(w_iw_j+x_ix_j+y_iy_j+z_iz_j)=1.
\label{1}
\end{equation}
But since the bodies are on ${\mathbb S}_\kappa^3$, the coordinates satisfy the conditions
\begin{equation*}
\kappa(w_i^2+x_i^2+y_i^2+z_i^2)=\kappa(w_j^2+x_j^2+y_j^2+z_j^2)=1.
\end{equation*}
Consequently we can write that
\begin{equation*}
(w_iw_j+x_ix_j+y_iy_j+z_iz_j)^2=(w_i^2+x_i^2+y_i^2+\sigma z_i^2)(w_j^2+x_j^2+y_j^2+z_j^2),
\end{equation*}
which is the equality case of Cauchy's inequality.
Therefore there is a constant $\tau\ne 0$ such that
$w_j=\tau w_i,\ x_j=\tau x_i,\ y_j=\tau y_i$, and $z_j=\tau z_i$. Substituting these values in equation \eqref{1}, we obtain that 
\begin{equation}
\kappa\tau(w_iw_j+x_ix_j+y_iy_j+z_iz_j)=1.
\label{1=1}
\end{equation}
Comparing \eqref{1} and \eqref{1=1}, it follows that
$\tau=1$, so $w_i=w_j$, $x_i=x_j$, $y_i=y_j$, and $z_i=z_j$, therefore the bodies $m_i$ and $m_j$ form a collision configuration.

The proof of the implication related to antipodal configurations for $\kappa>0$ is very similar. Instead of relation \eqref{1}, we have
\begin{equation*}
\kappa(w_iw_j+x_ix_j+y_iy_j+z_iz_j)=-1.
\label{opposite1}
\end{equation*} 
Then, following the above steps, we obtain that $\tau=-1$, so $w_i=-w_j$, $x_i=-x_j$, $y_i=-y_j$, and $z_i=-z_j$, therefore the bodies $m_i$ and $m_j$ form an antipodal configuration.

Let us now prove the implication related to collision configurations in the case $\kappa<0$. Assume that there exist $i,j\in\{1,2,\dots, n\},\ i\ne j$, such that $\kappa{\bf q}_i\odot{\bf q}_j=1$, relationship that can be written as
\begin{equation*}
\kappa(w_iw_j+x_ix_j+y_iy_j-z_iz_j)=1,
\end{equation*} 
which is equivalent to
\begin{equation}
w_iw_j+x_ix_j+y_iy_j-\kappa^{-1}=z_iz_j.
\label{equiv-1}
\end{equation}
But since the bodies are on ${\mathbb H}_\kappa^3$, the coordinates satisfy the conditions
\begin{equation*}
\kappa(w_i^2+x_i^2+y_i^2-z_i^2)=\kappa(w_j^2+x_j^2+y_j^2-z_j^2)=1,
\label{-1=-1}
\end{equation*}
which are equivalent to 
\begin{equation}
w_i^2+x_i^2+y_i^2-\kappa^{-1}=z_i^2\ \
{\rm and} \ \ w_j^2+x_j^2+y_j^2-\kappa^{-1}=z_j^2.
\label{equiv-1=-1}
\end{equation}
From \eqref{equiv-1} and \eqref{equiv-1=-1}, we can conclude that 
\begin{equation*}
(w_iw_j+x_ix_j+y_iy_j-\kappa^{-1})^2=
(w_i^2+x_i^2+y_i^2-\kappa^{-1})(w_j^2+x_j^2+y_j^2-\kappa^{-1}),
\end{equation*}
which is equivalent to
$$
(w_ix_j-w_jx_i)^2+(w_iy_j-w_jy_i)^2+(x_iy_j-x_jy_i)^2
$$
$$
-\kappa^{-1}[(w_i-w_j)^2+(x_i-x_j)^2+(y_i-y_j)^2]=0.
$$
Since $-\kappa^{-1}>0$, it follows from the above relation that 
$w_i=w_j,\ x_i=x_j$, and $y_i=y_j$. Then relation \eqref{equiv-1=-1} implies that $z_i^2=z_j^2$. But since for $\kappa<0$ all $z$ coordinates are positive, we can conclude that $z_i=z_j$, so the bodies $m_i$ and $m_j$ form a collision configuration.

For $\kappa<0$, we now finally prove the non-existence of configurations with $\kappa{\bf q}_i\boxdot{\bf q}_j=-1, \ i,j\in\{1,2,\dots,n\}, i\ne j$. Let us assume that they exist. Then
\begin{equation}
\label{product}
w_iw_j+x_ix_j+y_iy_j=z_iz_j-\kappa^{-1},
\end{equation}  
\begin{equation}
\label{squares}
w_i^2+x_i^2+y_i^2=z_i^2+\kappa^{-1}\ \ {\rm and}\ \ w_j^2+x_j^2+y_j^2=z_j^2+\kappa^{-1}.
\end{equation}
According to Cauchy's inequality, we have
$$
(w_iw_j+x_ix_j+y_iy_j)^2\le(w_i^2+x_i^2+y_i^2)(w_j^2+x_j^2+y_j^2),
$$
which, using \eqref{product} and \eqref{squares}, becomes
$$
(z_iz_j-\kappa^{-1})^2\le(z_i^2+\kappa^{-1})(z_j^2+\kappa^{-1}).
$$
Since $-\kappa^{-1}>0$, this inequality takes the form
$$
(z_i+z_j)^2\le 0,
$$
which is impossible because $z_i,z_j>0$, a contradiction that completes the proof.
\end{proof}

It is easy to construct solutions ending in collisions. Place, for instance, 3 bodies of equal masses at the vertices of an Euclidean equilateral triangle, not lying on the same geodesic if $\kappa>0$, and release them with zero initial velocities. The bodies will end up in a triple collision. (If, for $\kappa>0$, the bodies lie initially on a geodesic and have zero initial velocities, they won't move in time, a situation that corresponds to a fixed-point solution for the equations of motion, as we will show in Section \ref{fixedpoints}.) The question of whether there exist solutions ending in antipodal or hybrid (collision-antipodal) singularities is harder, and it was treated in \cite{Diacu1} and \cite{Diacu002}. But since we are not touching on this subject when dealing with relative equilibria, we will not discuss it further. All we need to worry about in this paper is to avoid placing the bodies at singular initial configurations, i.e.\ at collisions for $\kappa\ne 0$ or at antipodal positions for $\kappa>0$.

\newpage

\part{\rm ISOMETRIES AND RELATIVE EQUILIBRIA}

\section{Isometric rotations in ${\mathbb S}_\kappa^3$ and ${\mathbb H}_\kappa^3$}\label{isometries}

This section introduces the isometric rotations in ${\mathbb S}_\kappa^3$ and ${\mathbb H}_\kappa^3$, since they play an essential role in defining the relative equilibria of the curved $n$-body problem. There are many ways to express these rotations, but their matrix representation will suit our goals best, as it did in Subsection \ref{Weierstrass} for the 2-dimensional Weierstrass model of hyperbolic geometry. 

For $\kappa>0$, the isometric transformations of ${\mathbb S}_\kappa^3$ are given by the elements of the Lie group $SO(4)$ of ${\mathbb R}^4$ that keep ${\mathbb S}_\kappa^3$ invariant. They consist of all orthogonal transformations of the Lie group $O(4)$ represented by matrices of determinant 1, and have the form $PAP^{-1}$, with $P\in SO(4)$ and
\begin{equation}
A=
\left( \begin{array}{cccc}
\cos\theta & -\sin\theta & 0 & 0 \\
\sin\theta & \cos\theta & 0 & 0 \\
0 & 0 & \cos\phi & -\sin\phi\\
0 & 0 & \sin\phi & \cos\phi \end{array} \right),
\label{e-elliptic}
\end{equation}
where $\theta, \phi\in\mathbb R$. We will call these rotations $\kappa$-positive elliptic-elliptic if $\theta\ne 0$ and $\phi\ne 0$, and $\kappa$-positive elliptic if $\theta\ne 0$ and $\phi=0$ (or $\theta= 0$ and $\phi\ne 0$, which is a possibility we will never discuss since it perfectly resembles the previous one). When $\theta=\phi=0$, $A$ is the identity matrix, so no rotation takes place. The above description is a generalization to ${\mathbb S}_\kappa^3$ of Euler's Fixed Axis Theorem for ${\mathbb S}_\kappa^2$.
As we will next explain, the reference to a fixed axis is, from the geometric point of view, far from suggestive in $\mathbb R^4$.

The form of the matrix $A$ given by \eqref{e-elliptic} shows that the $\kappa$-positive elliptic-elliptic transformations have two circular rotations, one relative to the origin of the coordinate system in the plane $wx$ and the other relative to the same point in the plane $yz$. In this case, the only fixed point in $\mathbb R^4$ is the origin of the coordinate system. The $\kappa$-positive elliptic transformations have a single rotation around the origin of the coordinate system that leaves infinitely many axes (in fact, an entire plan) of $\mathbb R^4$ pointwise fixed.

For $\kappa<0$, the isometric transformations of ${\mathbb H}_\kappa^3$ are given by the elements of the Lorentz group ${\rm Lor}(3,1)$, a Lie group in the Minkowski space ${\mathbb R}^{3,1}$. ${\rm Lor}(3,1)$ is formed by all orthogonal transformations of determinant 1 that keep ${\mathbb H}_\kappa^3$ invariant. The elements of this group are the $\kappa$-negative elliptic, $\kappa$-negative hyperbolic, and $\kappa$-negative elliptic-hyperbolic transformations, on one hand, and the $\kappa$-negative parabolic transformations, on the other hand; they can be represented as matrices of the form $PBP^{-1}$ and $PCP^{-1}$, respectively, with $P\in{\rm Lor}(3,1)$, 
\begin{equation}
B=\left( \begin{array}{cccc}
\cos\theta & -\sin\theta & 0 & 0 \\
\sin\theta & \cos\theta & 0 & 0 \\
0 & 0 & \cosh s & \sinh s\\
0 & 0 & \sinh s & \cosh s \end{array} \right),
\label{e-hyperbolic} 
\end{equation}
\begin{equation}
C=\left( \begin{array}{cccc}
1 & 0 & 0 & 0\\
0 &1 & -\xi & \xi  \\
0 & \xi & 1-\xi^2/2 & \xi^2/2 \\
0 & \xi & -\xi^2/2 & 1+\xi^2/2 \end{array} \right),
\label{p-parabolic}
\end{equation}
where $\theta, s, \xi$ are some fixed values in $\mathbb R$. The $\kappa$-negative elliptic, $\kappa$-negative hyperbolic, and $\kappa$-negative elliptic-hyperbolic
transformations correspond to $\theta\ne 0$ and $s=0$, to $\theta=0$ and $s\ne 0$, and to $\theta\ne 0$ and $s\ne 0$, respectively. The above description is a generalization to ${\mathbb H}_\kappa^3$ of the Fixed Axis Theorem for ${\mathbb H}_\kappa^2$, which we presented in Subsection \ref{Weierstrass}. Again, as in the case of the group $SO(4)$, the reference to a fixed axis has no real geometric meaning in $\mathbb R^{3,1}$.

Indeed, from the geometric point of view, the $\kappa$-negative elliptic transformations of $\mathbb R^{3,1}$ are similar to their counterpart, $\kappa$-positive elliptic transformations, in $\mathbb R^4$, namely they have a single circular rotation around the origin of the coordinate system that leaves infinitely many axes of $\mathbb R^{3,1}$ pointwise invariant. The $\kappa$-negative hyperbolic transformations correspond to a single hyperbolic rotation around the origin of the coordinate system that also leaves infinitely many axes of $\mathbb R^{3,1}$ pointwise invariant. The $\kappa$-negative elliptic-hyperbolic transformations have two rotations, a circular one about the origin of the coordinate system, relative to the $wx$ plane, and a hyperbolic one about the origin of the coordinate system, relative to the $yz$ plane. The only point they leave fixed is the origin of the coordinate system. Finally, parabolic transformations correspond to parabolic rotations about the origin of the coordinate system that leave only the $w$ axis pointwise fixed.

\section{Some geometric properties of the isometric rotations}

In this section we aim to understand how the previously defined isometries act in ${\mathbb S}_\kappa^3$ and ${\mathbb H}_\kappa^3$. In fact we will be interested only in the transformations represented by the matrices $A$ and $B$, defined in \eqref{e-elliptic} and \eqref{e-hyperbolic}, respectively. As we will see in Subsection \ref{non-parabolic}, the $\kappa$-negative parabolic rotations represented by the matrix $C$, defined in \eqref{p-parabolic}, generate no relative equilibria in ${\mathbb H}_\kappa^3$, so we don't need to worry about their geometric properties for the purposes of this paper.

Since our earlier work focused on the curved $n$-body problem in the 2-dimensional manifolds ${\mathbb S}_\kappa^2$ and ${\mathbb H}_\kappa^2$, we would like to see whether the above rotations preserve 2-dimensional spheres in ${\mathbb S}_\kappa^3$ and ${\mathbb H}_\kappa^3$
and 2-dimensional hyperboloids in ${\mathbb H}_\kappa^3$.
We begin with the spheres.

\subsection{Invariance of 2-dimensional spheres in
${\mathbb S}_\kappa^3$ and ${\mathbb H}_\kappa^3$}

Let us start with the $\kappa$-positive elliptic-elliptic rotations in ${\mathbb S}_\kappa^3$ and consider first great spheres, which can be obtained, for instance, by the intersection of ${\mathbb S}_\kappa^3$ with the hyperplane $z=0$. We thus obtain the 2-dimensional great sphere
\begin{equation}
{\bf S}_{\kappa,z}^2=\{(w,x,y,0)|\ \! w^2+x^2+y^2=\kappa^{-1}\},
\end{equation}
already defined in \eqref{z-sphere}. Let $(w,x,y,0)$ to be a point on ${\bf S}_{\kappa,z}^2$. Then the $\kappa$-positive elliptic-elliptic transformation \eqref{e-elliptic} takes $(w,x,y,0)$ to the point $(w_1,x_1,y_1,z_1)$ given by 
\begin{equation}
\left(\begin{array}{c}
w_1\\
x_1\\
y_1\\
z_1
\end{array}\right)=
\left( \begin{array}{cccc}
\cos\theta & -\sin\theta & 0 & 0 \\
\sin\theta & \cos\theta & 0 & 0 \\
0 & 0 & \cos\phi & -\sin\phi\\
0 & 0 & \sin\phi & \cos\phi 
\end{array} \right)
\left(\begin{array}{c}
w\\
x\\
y\\
0
\end{array}\right)=
\end{equation}
$$
\left(\begin{array}{c}
w\cos\theta-x\sin\theta\\
w\sin\theta+x\cos\theta\\
y\cos\phi\\
y\sin\phi
\end{array}\right).
$$
Since, in general, $y$ is not zero, it follows that $z_1=y\sin\phi=0$ only if $\phi=0$, a case that corresponds to $\kappa$-positive elliptic transformations. In case of a $\kappa$-positive elliptic-elliptic transformation, the point $(w_1,x_1,y_1,z_1)$ does not lie on ${\bf S}_{\kappa,z}^2$ because this point is not in the hyperplane $z=0$. Without loss of generality, we can always find a coordinate system in which the considered sphere is ${\bf S}_{\kappa,z}^2$. We can therefore draw the following conclusion.

\begin{remark}
For every great sphere of ${\mathbb S}_\kappa^3$, there is a suitable system of coordinates such that $\kappa$-positive elliptic rotations leave the great sphere invariant. However, there is no system of coordinates for which we can find a $\kappa$-positive elliptic-elliptic rotation that leaves a great sphere invariant.
\label{z=0-remark}
\end{remark}

Let us now see what happens with non-great spheres of ${\mathbb S}_\kappa^3$. Such spheres can be obtained, for instance, by intersecting ${\mathbb S}_\kappa^3$ with a hyperplane $z=z_0$, where $|z_0|<\kappa^{-1/2}$ and $z_0\ne 0$. These 2-dimensional sphere are of the form
\begin{equation}
{\bf S}_{\kappa_0,z_0}^2=\{(w,x,y,z)|\ \! w^2+x^2+y^2=\kappa^{-1}-z_0^2,\ \! z=z_0\},
\end{equation}
where $\kappa_0=(\kappa^{-1}-z_0^2)^{-1/2}$ is its curvature.

Let $(w,x,y,z_0)$ be a point on a non-great sphere ${\bf S}_{\kappa_0,z_0}^2$, given by some $z_0$ as above. Then the $\kappa$-positive elliptic-elliptic transformation \eqref{e-elliptic} takes the point $(w,x,y,z_0)$ to the point  
$(w_2,x_2,y_2,z_2)$ given by

\begin{equation}
\left(\begin{array}{c}
w_2\\
x_2\\
y_2\\
z_2
\end{array}\right)=
\left( \begin{array}{cccc}
\cos\theta & -\sin\theta & 0 & 0 \\
\sin\theta & \cos\theta & 0 & 0 \\
0 & 0 & \cos\phi & -\sin\phi\\
0 & 0 & \sin\phi & \cos\phi 
\end{array} \right)
\left(\begin{array}{c}
w\\
x\\
y\\
z_0
\end{array}\right)=
\end{equation}
$$
\left(\begin{array}{c}
w\cos\theta-x\sin\theta\\
w\sin\theta+x\cos\theta\\
y\cos\phi-z_0\sin\phi\\
y\sin\phi+z_0\cos\phi
\end{array}\right).
$$
Since, in general, $y$ is not zero, it follows that $z_2=y\sin\phi+z_0\cos\phi=z_0$ only if $\phi=0$, a case that corresponds to $\kappa$-positive elliptic transformations. In the case of a $\kappa$-positive elliptic-elliptic transformation, the point $(w_2,x_2,y_2,z_2)$ does not lie on ${\bf S}_{\kappa_0,z_0}^2$ because this point is not in the hyperplane $z=z_0$. Without loss of generality, we can always reduce the question we posed above to the sphere ${\bf S}_{\kappa_0,z_0}^2$. We can therefore draw the following conclusion, which resembles Remark \ref{z=0-remark}. 

\begin{remark}
For every non-great sphere of ${\mathbb S}_\kappa^3$, there is a suitable system of coordinates such that $\kappa$-positive elliptic rotations leave that non-great sphere invariant. However, there is no system of coordinates for which there exists a $\kappa$-positive elliptic-elliptic rotation that leaves a non-great sphere invariant.
\label{z=z_0-remark}
\end{remark}

Since in ${\mathbb H}_\kappa^3$ we have $z>0$, 2-dimensional spheres cannot be centred around the origin of the coordinate system. We therefore look for 2-dimensional spheres centred on the $z$ axis, with $z>|\kappa|^{-1/2}$ (because $z=|\kappa|^{-1/2}$ is the smallest possible $z$ coordinate in ${\mathbb H}_\kappa^3$, attained only by the point $(0,0,0,|\kappa|^{-1/2})$. Such spheres can be obtained by intersecting ${\mathbb H}_\kappa^3$ with a plane $z=z_0$, where $z_0>|\kappa|^{-1/2}$, and they are given by
\begin{equation}
{\bf S}_{\kappa_0,z_0}^{2,h}=\{(w,x,y,z)|\ \! w^2+x^2+y^2=z_0^2+\kappa^{-1},\ \! z=z_0\},
\end{equation}
where $h$ indicates that the spheres lie in a 3-dimensional hyperbolic space, and $\kappa_0=(z_0^2+\kappa^{-1})^{-1/2}>0$ is the curvature of the sphere.

Let $(w,x,y,z_0)$ be a point on the sphere ${\bf S}_{\kappa_0,z_0}^{2,h}$. Then the $\kappa$-negative elliptic transformation $B$, given by \eqref{e-hyperbolic} with $s=0$, takes the point $(w,x,y,z_0)$ to the point 
$(w_3,x_3,y_3,z_3)$ given by
\begin{equation}
\left(\begin{array}{c}
w_3\\
x_3\\
y_3\\
z_3
\end{array}\right)=
\left( \begin{array}{cccc}
\cos\theta & -\sin\theta & 0 & 0 \\
\sin\theta & \cos\theta & 0 & 0 \\
0 & 0 & 1 & 0\\
0 & 0 & 0 & 1 
\end{array} \right)
\left(\begin{array}{c}
w\\
x\\
y\\
z_0
\end{array}\right)=
\end{equation}
$$
\left(\begin{array}{c}
w\cos\theta-x\sin\theta\\
w\sin\theta+x\cos\theta\\
y\\
z_0
\end{array}\right),
$$
which also lies on the sphere ${\bf S}_{\kappa_0,z_0}^{2,h}$. Indeed, since 
$z_3=z_0$ and $w_3^2+x_3^2+y_3^2=z_0^2+\kappa^{-1}$, it means that  $(w_3,x_3,y_3,z_3)$ also lies on the sphere ${\bf S}_{\kappa_0,z_0}^{2,h}$. 

Since for any 2-dimensional sphere of ${\mathbb H}_\kappa^3$ we can find a coordinate system and suitable values for $\kappa_0$ and $z_0$ such that the sphere has the form ${\bf S}_{\kappa_0,z_0}^{2,h}$, we can draw the following conclusion.

\begin{remark}
For every 2-dimensional sphere of ${\mathbb H}_\kappa^3$, there is a system of coordinates such that $\kappa$-negative  elliptic rotations leave the sphere invariant.
\end{remark}

Let us further see what happens with $\kappa$-negative hyperbolic transformations in 
${\mathbb H}_\kappa^3$.
Let $(w,x,y,z_0)$ be a point on the sphere ${\bf S}_{\kappa_0,z_0}^{2,h}$. Then the $\kappa$-negative  hyperbolic transformation $B$, given by \eqref{e-hyperbolic} with $\theta=0$, takes the point $(w,x,y,z_0)$ to $(w_4,x_4,y_4,z_4)$ given by
\begin{equation}
\left(\begin{array}{c}
w_4\\
x_4\\
y_4\\
z_4
\end{array}\right)=
\left( \begin{array}{cccc}
1 & 0 & 0 & 0 \\
0 & 1 & 0 & 0 \\
0 & 0 & \cosh s & \sinh s\\
0 & 0 & \sinh s & \cosh s 
\end{array} \right)
\left(\begin{array}{c}
w\\
x\\
y\\
z_0
\end{array}\right)=
\end{equation}
$$
\left(\begin{array}{c}
w\\
x\\
y\cosh s+z_0\sinh s\\
y\sinh s+z_0\cosh s
\end{array}\right),
$$
which does not lie on ${\bf S}_{\kappa_0,z_0}^{2,h}$. Indeed, since $z_3=y\sinh s+z_0\cosh s=z_0$ only for $s=0$, a case we exclude because the above transformation is the identity, the point $(w_4,x_4,y_4,z_4)$ does not lie on a sphere of radius $\sqrt{z_0^2+\kappa^{-1}}$. Therefore we can draw the following conclusion.

\begin{remark}
Given a 2-dimensional sphere of curvature $\kappa_0=(z_0+\kappa^{-1})^{-1/2}$, with $z_0>|\kappa|^{-1/2}$, in ${\mathbb H}_\kappa^3$, there is no coordinate system for which some $\kappa$-negative hyperbolic transformation would leave the sphere invariant. Consequently the same is true about $\kappa$-negative  elliptic-hyper\-bolic transformations.
\end{remark}

We will further see how the problem of invariance relates to 2-dimensional hyperboloids in ${\mathbb H}_\kappa^3$. Let us start with $\kappa$-negative elliptic transformations in ${\mathbb H}_\kappa^3$.

\subsection{Invariance of 2-dimensional hyperboloids in
${\mathbb H}_\kappa^3$}

Let us first check whether $\kappa$-negative  elliptic rotations preserve the great 2-dimensional hyperboloids of ${\mathbb H}_\kappa^3$. For this consider the 2-dimensional hyperboloid
\begin{equation}
{\bf H}_{\kappa, y}^2=\{(w,x,0,z)\ |\ w^2+x^2-z^2=\kappa^{-1}\},
\end{equation}
already defined in \eqref{great-Y}, and obtained by intersecting ${\mathbb H}_\kappa^3$ with the hyperplane $y=0$. Let $(w,x,0,z)$ be a point on ${\bf H}_{\kappa, y}^2$. Then a $\kappa$-negative  elliptic rotation takes the point $(w,x,0,z)$ to the point $(w_5,x_5,y_5,z_5)$ given by
\begin{equation}
\left(\begin{array}{c}
w_5\\
x_5\\
y_5\\
z_5
\end{array}\right)=
\left( \begin{array}{cccc}
\cos\theta & -\sin\theta & 0 & 0 \\
\sin\theta & \cos\theta & 0 & 0 \\
0 & 0 & 1 & 0\\
0 & 0 & 0 & 1 
\end{array} \right)
\left(\begin{array}{c}
w\\
x\\
0\\
z
\end{array}\right)=
\end{equation}
$$
\left(\begin{array}{c}
w\cos\theta-x\sin\theta\\
w\sin\theta+x\cos\theta\\
0\\
z
\end{array}\right),
$$
which, obviously, also belongs to ${\bf H}_{\kappa, y}^2$. Since for any 2-dimensional hyperboloid of curvature $\kappa$ we can find a coordinate system such that the hyperboloid can be represented as ${\bf H}_{\kappa, y}^2$, we can draw the following conclusion.

\begin{remark}
Given a 2-dimensional hyperboloid of curvature $\kappa$ in ${\mathbb H}_\kappa^3$, there is a coordinate system for which the hyperboloid is invariant to $\kappa$-negative  elliptic rotations.
\label{remark-ellip}
\end{remark}

Let us further check what happens in the case of $\kappa$-negative hyperbolic rotations. Consider the 2-dimensional hyperboloid of curvature $\kappa$ given by
\begin{equation}
{\bf H}_{\kappa, w}^2=\{(0,x,y,z)\ |\ x^2+y^2-z^2=\kappa^{-1},\},
\end{equation}
already defined in \eqref{h2w}, and obtained by intersecting ${\mathbb H}_\kappa^3$ with the hyperplane $w=0$. Let $(0,x,y,z)$ be a point on ${\bf H}_{\kappa, w}^2$. Then a $\kappa$-negative hyperbolic rotation takes the point $(0,x,y,z)$ to the point $(w_6,x_6,y_6,z_6)$ given by
\begin{equation}
\left(\begin{array}{c}
w_6\\
x_6\\
y_6\\
z_6
\end{array}\right)=
\left( \begin{array}{cccc}
1 & 0 & 0 & 0 \\
0 & 1 & 0 & 0 \\
0 & 0 & \cosh s & \sinh s\\
0 & 0 & \sinh s & \cosh s 
\end{array} \right)
\left(\begin{array}{c}
0\\
x\\
y\\
z
\end{array}\right)=
\end{equation}
$$
\left(\begin{array}{c}
0\\
x\\
y\cosh s+z\sinh s\\
y\sinh s+z\sinh s
\end{array}\right),
$$
which, obviously, also belongs to ${\bf H}_{\kappa, w}^2$. Since for any 2-dimensional hyperboloid of curvature $\kappa$ we can find a coordinate system such that the hyperboloid can be represented as ${\bf H}_{\kappa, w}^2$, we can draw the following conclusion.

\begin{remark}
Given a 2-dimensional hyperboloid of curvature $\kappa$ in ${\mathbb H}_\kappa^3$, there is a coordinate system for which the hyperboloid is invariant to $\kappa$-negative hyperbolic rotations.
\label{remark-hyp}
\end{remark}

\begin{remark}
The coordinate system in Remark \ref{remark-ellip} is different from the coordinate system in Remark \ref{remark-hyp}, so $\kappa$-negative elliptic-hyperbolic transformation don't leave 2-dimensional hyperboloids of curvature $\kappa$ invariant in ${\mathbb H}_\kappa^3$.
\end{remark}

The next step is to see whether $\kappa$-negative  elliptic rotations preserve the 2-dimensional hyperboloids of curvature $\kappa_0=-(y_0^2-\kappa^{-1})^{-1/2}\ne\kappa$ of ${\mathbb H}_\kappa^3$. For this consider the 2-dimensional hyperboloid
\begin{equation}
{\bf H}_{\kappa_0, y_0}^2=\{(w,x,y,z)\ |\ w^2+x^2-z^2=\kappa^{-1}-y_0^2,\ \! y=y_0\},
\end{equation}
obtained by intersecting ${\mathbb H}_\kappa^3$ with the hyperplane $y=y_0$, with $y_0\ne 0$. Let $(w,x,y_0,z)$ be a point on ${\bf H}_{\kappa, y}^2$. Then a $\kappa$-negative  elliptic rotation takes the point $(w,x,y_0,z)$ to the point $(w_7,x_7,y_7,z_7)$ given by
\begin{equation}
\left(\begin{array}{c}
w_7\\
x_7\\
y_7\\
z_7
\end{array}\right)=
\left( \begin{array}{cccc}
\cos\theta & -\sin\theta & 0 & 0 \\
\sin\theta & \cos\theta & 0 & 0 \\
0 & 0 & 1 & 0\\
0 & 0 & 0 & 1 
\end{array} \right)
\left(\begin{array}{c}
w\\
x\\
y_0\\
z
\end{array}\right)=
\end{equation}
$$
\left(\begin{array}{c}
w\cos\theta-x\sin\theta\\
w\sin\theta+x\cos\theta\\
y_0\\
z
\end{array}\right),
$$
which, obviously, also belongs to ${\bf H}_{\kappa_0, y_0}^2$. Since for any 2-dimensional hyperboloid of curvature $\kappa_0$ we can find a coordinate system such that the hyperboloid can be represented as ${\bf H}_{\kappa_0, y_0}^2$, we can draw the following conclusion.

\begin{remark}
Given a 2-dimensional hyperboloid of curvature $\kappa_0=-(y_0^2-\kappa^{-1})^{-1/2}\ne\kappa$ in ${\mathbb H}_\kappa^3$, there is a coordinate system for which the hyperboloid is invariant to $\kappa$-negative  elliptic rotations.
\label{remark-ellip0}
\end{remark}

Consider further the 2-dimensional hyperboloid of curvature $\kappa_0=-(w_0^2-\kappa^{-1})^{-1/2}\ne\kappa$ given by
\begin{equation}
{\bf H}_{\kappa_0, w_0}^2=\{(w,x,y,z)\ |\ x^2+y^2-z^2=\kappa^{-1}-w_0^2,\ \! w=w_0\},
\end{equation}
obtained by intersecting ${\mathbb H}_\kappa^3$ with the hyperplane $w=w_0\ne 0$. Let the point $(w_0,x,y,z)$ lie on ${\bf H}_{\kappa_0, w_0}^2$. Then a $\kappa$-negative hyperbolic rotation takes the point $(w_0,x,y,z)$ to the point $(w_8,x_8,y_8,z_8)$ given by
\begin{equation}
\left(\begin{array}{c}
w_8\\
x_8\\
y_8\\
z_8
\end{array}\right)=
\left( \begin{array}{cccc}
1 & 0 & 0 & 0 \\
0 & 1 & 0 & 0 \\
0 & 0 & \cosh s & \sinh s\\
0 & 0 & \sinh s & \cosh s 
\end{array} \right)
\left(\begin{array}{c}
w_0\\
x\\
y\\
z
\end{array}\right)=
\end{equation}
$$
\left(\begin{array}{c}
w_0\\
x\\
y\cosh s+z\sinh s\\
y\sinh s+z\sinh s
\end{array}\right),
$$
which, obviously, also belongs to ${\bf H}_{\kappa_0, w_0}^2$. Since for any 2-dimensional hyperboloid of curvature $\kappa_0=-(w_0^2-\kappa^{-1/2})^{-1}$ we can find a coordinate system such that the hyperboloid can be represented as ${\bf H}_{\kappa_0, w_0}^2$, we can draw the following conclusion.

\begin{remark}
Given a 2-dimensional hyperboloid of curvature $\kappa_0=-(w_0^2-\kappa^{-1})^{-1/2}$ in ${\mathbb H}_\kappa^3$, there is a coordinate system for which the hyperboloid is invariant to $\kappa$-negative  hyperbolic rotations.
\label{remark-hyp0}
\end{remark}

\begin{remark}
Since the coordinate system in Remark \ref{remark-ellip0} is different from the coordinate system in Remark \ref{remark-hyp0}, $\kappa$-negative  elliptic-hyperbolic transformation don't leave 2-dimensional hyperboloids of curvature $\kappa_0\ne\kappa$ invariant in ${\mathbb H}_\kappa^3$.
\label{remark-ellip-hyp}
\end{remark}

\section{Relative equilibria}\label{releq}

The goal of this section is to introduce the concepts we will explore in the rest of the paper, namely the relative equilibrium solutions (also called relative equilibrium orbits or, simply, relative equilibria) of the curved $n$-body problem. For relative equilibria, the particle system behaves like a rigid body, i.e.\ all the mutual distances between the point masses remain constant during the motion. In other words, the bodies move under the action of an element belonging to a rotation group, so, in the light of Section \ref{isometries}, we can define 6 types of relative equilibria in ${\mathbb M}_\kappa^3$: 2 in ${\mathbb S}_\kappa^3$ and 4 in ${\mathbb H}_\kappa^3$. In each case, we will bring the expressions involved in these natural definitions to simpler forms. We will later see that 1 of the 4 types of relative equilibria we define in ${\mathbb H}_\kappa^3$ does not translate into solutions of the equations of motion. In what follows, the upper $T$ will denote the transpose of a vector.

\subsection{Definition of $\kappa$-positive elliptic relative equilibria}

The first kind of relative equilibria we will introduce in this paper are inspired by the $\kappa$-positive elliptic rotations of ${\mathbb S}_\kappa^3$.

\begin{definition}{\bf ($\kappa$-positive elliptic relative equilibria)}
Let ${\bf q}^0=({\bf q}_1^0, {\bf q}_2^0,\dots, {\bf q}_n^0)$ be a nonsingular initial position of the point particles of masses $m_1, m_2,\dots, m_n>0, n\ge 2$, on the manifold ${\mathbb S}^3_\kappa$, i.e.\ for $\kappa>0$, where ${\bf q}_i^0=(w_i^0, x_i^0, y_i^0, z_i^0), \ i=1,2,\dots,n$. Then a solution of the form ${\bf q}=({\mathcal A}[{\bf q}_1^0]^T, {\mathcal A}[{\bf q}_2^0]^T, \dots, {\mathcal A}[{\bf q}_n^0]^T)$ of system \eqref{second}, with 
\begin{equation}
{\mathcal A}(t)=\left( \begin{array}{cccc}
\cos\alpha t & -\sin\alpha t & 0 & 0 \\
\sin\alpha t & \cos\alpha t & 0 & 0 \\
0 & 0 & 1 & 0\\
0 & 0 & 0 & 1 \end{array} \right),
\end{equation}
where $\alpha\ne 0$ denotes the frequency, is called a (simply rotating) $\kappa$-positive elliptic relative equilibrium.
\label{elliptic-positive}
\end{definition}

\begin{remark}
In $\mathcal A$, the elements involving trigonometric functions could well be in the lower right corner instead of the upper left corner of the matrix, but the behaviour of the bodies would be similar, so we will always use the above form of the matrix ${\mathcal A}$.
\end{remark}

If $r_i=:\sqrt{(w_i^0)^2+(x_i^0)^2}$, we can find constants $a_i\in\mathbb R, \ i=1,2,\dots,n$, such that $w_i^0 = r_i\cos a_i,\ x_i^0=r_i\sin a_i, \ i=1,2,\dots,n$.  Then 
\begin{equation*}
{\mathcal A}(t)[{\bf q}_i^0]^T=\left( \begin{array}{c}
w_i^0\cos\alpha t - x_i^0\sin\alpha t\\
w_i^0\sin\alpha t + x_i^0\cos\alpha t \\
y_i^0\\
z_i^0 \end{array} \right)=
\end{equation*}
\begin{equation*}
\left( \begin{array}{c}
r_i\cos a_i\cos\alpha t - r_i\sin a_i\sin\alpha t\\
r_i\cos a_i\sin\alpha t + r_i\sin a_i\cos\alpha t \\
y_i^0\\
z_i^0 \end{array} \right)=
\left( \begin{array}{c}
r_i\cos(\alpha t+ a_i)\\
r_i\sin(\alpha t+a_i) \\
y_i^0\\
z_i^0\end{array} \right),  
\end{equation*}
$i=1,2,\dots,n.$

\subsection{Definition of $\kappa$-positive elliptic-elliptic relative equilibria}

The second kind of relative equilibria we introduce are inspired by the $\kappa$-positive elliptic-elliptic rotations of ${\mathbb S}_\kappa^3$.

\begin{definition}{\bf ($\kappa$-positive elliptic-elliptic relative equilibria)}
Let ${\bf q}^0=({\bf q}_1^0, {\bf q}_2^0,\dots, {\bf q}_n^0)$ be a nonsingular initial position of the bodies of masses $m_1, m_2,\dots, m_n>0, n\ge 2$, on the manifold ${\mathbb S}^3_\kappa$, i.e.\ for $\kappa>0$, where ${\bf q}_i^0=(w_i^0, x_i^0, y_i^0, z_i^0), \ i=1,2,\dots,n$. Then a solution of the form ${\bf q}=({\mathcal B}[{\bf q}_1^0]^T, {\mathcal B}[{\bf q}_2^0]^T, \dots, {\mathcal B}[{\bf q}_n^0]^T)$ of system \eqref{second}, with 
\begin{equation}
{\mathcal B}(t)=\left( \begin{array}{cccc}
\cos\alpha t & -\sin\alpha t & 0 & 0 \\
\sin\alpha t & \cos\alpha t & 0 & 0 \\
0 & 0 & \cos\beta t & -\sin\beta t\\
0 & 0 & \sin\beta t & \cos\beta t \end{array} \right),
\end{equation}
where $\alpha,\beta\ne 0$ are the frequencies, is called a (doubly rotating) $\kappa$-positive elliptic-elliptic relative equilibrium. 
\label{elliptic-elliptic}
\end{definition}

If $r_i=:\sqrt{(w_i^0)^2+(x_i^0)^2},\ \rho_i:=\sqrt{(y_i^0)^2+(z_i^0)^2}$, we can find constants $a_i, b_i\in\mathbb R, \ i=1,2,\dots,n$, such that $w_i^0 = r_i\cos a_i, x_i^0=r_i\sin a_i, y_i^0=\rho_i\cos b_i$, and $z_i^0=\rho_i\sin b_i, \ i=1,2,\dots,n$.  Then 
\begin{equation*}
{\mathcal B}(t)[{\bf q}_i^0]^T=\left( \begin{array}{c}
w_i^0\cos\alpha t - x_i^0\sin\alpha t\\
w_i^0\sin\alpha t + x_i^0\cos\alpha t \\
y_i^0\cos\beta t  - z_i^0\sin\beta t\\
y_i^0\sin\beta t + z_i^0\cos\beta t \end{array} \right)=
\end{equation*}
\begin{equation*}
\left( \begin{array}{c}
r_i\cos a_i\cos\alpha t - r_i\sin a_i\sin\alpha t\\
r_i\cos a_i\sin\alpha t + r_i\sin a_i\cos\alpha t \\
\rho_i\cos b_i\cos\beta t  - \rho_i\sin b_i\sin\beta t\\
\rho_i\cos b_i\sin\beta t + \rho_i\sin b_i\cos\beta t \end{array} \right)=
\left( \begin{array}{c}
r_i\cos(\alpha t+ a_i)\\
r_i\sin(\alpha t+a_i) \\
\rho_i\cos(\beta t+b_i)\\
\rho_i\sin(\beta t+b_i)\end{array} \right),  
\end{equation*}
$i=1,2,\dots,n.$

\subsection{Definition of $\kappa$-negative elliptic relative equilibria}

The third kind of relative equilibria we introduce here are inspired by the $\kappa$-negative elliptic rotations of ${\mathbb H}_\kappa^3$.

\begin{definition}{\bf ($\kappa$-negative elliptic relative equilibria)}
Let ${\bf q}^0=({\bf q}_1^0, {\bf q}_2^0,\dots, {\bf q}_n^0)$ be a nonsingular initial position  of the point particles of masses $m_1, m_2,\dots, m_n>0, n\ge 2,$ in ${\mathbb H}^3_\kappa$, i.e.\ for $\kappa<0$, where ${\bf q}_i^0=(w_i^0, x_i^0, y_i^0, z_i^0)$,\ $i=1,2,\dots,n$. Then a solution of the form ${\bf q}=({\mathcal C}[{\bf q}_1^0]^T, {\mathcal C}[{\bf q}_2^0]^T, \dots, {\mathcal C}[{\bf q}_n^0]^T)$  of system \eqref{second}, with
 \begin{equation}
{\mathcal C}(t)=\left( \begin{array}{cccc}
\cos\alpha t & -\sin\alpha t & 0 & 0 \\
\sin\alpha t & \cos\alpha t & 0 & 0 \\
0 & 0 & 1 & 0\\
0 & 0 & 0 & 1 \end{array} \right),
\end{equation}
where $\alpha\ne 0$ is the frequency, is called a (simply rotating) $\kappa$-negative elliptic relative equilibrium.
\label{elliptic-negative}
\end{definition}

If $r_i=:\sqrt{(w_i^0)^2+(x_i^0)^2}$, we can find $a_i\in\mathbb R, \ i=1,2,\dots,n,$ such that $w_i^0 = r_i\cos a_i, x_i^0=r_i\sin a_i, \ i=1,2,\dots,n$, so 
\begin{equation}
{\mathcal C}(t)[{\bf q}_i^0]^T=\left( \begin{array}{c}
w_i^0\cos\alpha t - x_i^0\sin\alpha t\\
w_i^0\sin\alpha t + x_i^0\cos\alpha t \\
y_i^0\\
z_i^0 \end{array} \right)=   
\end{equation}
\begin{equation*}
\left( \begin{array}{c}
r_i\cos a_i\cos\alpha t - r_i\sin a_i\sin\alpha t\\
r_i\cos a_i\sin\alpha t + r_i\sin a_i\cos\alpha t \\
y_i^0\\
z_i^0 \end{array} \right)=
\left( \begin{array}{c}
r_i\cos(\alpha t+ a_i)\\
r_i\sin(\alpha t+a_i) \\
y_i^0\\
z_i^0 \end{array} \right),
\end{equation*}
$i=1,2,\dots,n.$

\subsection{Definition of $\kappa$-negative hyperbolic relative equilibria}

The fourth kind of relative equilibria we introduce here are inspired by the $\kappa$-negative hyperbolic rotations of ${\mathbb H}_\kappa^3$.

\begin{definition}{\bf ($\kappa$-negative hyperbolic relative equilibria)}
Let ${\bf q}^0=({\bf q}_1^0, {\bf q}_2^0,\dots, {\bf q}_n^0)$ be a nonsingular initial position of the bodies of masses $m_1, m_2,\dots, m_n>0, n\ge 2,$ in ${\mathbb H}^3_\kappa$, i.e.\ for $\kappa<0$, where the initial positions are ${\bf q}_i^0=(w_i^0, x_i^0, y_i^0, z_i^0), \ i=1,2,\dots,n$. Then a solution of system \eqref{second} of the form ${\bf q}=({\mathcal D}[{\bf q}_1^0]^T, {\mathcal D}[{\bf q}_2^0]^T, \dots, {\mathcal D}[{\bf q}_n^0]^T)$, with 
\begin{equation}
{\mathcal D}(t)=\left( \begin{array}{cccc}
1 & 0 & 0 & 0 \\
1 & 0 & 0 & 0 \\
0 & 0 & \cosh \beta t & \sinh \beta t\\
0 & 0 & \sinh \beta t & \cosh \beta t \end{array} \right),
\end{equation}
where $\beta\ne 0$ denotes the frequency, is called a (simply rotating) $\kappa$-negative hyperbolic relative equilibrium.
\label{hyperbolic}
\end{definition}

If $\eta_i:=\sqrt{(z_i^0)^2-(y_i^0)^2}$, we can find constants $b_i\in\mathbb R, \ i=1,2,\dots,n$, such that  $y_i^0=\eta_i\sinh b_i$ and $z_i^0=\eta_i\cosh b_i, \ i=1,2,\dots,n$.  Then 
\begin{equation*}
{\mathcal D}(t)[{\bf q}_i^0]^T=\left( \begin{array}{c}
w_i^0\\
x_i^0\\
y_i^0\cosh bt  + z_i^0\sinh bt\\
y_i^0\sinh bt + z_i^0\cosh bt \end{array} \right)=
\end{equation*}
\begin{equation*}
\left( \begin{array}{c}
w_i^0\\
x_i^0 \\
\eta_i\sinh b_i\cosh bt + \eta_i\cosh b_i\sinh bt\\
\eta_i\sinh b_i\sinh bt + \eta_i\cosh b_i\cosh bt \\
\end{array} \right)
=\left( \begin{array}{c}
w_i^0\\
x_i^0 \\
\eta_i\sinh(bt+ b_i)\\
\eta_i\cosh(bt+b_i) \\
\end{array} \right),
\end{equation*} 
$ i=1,2,\dots,n.$

\subsection{Definition of $\kappa$-negative elliptic-hyperbolic relative equilibria}

The fifth kind of relative equilibria we introduce here are inspired the $\kappa$-negative elliptic-hyperbolic rotations of ${\mathbb H}_\kappa^3$.

\begin{definition}{\bf ($\kappa$-negative elliptic-hyperbolic relative equilibria)}
Let ${\bf q}^0=({\bf q}_1^0, {\bf q}_2^0,\dots, {\bf q}_n^0)$ be a nonsingular initial position of the point particles of masses $m_1, m_2,\dots, m_n>0, n\ge 2,$ in ${\mathbb H}^3_\kappa$, i.e.\ for $\kappa<0$, where ${\bf q}_i^0=(w_i^0, x_i^0, y_i^0, z_i^0), \ i=1,2,\dots,n$. Then a solution of system \eqref{second} of the form ${\bf q}=({\mathcal E}[{\bf q}_1^0]^T, {\mathcal E}[{\bf q}_2^0]^T, \dots, {\mathcal E}[{\bf q}_n^0]^T)$, with 
\begin{equation}
{\mathcal E}(t)=\left( \begin{array}{cccc}
\cos\alpha t & -\sin\alpha t & 0 & 0 \\
\sin\alpha t & \cos\alpha t & 0 & 0 \\
0 & 0 & \cosh \beta t & \sinh \beta t\\
0 & 0 & \sinh \beta t & \cosh \beta t \end{array} \right),
\end{equation}
where $\alpha, \beta \ne 0$ denote the frequencies, is called a (doubly rotating) $\kappa$-negative elliptic-hyperbolic relative equilibrium.
\label{elliptic-hyperbolic}
\end{definition}

If $r_i:=\sqrt{(w_i^0)^2+(x_i^0)^2},\ \eta_i:=\sqrt{(z_i^0)^2-(y_i^0)^2}$, we can find constants $a_i, b_i\in\mathbb R, \ i=1,2,\dots,n$, such that $w_i^0 = r_i\cos a_i, x_i^0=r_i\sin a_i, y_i^0=\eta_i\sinh b_i$, and $z_i^0=\eta_i\cosh b_i, \ i=1,2,\dots,n$.  Then 
\begin{equation*}
{\mathcal E}(t)[{\bf q}_i^0]^T=\left( \begin{array}{c}
w_i^0\cos\alpha t - x_i^0\sin\alpha t\\
w_i^0\sin\alpha t + x_i^0\cos\alpha t \\
y_i^0\cosh \beta t  + z_i^0\sinh \beta t\\
y_i^0\sinh \beta t + z_i^0\cosh \beta t \end{array} \right)=
\end{equation*}
\begin{equation*}
\left( \begin{array}{c}
r_i\cos a_i\cos\alpha t - r_i\sin a_i\sin\alpha t\\
r_i\cos a_i\sin\alpha t + r_i\sin a_i\cos\alpha t \\
\eta_i\sinh \beta_i\cosh \beta t + \eta_i\cosh \beta_i\sinh \beta t\\
\eta_i\sinh \beta_i\sinh \beta t + \eta_i\cosh \beta_i\cosh \beta t \\
\end{array} \right)
=\left( \begin{array}{c}
r_i\cos(\alpha t+ a_i)\\
r_i\sin(\alpha t+a_i) \\
\eta_i\sinh(\beta t+ b_i)\\
\eta_i\cosh(\beta t+b_i) \\
\end{array} \right),
\end{equation*} 
$ i=1,2,\dots,n.$

\subsection{Definition of $\kappa$-negative parabolic relative equilibria}

The sixth class of relative equilibria we introduce here are inspired by the $\kappa$-negative parabolic rotations of ${\mathbb H}_\kappa^3$.

\begin{definition}{\bf ($\kappa$-negative parabolic relative equilibria)}
Consider a nonsingular initial position  ${\bf q}^0=({\bf q}_1^0, {\bf q}_2^0,\dots, {\bf q}_n^0)$ of the point particles of masses $m_1, m_2,\dots, m_n>0, n\ge 2,$ on the manifold ${\mathbb H}^3_\kappa$, i.e.\ for $\kappa<0$, where ${\bf q}_i^0=(w_i^0, x_i^0, y_i^0, z_i^0),\  i=1,2,\dots,n$. Then a solution  of the form
${\bf q}=({\mathcal F}[{\bf q}_1^0]^T, {\mathcal F}[{\bf q}_2^0]^T, \dots, {\mathcal F}[{\bf q}_n^0]^T)$ of system \eqref{second}, with 
\begin{equation}
{\mathcal F}(t)=\left( \begin{array}{cccc}
1 & 0 & 0 & 0 \\
0 & 1 & -t & t \\
0 & t & 1-t^2/2 & t^2/2 \\
0 & t & -t^2/2 & 1+t^2/2 \end{array} \right),
\end{equation}
is called a (simply rotating) $\kappa$-negative parabolic relative equilibrium.
\label{parabolic}
\end{definition}

For simplicity, we denote $\alpha_i:=w_i^0, \beta_i:=x_i^0, \gamma_i:=y_i^0,
\delta_i:=z_i^0,\ i=1,2,\dots,n$. Then parabolic relative equilibria take the form
\begin{equation*}
{\mathcal F}(t)[{\bf q}_i^0]^T=\left( \begin{array}{c}
w_i^0\\
x_i^0 - y_i^0t + z_i^0t \\
x_i^0t + y_i^0(1-t^2/2) + z_i^0t^2/2\\
x_i^0t -y_i^0t^2/2 + z_i^0(1+t^2/2) \end{array} \right)
\end{equation*}
\begin{equation*}
=\left( \begin{array}{c}
\alpha_i \\
\beta_i+(\delta_i-\gamma_i)t\\
\gamma_i+\beta_i t+(\delta_i-\gamma_i)t^2/2\\
\delta_i+\beta_i t+(\delta_i-\gamma_i)t^2/2,\\ \end{array} \right),\  i=1,2,\dots,n.
\end{equation*}

\subsection{Formal expressions of the relative equilibria}
To summarize the previous findings, we can represent the above 6 types of relative equilibria of the 3-dimensional curved $n$-body problem in the form
\begin{equation*}
{\bf q}=({\bf q}_1, {\bf q}_2,\dots, {\bf q}_n),\ \ {\bf q}_i=(w_i, x_i, y_i, z_i),\ i=1,2,\dots,n,
\end{equation*}
\begin{equation} 
[\kappa>0, elliptic]:\ \ \
\begin{cases}
w_i(t)=r_i\cos(\alpha t+a_i)\\  
x_i(t)=r_i\sin(\alpha t+a_i)\\
y_i(t)=y_i\ {\rm(constant)}\\  
z_i(t)=z_i\ {\rm(constant)},\\
\end{cases}
\label{elliptic1-summarized}
\end{equation}
with $w_i^2+x_i^2=r_i^2,\ r_i^2+y_i^2+z_i^2=\kappa^{-1},
\ i=1,2,\dots,n$;
\begin{equation} 
[\kappa>0, elliptic{\rm-}elliptic]:\ \ \
\begin{cases}
w_i(t)=r_i\cos(\alpha t+a_i)\\  
x_i(t)=r_i\sin(\alpha t+a_i)\\
y_i(t)=\rho_i\cos(\beta t+b_i)\\  
z_i(t)=\rho_i\sin(\beta t+b_i),\\
\end{cases}
\label{elliptic-elliptic-summarized}
\end{equation}
with $w_i^2+x_i^2=r_i^2,\ y_i^2+z_i^2=\rho_i^2,$ so $r_i^2+\rho_i^2=\kappa^{-1},
\ i=1,2,\dots,n$;
\begin{equation} 
[\kappa<0, elliptic]:\ \ \ 
\begin{cases}
w_i(t)=r_i\cos(\alpha t+a_i)\\
x_i(t)=r_i\sin(\alpha t+a_i)\\
y_i(t)=y_i\ {\rm (constant)}\\
z_i(t)=z_i\ {\rm (constant)},\\
\end{cases}
\label{elliptic2-summarized}
\end{equation}
with $w_i^2+x_i^2=r_i^2, \ r_i^2+y_i^2-z_i^2=\kappa^{-1}, \ i=1,2,\dots,n$;
\begin{equation} 
[\kappa<0, hyperbolic]:\ \ \
\begin{cases}
w_i(t)=w_i\ {\rm(constant)}\\ 
x_i(t)=x_i\ {\rm(constant)}\\
y_i(t)=\eta_i\sinh(\beta t+b_i)\\
z_i(t)=\eta_i\cosh(\beta t+b_i),\\
\end{cases}
\label{hyperbolic-summarized}
\end{equation}
with $y_i^2-z_i^2=-\eta_i^2,\ w_i^2+x_i^2-\eta_i^2=\kappa^{-1},
\ i=1,2,\dots,n$;
\begin{equation} 
[\kappa<0, elliptic{\rm-}hyperbolic]:\ \ \
\begin{cases}
w_i(t)=r_i\cos(\alpha t+a_i)\\ 
x_i(t)=r_i\sin(\alpha t+a_i)\\
y_i(t)=\eta_i\sinh(\beta t+b_i)\\
z_i(t)=\eta_i\cosh(\beta t+b_i),\\
\end{cases}
\label{elliptic-hyperbolic-summarized}
\end{equation}
with $w_i^2+x_i^2=r_i^2,\ y_i^2-z_i^2=-\eta_i^2,$ so $r_i^2-\eta_i^2=\kappa^{-1},
\ i=1,2,\dots,n$;
\begin{equation} 
[\kappa<0, parabolic]:\ \ \
\begin{cases}
w_i(t)= \alpha_i \ {\rm (constant)}\\
x_i(t)=\beta_i+(\delta_i-\gamma_i)t\\
y_i(t)=\gamma_i+\beta_i t+(\delta_i-\gamma_i)t^2/2\\
z_i(t)=\delta_i+\beta_i t+(\delta_i-\gamma_i)t^2/2,\\
\end{cases}
\label{para}
\end{equation}
with $\alpha_i^2+\beta_i^2+\gamma_i^2-\delta_i^2=\kappa^{-1}, \ i=1,2,\dots,n$.

\section{Fixed Points}\label{fixedpoints}

In this section we introduce the concept of fixed-point solution of the equations of motion, show that fixed points exist in $\mathbb S_\kappa^3$, provide a couple of examples, and finally prove that they cannot show up in $\mathbb H_\kappa^3$ and in hemispheres of $\mathbb S_\kappa^3$.

\subsection{Existence of fixed points in ${\mathbb S}_\kappa^3$}
Although the goal of this paper is to study relative equilibria of the curved $n$-body problem in 3-dimensional space, some of these orbits can be generated from fixed-point configurations by imposing on the initial positions of the bodies suitable nonzero initial velocities. It is therefore necessary to discuss fixed points as well. We start 
with their definition.

\begin{definition} 
A solution of system \eqref{Ham} is called a fixed point if it is a zero of the
vector field, i.e.\ ${\bf p}_i(t)={\nabla}_{{\bf q}_i}U_\kappa({\bf q}(t))={\bf 0}$ for all $t\in\mathbb R$, $i=1,2,\dots,n$. 
\end{definition}

In \cite{Diacu4}, \cite{Diacu1}, and \cite{Diacu001}, we showed that fixed points exist in ${\mathbb S}_\kappa^2$, but they don't exist in ${\mathbb H}_\kappa^2$. Examples of fixed points are the equilateral triangle of equal masses lying on any great circle of ${\mathbb S}_\kappa^2$ in the curved 3-body problem and the regular tetrahedron of equal masses inscribed in ${\mathbb S}_\kappa^2$ in the 4-body case. There are also examples of fixed points of non-equal masses. We showed in \cite{DiacuPol} that, for any acute triangle inscribed in a great circle of the sphere $\mathbb S_\kappa^2$, there exist masses $m_1, m_2, m_3>0$ that can be placed at the vertices of the triangle such that they form a fixed point, and therefore can generate relative equilibria in the curved 3-body problem. All these examples can be transferred to ${\mathbb S}_\kappa^3$.

\subsection{Two examples of fixed points specific to ${\mathbb S}_\kappa^3$}\label{specific-fixed}
We can construct fixed points of ${\mathbb S}_\kappa^3$ for which none of its great spheres contains them. A first simple example occurs in the 6-body problem if we
take 6 bodies of equal positive masses, place 3 of them, with zero initial velocities, at the vertices of an equilateral triangle inscribed in a great circle of a great sphere, and
place the other 3 bodies, with zero initial velocities as well, at the vertices of an
equilateral triangle inscribed in a complementary great circle (see Definition \ref{complementary}) of another great sphere. Some straightforward computations show
that 6 bodies of masses $m_1=m_2=m_3=m_4=m_5=m_6=:m>0$, with zero initial velocities and initial conditions given, for instance, by
\begin{align*}
w_1&=\kappa^{-1/2},& x_1&=0,& y_1&=0,& z_1&=0,\displaybreak[0]\\
w_2&=-\frac{\kappa^{-1/2}}{2},& x_2&=\frac{\sqrt{3}\kappa^{-1/2}}{2},& y_2&=0,& z_2&=0,\displaybreak[0]\\
w_3&=-\frac{\kappa^{-1/2}}{2},& x_3&=-\frac{\sqrt{3}\kappa^{-1/2}}{2},& y_3&=0,& z_3&=0,\displaybreak[0]\\
w_4&=0,& x_4&=0,& y_4&=\kappa^{-1/2},& z_4&=0,\displaybreak[0]\\
w_5&=0,& x_5&=0,& y_5&=-\frac{\kappa^{-1/2}}{2},& z_5&=\frac{\sqrt{3}\kappa^{-1/2}}{2},\displaybreak[0]\\
w_6&=0,& x_6&=0,& y_6&=-\frac{\kappa^{-1/2}}{2},& z_6&=-\frac{\sqrt{3}\kappa^{-1/2}}{2},
\end{align*}
form a fixed point.

The second example is inspired from the theory of regular polytopes, \cite{Coxeter}, \cite{Coxeter1}. The simplest regular polytope in $\mathbb R^4$ is the pentatope (also called 5-cell, 4-simplex, pentachrone, pentahedroid, or hyperpiramid). The pentatope has Schl\"afli symbol $\{3,3,3\}$, which translates into: 3 regular polyhedra that  have 3 regular polygons of 3 edges at every vertex (i.e.\ 3 regular tetrahedra) are attached to each of the pentatope's edges. (From the left to the right, the numbers in the Schl\"afli symbol are in the order we described them.) 

A different way to understand the pentatope is to think of it as the generalization to $\mathbb R^4$ of the equilateral triangle of $\mathbb R^2$ or of the regular tetrahedron of $\mathbb R^3$. Then the pentatope can be constructed by adding to the regular tetrahedron a fifth vertex in $\mathbb R^4$  that connects the other four vertices with edges of the same length as those of the tetrahedron. Consequently the pentatope can be inscribed in a sphere ${\mathbb S}_\kappa^3$, in which it has no antipodal vertices, so there is no danger of encoutering singular configurations for the fixed point we want to construct. Specifically, the coordinates of the 5 vertices of a pentatope inscribed in the sphere ${\mathbb S}_\kappa^3$ can be taken, for example, as
\begin{align*}
w_1&=\kappa^{-1/2},& x_1&=0,& y_1&=0,& z_1&=0,\displaybreak[0]\\
w_2&=-\frac{\kappa^{-1/2}}{4},& x_2&=\frac{\sqrt{15}\kappa^{-1/2}}{4},& y_2&=0,& z_2&=0,\displaybreak[0]\\
w_3&=-\frac{\kappa^{-1/2}}{4},& x_3&=-\frac{\sqrt{5}\kappa^{-1/2}}{4\sqrt{3}},& y_3&=\frac{\sqrt{5}\kappa^{-1/2}}{\sqrt{6}},& z_3&=0,\displaybreak[0]\\
w_4&=-\frac{\kappa^{-1/2}}{4},& x_4&=-\frac{\sqrt{5}\kappa^{-1/2}}{4\sqrt{3}},& y_4&=-\frac{\sqrt{5}\kappa^{-1/2}}{2\sqrt{6}},& z_4&=\frac{\sqrt{5}\kappa^{-1/2}}{2\sqrt{2}},\displaybreak[0]\\
w_5&=-\frac{\kappa^{-1/2}}{4},& x_5&=-\frac{\sqrt{5}\kappa^{-1/2}}{4\sqrt{3}},& y_5&=-\frac{\sqrt{5}\kappa^{-1/2}}{2\sqrt{6}},& z_5&=-\frac{\sqrt{5}\kappa^{-1/2}}{2\sqrt{2}}.
\end{align*}
 Straightforward computations show that the distance from the origin of the coordinate system to each of the 5 vertices is $\kappa^{-1/2}$ and that, for equal masses $m_1=m_2=m_3=m_4=m_5=:m>0$, this configuration produces a fixed-point solution of system \eqref{second} for $\kappa>0$. Like in the previous example, this fixed point lying at the vertices of the pentatope is specific to ${\mathbb S}_\kappa^3$ in the sense that there is no 2-dimensional sphere that contains it.
 
It is natural to ask whether other convex regular polytopes of $\mathbb R^4$ can form fixed points in ${\mathbb S}_\kappa^3$ if we place equal masses at their vertices. Apart from the pentatope, there are five other such geometrical objects: the tesseract (also called 8-cell, hypercube, or 4-cube, with 16 vertices), the orthoplex (also called 16-cell or hyperoctahedron, with 8 vertices), the octaplex (also called 24-cell or polyoctahedron, with 24 vertices), the dodecaplex (also called 120-cell, hyperdodecahedron, or polydodecahedron, with 600 vertices), and
the tetraplex (also called 600-cell, hypericosahedron, or polytetrahedron, with 120 vertices). All these polytopes, however, are centrally symmetric, so they have antipodal vertices. Therefore, if we place bodies of equal masses at their vertices, we encounter singularities. Consequently the only convex regular polytope of $\mathbb R^4$ that can form a fixed point if we place equal masses at its vertices is the pentatope.

\subsection{Nonexistence of fixed points in ${\mathbb H}_\kappa^3$ and hemispheres of ${\mathbb S}_\kappa^3$}
We will show further that there are no fixed points in ${\mathbb H}^3_\kappa$ or in any hemisphere of ${\mathbb S}_\kappa^3$. In the latter case, for fixed points not to exist it is necessary that at least one body is not on the boundary of the hemisphere.

\begin{proposition}{\bf (No fixed points in ${\mathbb H}_\kappa^3$)}
For $\kappa<0$, there are no masses $m_1,m_2,\dots,m_n>0$, $n\ge 2$, that can form a fixed point.
\label{no-fixed-hyp}
\end{proposition}
\begin{proof}
Consider $n$ bodies of masses $m_1,m_2,\dots,m_n>0$, $n\ge 2$, lying on ${\mathbb H}_\kappa^3$, in a nonsingular configuration (i.e.\ without collisions) and with zero initial velocities. Then one or more bodies, say, $m_1, m_2, \dots, m_k$ with $k\le n$,
have the largest $z$ coordinate. Consequently each of the bodies $m_{k+1},\dots, m_{n}$ will attract each of the bodies $m_1, m_2, \dots, m_k$ along a geodesic hyperbola towards lowering the $z$ coordinate of the latter. For any 2 bodies with the same largest $z$ coordinate, the segment of hyperbola connecting them has points with lower $z$ coordinates. Therefore these 2 bodies attract each other towards lowering their $z$ coordinates as well. So each of the bodies $m_1, m_2, \dots, m_k$ will move towards lowering their $z$ coordinate, therefore the initial configuration of the bodies is not fixed. 
\end{proof}

\begin{proposition}{\bf (No fixed points in hemispheres of ${\mathbb S}_\kappa^3$)}\label{hemi}
For $\kappa>0$, there are no masses $m_1, m_2,\dots, m_n>0, n\ge 2$, that can form a fixed point in any closed hemisphere of ${\mathbb S}_\kappa^3$ (i.e.\ a hemisphere that contains its boundary), as long as at least one body doesn't lie on the boundary.
\end{proposition}
\begin{proof}
The idea of the proof is similar to the idea of the proof we gave for Proposition \ref{no-fixed-hyp}. Let us assume, without loss of generality, that the bodies are in the hemisphere $z\le 0$ and they form a nonsingular initial configuration (i.e.\ without collisions or antipodal positions), with at least one of the bodies not on the boundary $z=0$, and with zero initial velocities. Then one or more bodies, say, $m_1, m_2, \dots, m_k,$ with $k<n$ (a strict inequality is essential to the proof), have the largest $z$ coordinate, which can be at most 0. Consequently the bodies $m_{k+1},\dots, m_n$ have lower $z$ coordinates. Each of the bodies $m_{k+1},\dots, m_n$ attract each of the bodies $m_1, m_2, \dots, m_k$ along a geodesic arc of a great circle towards lowering the $z$ coordinate of the latter.

The attraction between any 2 bodies among $m_1, m_2, \dots, m_k$ is either towards lowering each other's $z$ coordinate, when $z<0$ or along the geodesic $z=0$, when they are on that geodesic. In both cases, however, composing all the forces that act on each of the bodies $m_1, m_2, \dots, m_k$ will make them move towards a lower $z$ coordinate, which means that the initial configuration is not fixed.
This remark completes the proof.
\end{proof}

\newpage

\part{\rm CRITERIA AND QUALITATIVE BEHAVIOUR}

\section{Existence criteria for the relative equilibria}

In this section we establish criteria for the existence of $\kappa$-positive elliptic and elliptic-elliptic as well as $\kappa$-negative elliptic, hyperbolic, and elliptic-hyperbolic relative equilibria. These criteria will be employed in later sections to obtain concrete examples of such orbits. We close this section by showing that $\kappa$-negative parabolic relative equilibria do not exist in the curved $n$-body problem.

\subsection{Criteria for $\kappa$-positive elliptic relative equilibria}

We provide now a criterion for the existence of simply rotating $\kappa$-positive
elliptic orbits and then prove a corollary that shows under what conditions such
solutions can be generated from fixed-point configurations.

\begin{criterion}{\bf ($\kappa$-positive elliptic relative equilibria)}
Consider the point particles of masses $m_1,m_2,\dots,m_n>0,\ n\ge 2$, in ${\mathbb S}_\kappa^3$, i.e.\ for $\kappa>0$. Then system \eqref{second} admits a solution of
the form 
$$
{\bf q}=({\bf q}_1,{\bf q}_2,\dots, {\bf q}_n),\ {\bf q}_i=(w_i,x_i,y_i,z_i),\ i=1,2,\dots,n,
$$
$$
w_i(t)=r_i\cos(\alpha t+a_i), x_i(t)=r_i\sin(\alpha t+a_i), y_i(t)=y_i,
z_i(t)=z_i,
$$
with $w_i^2+x_i^2=r_i^2, r_i^2+y_i^2+z_i^2=\kappa^{-1}$, and $y_i,z_i$ constant, $i=1,2,\dots,n$,
i.e.\ a (simply rotating) $\kappa$-positive elliptic relative equilibrium, if and only if there are constants $r_i, a_i, y_i, z_i,\ i=1,2,\dots, n,$ and $\alpha\ne 0$, such that the following $4n$ conditions are satisfied:
\beq
\sum_{\stackrel{j=1}{j\ne i}}^n\frac{m_j\kappa^{3/2}(r_j\cos a_j-\kappa\nu_{ij}r_i\cos a_i)}{(1-\kappa^2\nu_{ij}^2)^{3/2}}
=(\kappa r_i^2-1)\alpha^2r_i\cos a_i,
\label{w0}
\eeq
\beq
\sum_{\stackrel{j=1}{j\ne i}}^n\frac{m_j\kappa^{3/2}(r_j\sin a_j-\kappa\nu_{ij}r_i\sin a_i)}{(1-\kappa^2\nu_{ij}^2)^{3/2}}
=(\kappa r_i^2-1)\alpha^2r_i\sin a_i, 
\label{x0}
\eeq
\beq
\sum_{\stackrel{j=1}{j\ne i}}^n\frac{m_j\kappa^{3/2}(y_j-\kappa\nu_{ij}y_i)}{(1-\kappa^2\nu_{ij}^2)^{3/2}}=\kappa\alpha^2r_i^2y_i,
\label{y0}
\eeq
\beq
\sum_{\stackrel{j=1}{j\ne i}}^n\frac{m_j\kappa^{3/2}(z_j-\kappa\nu_{ij}z_i)}{(1-\kappa^2\nu_{ij}^2)^{3/2}}=\kappa\alpha^2r_i^2z_i, 
\label{z0}
\eeq
$i=1,2,\dots,n$, where $\nu_{ij}=r_ir_j\cos(a_i-a_j)+y_iy_j+z_iz_j,\ i,j=1,2\dots,n,\ i\ne j$.
\label{cri-elliptic1}
\end{criterion}
\begin{proof}
Consider a candidate of a solution ${\bf q}$ as above for system \eqref{second}. Some straightforward computations show that
$$
\nu_{ij}:={\bf q}_i\odot{\bf q}_j=r_ir_j\cos(a_i-a_j)+y_iy_j+z_iz_j,\ i,j=1,2\dots,n,\ i\ne j,
$$
$$
\dot{\bf q}_i\odot\dot{\bf q}_i=\alpha^2r_i^2,\ i=1,2,\dots,n,
$$
$$
\ddot{w}_i=-\alpha^2r_i\cos(\alpha t+a_i), \ddot{x}_i=-\alpha^2r_i\sin(\alpha t+a_i),
$$
$$
\ddot{y}_i=\ddot{z}_i=0,\ i=1,2,\dots,n.
$$
Substituting the suggested solution and the above expressions into system \eqref{second}, for the $w$ coordinates we obtain conditions involving $\cos(\alpha t+a_i)$, whereas for the $x$ coordinates we obtain conditions involving $\sin(\alpha t+a_i)$. In the former case, using the fact that $\cos(\alpha t+a_i)=\cos\alpha t\cos a_i-\sin\alpha t\sin a_i$, we can split each equation in two, one involving $\cos\alpha t$ and the other $\sin\alpha t$ as factors. The same thing happens in the latter case if we use the formula $\sin(\alpha t+a_i)=\sin\alpha t\cos a_i+\cos\alpha t\sin a_i$. Each of these equations are satisfied if and only if conditions \eqref{w0} and \eqref{x0} take place. Conditions \eqref{y0} and \eqref{z0} follow directly from the
equations involving the coordinates $y$ and $z$. This remark completes the proof.
\end{proof}

\begin{criterion}{\bf ($\kappa$-positive elliptic relative equilibria generated from fixed-point configurations)}
Consider the point particles of masses $m_1,m_2,\dots,m_n>0, n\ge 2$,  in ${\mathbb S}_\kappa^3$, i.e.\ for $\kappa>0$. Then, for any $\alpha\ne 0$, system \eqref{second} admits a solution of
the form \eqref{elliptic1-summarized}:
$$
{\bf q}=({\bf q}_1,{\bf q}_2,\dots, {\bf q}_n), {\bf q}_i=(w_i,x_i,y_i,z_i),\ i=1,2,\dots,n,
$$
$$
w_i(t)=r_i\cos(\alpha t+a_i),\ x_i(t)=r_i\sin(\alpha t+a_i),\ y_i(t)=y_i,\ z_i(t)=z_i,
$$
with $w_i^2+x_i^2=r_i^2, r_i^2+y_i^2+z_i^2=\kappa^{-1}$, and $y_i,z_i$ constant, $i=1,2,\dots,n$, generated from a fixed point, i.e.\ a (simply rotating) $\kappa$-positive elliptic relative equilibrium generated from the same initial positions that would form a fixed point for zero initial velocities, if and only if there are constants $r_i, a_i, y_i, z_i,\ i=1,2,\dots, n,$ such that the following $4n$ conditions are satisfied: 
\beq
\sum_{\stackrel{j=1}{j\ne i}}^n\frac{m_j\kappa^{3/2}(r_j\cos a_j-\kappa\nu_{ij}r_i\cos a_i)}{(1-\kappa^2\nu_{ij}^2)^{3/2}}=0,
\label{w00}
\eeq
\beq
\sum_{\stackrel{j=1}{j\ne i}}^n\frac{m_j\kappa^{3/2}(r_j\sin a_j-\kappa\nu_{ij}r_i\sin a_i)}{(1-\kappa^2\nu_{ij}^2)^{3/2}}=0,
\label{x00}
\eeq
\beq
\sum_{\stackrel{j=1}{j\ne i}}^n\frac{m_j\kappa^{3/2}(y_j-\kappa\nu_{ij}y_i)}{(1-\kappa^2\nu_{ij}^2)^{3/2}}=0,
\label{y00}
\eeq
\beq
\sum_{\stackrel{j=1}{j\ne i}}^n\frac{m_j\kappa^{3/2}(z_j-\kappa\nu_{ij}z_i)}{(1-\kappa^2\nu_{ij}^2)^{3/2}}=0,
\label{z00}
\eeq
$i=1,2,\dots,n$, where $\nu_{ij}=r_ir_j\cos(a_i-a_j)+y_iy_j+z_iz_j,\ i,j=1,2,\dots,n,\ i\ne j$,
and one of the following two properties takes place:

(i) $r_i=\kappa^{-1/2}$ for all $i\in\{1,2\dots,n\}$,

(ii) there is a proper subset ${\mathcal I}\subset\{1,2,\dots,n\}$ such that
$r_i=0$ for all $i\in{\mathcal I}$ and $r_j=\kappa^{-1/2}$ for all $j\in\{1,2,\dots,n\}\setminus{\mathcal I}$.
\label{fixed-simply}
\end{criterion}
\begin{proof}
We are seeking a simply rotating elliptic relative equilibrium, as in Criterion \ref{cri-elliptic1}, that is valid for any $\alpha\ne 0$. But the solution is also generated from a fixed-point configuration, so the left hand sides of equations \eqref{w0}, \eqref{x0}, \eqref{y0}, and \eqref{z0} necessarily vanish, thus leading to
conditions \eqref{w00}, \eqref{x00}, \eqref{y00}, and \eqref{z00}. However, the right hand sides of equations \eqref{w0}, \eqref{x0}, \eqref{y0}, and \eqref{z0} must also vanish, so we have the $4n$ conditions:
$$
(\kappa r_i^2-1)\alpha^2r_i\cos a_i=0,\ i=1,2,\dots,n,
$$
$$
(\kappa r_i^2-1)\alpha^2r_i\sin a_i=0,\ i=1,2,\dots,n,
$$
$$
\kappa\alpha^2r_i^2y_i=0,\ i=1,2,\dots,n, 
$$
$$
\kappa\alpha^2r_i^2z_i=0,\ i=1,2,\dots,n.
$$
Since $\alpha\ne 0$ and there is no $\gamma\in\mathbb R$ such that the quantities $\sin\gamma$ and $\cos\gamma$ vanish simultaneously, the above $4n$ conditions are satisfied in each of the following cases:

(a) $r_i=0$ (consequently $w_i=x_i=0$ and $y_i^2+z_i^2=\kappa^{-1}$) for all $i\in\{1,2,\dots,n\}$,

(b) $r_i=\kappa^{-1/2}$ (consequently $w_i^2+x_i^2=\kappa^{-1}$ and $y_i=z_i=0$) for all $i\in\{1,2\dots,n\}$,

(c) there is a proper subset ${\mathcal I}\subset\{1,2,\dots,n\}$ such that
$r_i=0$ (consequently $y_i^2+z_i^2=\kappa^{-1}$) for all $i\in{\mathcal I}$
and $r_j=\kappa^{-1/2}$ (consequently $y_j=z_j=0$) for all $j\in\{1,2,\dots,n\}\setminus{\mathcal I}$.

In case (a), we recover the fixed point, so there is no rotation of any kind, therefore
this case does not lead to any simply rotating $\kappa$-positive elliptic relative
equilibrium. As we will see in Theorem \ref{theorem-fixed}, case (b), which corresponds to {\it (i)} in the above statement, and case (c), which corresponds to {\it (ii)}, lead to relative equilibria of this kind. This remark completes the proof.
\end{proof}

\subsection{Criteria for $\kappa$-positive elliptic-elliptic relative equilibria}
We can now provide a criterion for the existence of doubly rotating $\kappa$-positive
elliptic-elliptic relative equilibria and a criterion about how such orbits can be obtained from fixed-point configurations. 

\begin{criterion}{\bf ($\kappa$-positive elliptic-elliptic relative equilibria)}
Consider the bodies of masses $m_1,m_2,\dots,m_n>0, n\ge 2$, in ${\mathbb S}_\kappa^3$, i.e.\ for $\kappa>0$. Then system \eqref{second} admits a solution of
the form \eqref{elliptic-elliptic-summarized}:
$$
{\bf q}=({\bf q}_1,{\bf q}_2,\dots, {\bf q}_n), {\bf q}_i=(w_i,x_i,y_i,z_i),\ i=1,2,\dots,n,
$$
\begin{align*}
w_i(t)&=r_i\cos(\alpha t+a_i), & x_i(t)&=r_i\sin(\alpha t+a_i),\\
y_i(t)&=\rho_i\cos(\beta t+b_i), & z_i(t)&=\rho_i\sin(\beta t+b_i),
\end{align*}
with $w_i^2+x_i^2=r_i^2, y_i^2+z_i^2\rho^2, r_i^2+\rho_i^2=\kappa^{-1},\ i=1,2,\dots,n$,
i.e.\ a (doubly rotating) $\kappa$-positive elliptic-elliptic relative equilibrium, if and only if there are constants $r_i, \rho_i, a_i, b_i,\ i=1,2,\dots, n$, and $\alpha,\beta\ne 0$, such that the following $4n$ conditions are satisfied
\beq
\sum_{\stackrel{j=1}{j\ne i}}^n\frac{m_j\kappa^{3/2}(r_j\cos a_j-\kappa\omega_{ij}r_i\cos a_i)}{(1-\kappa^2\omega_{ij}^2)^{3/2}}=(\kappa\alpha^2r_i^2+\kappa\beta^2\rho_i^2-\alpha^2)r_i\cos a_i, 
\label{w1}
\eeq
\beq
\sum_{\stackrel{j=1}{j\ne i}}^n\frac{m_j\kappa^{3/2}(r_j\sin a_j-\kappa\omega_{ij}r_i\sin a_i)}
{(1-\kappa^2\omega_{ij}^2)^{3/2}}=(\kappa\alpha^2r_i^2+\kappa\beta^2\rho_i^2-\alpha^2)r_i\sin a_i, 
\label{x1}
\eeq
\beq
\sum_{\stackrel{j=1}{j\ne i}}^n\frac{m_j\kappa^{3/2}(\rho_j\cos b_j-\kappa\omega_{ij}\rho_i\cos b_i)}{(1-\kappa^2\omega_{ij}^2)^{3/2}}=(\kappa\alpha^2r_i^2+\kappa\beta^2\rho_i^2-\beta^2)\rho_i\cos b_i, 
\label{y1}
\eeq
\beq
\sum_{\stackrel{j=1}{j\ne i}}^n\frac{m_j\kappa^{3/2}(\rho_j\sin b_j-\kappa\omega_{ij}\rho_i\sin b_i)}
{(1-\kappa^2\omega_{ij}^2)^{3/2}}=(\kappa\alpha^2r_i^2+\kappa\beta^2\rho_i^2-\beta^2)\rho_i\sin b_i, 
\label{z1}
\eeq
$i=1,2,\dots,n$, where $\omega_{ij}=r_ir_j\cos(a_i-a_j)+\rho_i\rho_j\cos(b_i-b_j), \ i,j=1,2,\dots,n,\ i\ne j$, and $r_i^2+\rho_i^2=\kappa^{-1},\ i=1,2,\dots,n$. 
\end{criterion}
\begin{proof}
Consider a candidate ${\bf q}$ as above for a solution of system \eqref{second}. Some straightforward computations show that
$$
\omega_{ij}:={\bf q}_i\odot{\bf q}_j=r_ir_j\cos(a_i-a_j)+\rho_i\rho_j\cos(b_i-b_j),\ i,j=1,2\dots,n,\ i\ne j,
$$
$$
\dot{\bf q}_i\odot\dot{\bf q}_i=\alpha^2r_i^2+\beta^2\rho_i^2,\ i=1,2,\dots,n,
$$
$$
\ddot{w}_i=-\alpha^2r_i\cos(\alpha t+a_i), \ddot{x}_i=-\alpha^2r_i\sin(\alpha t+a_i),
$$
$$
\ddot{y}_i=-\beta^2\rho_i\cos(\beta t+b_i), \ddot{z}_i=-\beta^2\rho_i\sin(\beta t+b_i),\ i=1,2,\dots,n.
$$
Substituting $\bf q$ and the above expressions into system \eqref{second}, for the $w$ coordinates we obtain conditions involving $\cos(\alpha t+a_i)$, whereas for the $x$ coordinates we obtain conditions involving $\sin(\alpha t+a_i)$. In the former case, using the fact that $\cos(\alpha t+a_i)=\cos\alpha t\cos a_i-\sin\alpha t\sin a_i$, we can split each equation in two, one involving $\cos\alpha t$ and the other $\sin\alpha t$ as factors. The same thing happens in the latter case if we use the formula $\sin(\alpha t+a_i)=\sin\alpha t\cos a_i+\cos\alpha t\sin a_i$. Each of these equations are satisfied if and only if conditions \eqref{w1} and \eqref{x1} take place. 

For the $y$ coordinate, the substitution of the above solution leads to conditions involving $\cos(\beta+b_i)$, whereas for $z$ coordinate it leads to conditions involving $\sin(\beta+b_i)$. Then we proceed as we did for the $w$ and $x$ coordinates and obtain conditions 
\eqref{y1} and \eqref{z1}. This remark completes the proof.
\end{proof}

\begin{criterion}{\bf ($\kappa$-positive elliptic-elliptic relative equilibria generated from fixed-point configurations)}
Consider the point particles of masses $m_1,m_2,\dots,$ $m_n>0, n\ge 2$, in ${\mathbb S}_\kappa^3$, i.e.\ for $\kappa>0$. Then, for any $\alpha,\beta\ne 0$, system \eqref{second} admits a solution of
the form \eqref{elliptic-elliptic-summarized}:
$$
{\bf q}=({\bf q}_1,{\bf q}_2,\dots, {\bf q}_n), {\bf q}_i=(w_i,x_i,y_i,z_i),\ i=1,2,\dots,n,
$$
\begin{align*}
w_i(t)&=r_i\cos(\alpha t+a_i), & x_i(t)&=r_i\sin(\alpha t+a_i),\\
y_i(t)&=\rho_i\cos(\beta t+b_i), & z_i(t)&=\rho_i\sin(\beta t+b_i),
\end{align*}
with $w_i^2+x_i^2=r_i^2, y_i^2+z_i^2\rho^2, r_i^2+\rho_i^2=\kappa^{-1}, i\in\{1,2,\dots,n\}$, 
generated from a fixed-point configuration, i.e.\ a (doubly rotating) $\kappa$-positive elliptic-elliptic relative equilibrium generated from the same initial positions that would form a fixed point for zero initial velocities, if and only if there are constants $r_i, \rho_i, a_i, b_i,\ i=1,2,\dots,n$, such that the $4n$ relationships below are satisfied:
\beq
\sum_{\stackrel{j=1}{j\ne i}}^n\frac{m_j\kappa^{3/2}(r_j\cos a_j-\kappa\omega_{ij}r_i\cos a_i)}{(1-\kappa^2\omega_{ij}^2)^{3/2}}=0,
\label{w10}
\eeq
\beq
\sum_{\stackrel{j=1}{j\ne i}}^n\frac{m_j\kappa^{3/2}(r_j\sin a_j-\kappa\omega_{ij}r_i\sin a_i)}
{(1-\kappa^2\omega_{ij}^2)^{3/2}}=0,
\label{x10}
\eeq
\beq
\sum_{\stackrel{j=1}{j\ne i}}^n\frac{m_j\kappa^{3/2}(\rho_j\cos b_j-\kappa\omega_{ij}\rho_i\cos b_i)}{(1-\kappa^2\omega_{ij}^2)^{3/2}}=0,
\label{y10}
\eeq
\beq
\sum_{\stackrel{j=1}{j\ne i}}^n\frac{m_j\kappa^{3/2}(\rho_j\sin b_j-\kappa\omega_{ij}\rho_i\sin b_i)}
{(1-\kappa^2\omega_{ij}^2)^{3/2}}=0, 
\label{z10}
\eeq
$i=1,2,\dots,n$, where $\omega_{ij}=r_ir_j\cos(a_i-a_j)+\rho_i\rho_j\cos(b_i-b_j), \ i,j=1,2,\dots,n,\ i\ne j$, and $r_i^2+\rho_i^2=R^2=\kappa^{-1},\ i=1,2,\dots,n$,
and, additionally, one of the following properties takes place:

(i) there is a proper subset ${\mathcal J}\subset\{1,2,\dots,n\}$ such that
$r_i=0$ for all $i\in{\mathcal J}$ and $\rho_j=0$ for all $j\in\{1,2,\dots,n\}\setminus{\mathcal J}$,

(ii) the frequencies $\alpha,\beta\ne 0$ satisfy the condition $|\alpha|=|\beta|$.
\label{fixed-doubly}
\end{criterion}
\begin{proof}
A fixed-point configuration requires that the left hand sides of equations \eqref{w1}, \eqref{x1}, \eqref{y1}, and \eqref{z1} vanish, so we obtain the conditions \eqref{w10}, \eqref{x10}, \eqref{y10}, and \eqref{z10}. A relative equilibrium can be generated from a fixed-point configuration if and only if the right hand sides of \eqref{w1}, \eqref{x1}, \eqref{y1}, and \eqref{z1} vanish as well, i.e.
$$
(\kappa\alpha^2r_i^2+\kappa\beta^2\rho_i^2-\alpha^2)r_i\cos a_i=0, 
$$
$$
(\kappa\alpha^2r_i^2+\kappa\beta^2\rho_i^2-\alpha^2)r_i\sin a_i=0,
$$
$$
(\kappa\alpha^2r_i^2+\kappa\beta^2\rho_i^2-\beta^2)\rho_i\cos b_i=0,
$$
$$
(\kappa\alpha^2r_i^2+\kappa\beta^2\rho_i^2-\beta^2)\rho_i\sin b_i=0,
$$
where $r_i^2+\rho_i^2=\kappa^{-1},\ i=1,2,\dots,n$. Since there is no $\gamma\in\mathbb R$ such that $\sin\gamma$ and $\cos\gamma$ vanish simultaneously, the above expressions are zero in one of the following circumstances:

(a) $r_i=0$, and consequently $\rho_i=\kappa^{-1/2}$, for all $i\in\{1,2,\dots,n\}$,

(b) $\rho_i=0$, and consequently $r_i=\kappa^{-1/2}$, for all $i\in\{1,2,\dots,n\}$,

(c) there is a proper subset ${\mathcal J}\subset\{1,2,\dots,n\}$ such that
$r_i=0$ (consequently $\rho_i=\kappa^{-1/2}$) for all $i\in{\mathcal J}$
and $\rho_j=0$ (consequently $r_j=\kappa^{-1/2}$) for all $j\in\{1,2,\dots,n\}\setminus{\mathcal J}$,

(d) $\kappa\alpha^2r_i^2+\kappa\beta^2\rho_i^2-\alpha^2=\kappa\alpha^2r_i^2+\kappa\beta^2\rho_i^2-\beta^2=0, \ i\in\{1,2,\dots,n\}$.

Cases (a) and (b) correspond to simply rotating $\kappa$-positive relative equilibria, thus recovering condition {\it (i)} in Criterion \ref{fixed-simply}. Case (c) corresponds to {\it (i)} in the above statement. Since, from Definition \ref{elliptic-positive}, it follows that the frequencies $\alpha$ and $\beta$ are nonzero, the identities in case (d) can obviously take place only if $\alpha^2=\beta^2$, i.e.\ $|\alpha|=|\beta|\ne 0$, so (d) corresponds to condition {\it (ii)} in the above statement. This remark completes the proof.
\end{proof}

\subsection{Criterion for $\kappa$-negative elliptic relative equilibria}

We further consider the motion in $\mathbb H_\kappa^3$ and start with proving a
criterion for the existence of simply rotating $\kappa$-negative elliptic relative equilibria.

\begin{criterion}{\bf ($\kappa$-negative elliptic relative equilibria)}
Consider the point particles of masses $m_1,m_2,\dots,m_n>0, n\ge 2$, in ${\mathbb H}_\kappa^3$, i.e.\ for $\kappa<0$. Then system \eqref{second} admits solutions of
the form \eqref{elliptic2-summarized}:
$$
{\bf q}=({\bf q}_1,{\bf q}_2,\dots, {\bf q}_n), {\bf q}_i=(w_i,x_i,y_i,z_i), i=1,2,\dots,n,
$$
$$
w_i(t)=r_i\cos(\alpha t+a_i),\ x_i(t)=r_i\sin(\alpha t+a_i),\ y_i(t)=y_i,\
z_i(t)=z_i,
$$
with $w_i^2+x_i^2=r_i^2, r_i^2+y_i^2-z_i^2=\kappa^{-1}$, and $y_i,z_i$ constant, $i=1,2,\dots,n$,
i.e.\ (simply rotating) $\kappa$-negative elliptic relative equilibria, if and only if there are constants $r_i, a_i, y_i, z_i,\ i=1,2,\dots,n$, and $\alpha\ne 0$, such that the following $4n$ conditions are satisfied:
\beq
\sum_{\stackrel{j=1}{j\ne i}}^n\frac{m_j|\kappa|^{3/2}(r_j\cos a_j-\kappa\epsilon_{ij}r_i\cos a_i)}{(\kappa^2\epsilon_{ij}^2-1)^{3/2}}
=(\kappa r_i^2-1)\alpha^2r_i\cos a_i,
\label{w3}
\eeq
\beq
\sum_{\stackrel{j=1}{j\ne i}}^n\frac{m_j|\kappa|^{3/2}(r_j\sin a_j-\kappa\epsilon_{ij}r_i\sin a_i)}{(\kappa^2\epsilon_{ij}^2-1)^{3/2}}
=(\kappa r_i^2-1)\alpha^2r_i\sin a_i, 
\label{x3}
\eeq
\beq
\sum_{\stackrel{j=1}{j\ne i}}^n\frac{m_j|\kappa|^{3/2}(y_j-\kappa\epsilon_{ij}y_i)}{(\kappa^2\epsilon_{ij}^2-1)^{3/2}}=\kappa\alpha^2r_i^2y_i,
\label{y3}
\eeq
\beq
\sum_{\stackrel{j=1}{j\ne i}}^n\frac{m_j|\kappa|^{3/2}(z_j-\kappa\epsilon_{ij}z_i)}{(\kappa^2\epsilon_{ij}^2-1)^{3/2}}=\kappa\alpha^2r_i^2z_i,
\label{z3} 
\eeq
$i=1,2,\dots,n$, where $\epsilon_{ij}=r_ir_j\cos(a_i-a_j)+y_iy_j-z_iz_j,\ i,j=1,2,\dots,n,\ i\ne j$.
\label{cri-elliptic2}
\end{criterion}
\begin{proof}
Consider a candidate ${\bf q}$ as above for a solution of system \eqref{second}. Some straightforward computations show that
$$
\epsilon_{ij}:={\bf q}_i\odot{\bf q}_j=r_ir_j\cos(a_i-a_j)+y_iy_j-z_iz_j,\ i,j=1,2\dots,n,\ i\ne j,
$$
$$
\dot{\bf q}_i\odot\dot{\bf q}_i=\alpha^2r_i^2,\ i=1,2,\dots,n,
$$
$$
\ddot{w}_i=-\alpha^2r_i\cos(\alpha t+a_i),\ \ddot{x}_i=-\alpha^2r_i\sin(\alpha t+a_i),
$$
$$
\ddot{y}_i=\ddot{z}_i=0,\ i=1,2,\dots,n.
$$
Substituting $\bf q$ and the above expressions into the equations of motion \eqref{second}, for the $w$ coordinates we obtain conditions involving $\cos(\alpha t+a_i)$, whereas for the $x$ coordinates we obtain conditions involving $\sin(\alpha t+a_i)$. In the former case, using the fact that $\cos(\alpha t+a_i)=\cos\alpha t\cos a_i-\sin\alpha t\sin a_i$, we can split each equation in two, one involving $\cos\alpha t$ and the other $\sin\alpha t$ as factors. The same thing happens in the latter case if we use the formula $\sin(\alpha t+a_i)=\sin\alpha t\cos a_i+\cos\alpha t\sin a_i$. Each of these equations are satisfied if and only if conditions \eqref{w3} and \eqref{x3} take place. Conditions \eqref{y3} and \eqref{z3} follow directly from the
equations involving the coordinates $y$ and $z$. This remark completes the proof.
\end{proof}

\subsection{Criterion for $\kappa$-negative hyperbolic relative equilibria}

We continue our study of the hyperbolic space with proving a criterion that shows under
what conditions simply rotating $\kappa$-negative hyperbolic orbits exist.

\begin{criterion}{\bf ($\kappa$-negative hyperbolic relative equilibria)}
Consider the point particles of masses $m_1,m_2,\dots,m_n>0, n\ge 2$, in ${\mathbb H}_\kappa^3$, i.e.\ for $\kappa<0$. Then the equations of motion \eqref{second} admit solutions of
the form \eqref{hyperbolic-summarized}:
$$
{\bf q}=({\bf q}_1,{\bf q}_2,\dots,{\bf q}_n), {\bf q}_i=(w_i,x_i,y_i,z_i),\ i=1,2,\dots,n,
$$
\begin{align*}
w_i(t)&=w_i\ {\rm(constant)}, &  x_i(t)&=x_i\ {\rm(constant)},\\
y_i(t)&=\eta_i\sinh(\beta t+b_i), &
z_i(t)&=\eta_i\cosh(\beta t+b_i),
\end{align*}
with $y_i^2-z_i^2=-\eta_i^2,\ w_i^2+x_i^2-\eta_i^2=\kappa^{-1},
\ i=1,2,\dots,n$, i.e.\ (simply rotating) $\kappa$-negative hyperbolic relative equilibria, if and only if there are constants $\eta_i, w_i, x_i,\ i=1,2,\dots,n$, and $\beta\ne 0$, such that the following $4n$ conditions are satisfied:
\beq
\sum_{\stackrel{j=1}{j\ne i}}^n\frac{m_j|\kappa|^{3/2}(w_j-\kappa\mu_{ij}w_i)}{(\kappa^2\mu_{ij}^2-1)^{3/2}}=\kappa\beta^2\eta_i^2w_i,
\label{w4}
\eeq
\beq
\sum_{\stackrel{j=1}{j\ne i}}^n\frac{m_j|\kappa|^{3/2}(x_j-\kappa\mu_{ij}x_i)}{(\kappa^2\mu_{ij}^2-1)^{3/2}}=\kappa\beta^2\eta_i^2x_i,
\label{x4} 
\eeq
\beq
\sum_{\stackrel{j=1}{j\ne i}}^n\frac{m_j|\kappa|^{3/2}(\eta_j\sinh b_j-\kappa\mu_{ij}\eta_i\sinh b_i)}{(\kappa^2\mu_{ij}^2-1)^{3/2}}
=(\kappa\eta_i^2+1)\beta^2\eta_i\sinh b_i, 
\label{y4}
\eeq
\beq
\sum_{\stackrel{j=1}{j\ne i}}^n\frac{m_j|\kappa|^{3/2}(\eta_j\cosh b_j-\kappa\mu_{ij}\eta_i\cosh b_i)}{(\kappa^2\mu_{ij}^2-1)^{3/2}}
=(\kappa\eta_i^2+1)\beta^2\eta_i\cosh b_i, 
\label{z4}
\eeq
$i=1,2,\dots,n$, where $\mu_{ij}=w_iw_j+x_ix_j-\eta_i\eta_j\cosh(b_i-b_j),\ i,j=1,2,\dots,n,\ i\ne j$.
\label{crit-hyp}
\end{criterion}
\begin{proof}
Consider a candidate ${\bf q}$ as above for a solution of system \eqref{second}. Some straightforward computations show that 
$$
\mu_{ij}:={\bf q}_i\boxdot{\bf q}_j=w_iw_j+x_ix_j-\eta_i\eta_j\cosh(b_i-b_j),\ i,j=1,2,\dots,n,\ i\ne j,
$$
$$
\dot{\bf q}_i\boxdot\dot{\bf q}_i=\beta^2\eta_i^2,\ i=1,2,\dots, n,
$$
$$
\ddot{w}_i=\ddot{x}_i=0, 
$$
$$
\ddot{y}_i=\beta^2\eta_i\sinh(\beta t+b_i),\  \ddot{z}_i= \beta^2\eta_i\cosh(\beta i+b_i),\ i=1,2,\dots,n.
$$
Substituting $\bf q$ and the above expressions into the equations of motion \eqref{second}, we immediately obtain for the $w$ and $x$ coordinates the equations \eqref{w4} and \eqref{x4}, respectively. For the $y$ and $z$ coordinates we obtain conditions involving $\sinh(\beta t+b_i)$ and $\cosh(\beta t+b_i)$, respectively. In the former case, using the fact that $\sinh(\beta t+ b_i)=\sinh \beta t\cosh b_i+\cosh \beta t\sinh b_i$, we can split each equation in two, one involving $\sinh \beta t$ and the other $\cosh \beta t$ as factors. The same thing happens in the later case if we use the formula $\cosh(\beta t+b_i)=\cosh\beta t\cosh b_i+\sinh\beta t\sinh b_i$. Each of these conditions are satisfied if and only of conditions \eqref{y4} and \eqref{z4} take place. This remark completes the proof.
\end{proof}

\subsection{Criterion for $\kappa$-negative elliptic-hyperbolic relative equilibria}

We end our study of existence criteria for relative equilibria in hyperbolic space with a result that shows under what conditions simply rotating $\kappa$-negative elliptic-hyperbolic orbits exist.

\begin{criterion}{\bf ($\kappa$-negative elliptic-hyperbolic relative equilibria)} 
Consider the point particles of masses $m_1,m_2,\dots,m_n>0, n\ge 2$, in ${\mathbb H}_\kappa^3$, i.e.\ for $\kappa<0$. Then the equations of motion \eqref{second} admit solutions of
the form \eqref{elliptic-hyperbolic-summarized}:
$$
{\bf q}=({\bf q}_1,{\bf q}_2,\dots,{\bf q}_n), {\bf q}_i=(w_i,x_i,y_i,z_i),\ i=1,2,\dots,n,
$$
\begin{align*}
w_i(t)&=r_i\cos(\alpha t+a_i), &
x_i(t)&=r_i\sin(\alpha t+a_i),\\
y_i(t)&=\eta_i\sinh(\beta t+b_i), &
z_i(t)&=\eta_i\cosh(\beta t+b_i),
\end{align*}
i.e.\ (doubly rotating) $\kappa$-negative elliptic-hyperbolic relative equilibria, if and only if there are constants $r_i,\eta_i, a_i,b_i,\ i=1,2\dots,n$, and $\alpha,\beta\ne 0$, such that the following $4n$ conditions are satisfied:
\beq
\sum_{\stackrel{j=1}{j\ne i}}^n\frac{m_j|\kappa|^{3/2}(r_j\cos a_j-\kappa\gamma_{ij}r_i\cos a_i)}{(\kappa^2\gamma_{ij}^2-1)^{3/2}}=(\kappa\alpha^2 r_i^2+\kappa\beta^2\eta_i^2-\alpha^2)r_i\cos a_i,
\label{w5}
\eeq
\beq
\sum_{\stackrel{j=1}{j\ne i}}^n\frac{m_j|\kappa|^{3/2}(r_j\sin a_j-\kappa\gamma_{ij}r_i\sin a_i)}{(\kappa^2\gamma_{ij}^2-1)^{3/2}}=(\kappa\alpha^2 r_i^2+\kappa\beta^2\eta_i^2-\alpha^2)r_i\sin a_i,
\label{x5} 
\eeq
\beq
\sum_{\stackrel{j=1}{j\ne i}}^n\frac{m_j|\kappa|^{3/2}(\eta_j\sinh b_j-\kappa\gamma_{ij}\eta_i\sinh b_i)}{(\kappa^2\gamma_{ij}^2-1)^{3/2}}
=(\kappa\alpha^2 r_i^2+\kappa\beta^2\eta_i^2+\beta^2)\eta_i\sinh b_i, 
\label{y5}
\eeq
\beq
\sum_{\stackrel{j=1}{j\ne i}}^n\frac{m_j|\kappa|^{3/2}(\eta_j\cosh b_j-\kappa\gamma_{ij}\eta_i\cosh b_i)}{(\kappa^2\gamma_{ij}^2-1)^{3/2}}
=(\kappa\alpha^2 r_i^2+\kappa\beta^2\eta_i^2+\beta^2)\eta_i\cosh b_i, 
\label{z5}
\eeq
$i=1,2,\dots,n$, where $\gamma_{ij}=r_ir_j\cos(a_i-a_j)-\eta_i\eta_j\cosh(b_i-b_j),\ i,j=1,2,\dots,n,\ i\ne j$.
\label{crit-elliptic-hyp}
\end{criterion}
\begin{proof}
Consider a candidate ${\bf q}$ as above for a solution of system \eqref{second}. Some straightforward computations show that 
$$
\gamma_{ij}:={\bf q}_i\boxdot{\bf q}_j=r_ir_j\cos(a_i-a_j)-\eta_i\eta_j\cosh(b_i-b_j), i,j=1,2,\dots,n,\ i\ne j,
$$
$$
\dot{\bf q}_i\boxdot\dot{\bf q}_i=\alpha^2 r_i^2+\beta^2\eta_i^2,\ i=1,2,\dots, n,
$$
$$
\ddot{w}_i=-\alpha^2r_i\cos(\alpha t+a_i),\ \ddot{x}_i=-\alpha^2r_i\sin(\alpha t+a_i), 
$$
$$
\ddot{y}_i=\beta^2\eta_i\sinh(\beta t+b_i),\  \ddot{z}_i= \beta^2\eta_i\cosh(\beta i+b_i),\ i=1,2,\dots,n.
$$
Substituting these expression and those that define $\bf q$ into the equations of motion \eqref{second}, we 
obtain for the $w$ and $x$ coordinates conditions involving $\cos(\alpha t+a_i)$ and $\sin(\alpha t+a_i)$,
respectively. In the former case, using the fact that $\cos(\alpha t+a_i)=\cos\alpha t\cos a_i-\sin\alpha t\sin a_i$, we can split each equation in two, one involving $\cos\alpha t$ and the other $\sin\alpha t$ as factors. The same thing happens in the latter case if we use the formula $\sin(\alpha t+a_i)=\sin\alpha t\cos a_i+\cos\alpha t\sin a_i$. Each of these equations are satisfied if and only if conditions \eqref{w5} and \eqref{x5} take place.

For the $y$ and $z$ coordinates we obtain conditions involving $\sinh(\beta t+b_i)$ and $\cosh(\beta t+b_i)$, respectively. In the former case, using the fact that $\sinh(\beta t+ b_i)=\sinh \beta t\cosh b_i+\cosh \beta t\sinh b_i$, we can split each equation in two, one involving $\sinh \beta t$ and the other $\cosh \beta t$ as factors. The same thing happens in the later case if we use the formula $\cosh(\beta t+b_i)=\cosh\beta t\cosh b_i+\sinh\beta t\sinh b_i$. Each of these conditions are satisfied if and only of conditions \eqref{y5} and \eqref{z5} take place. This remark completes the proof.
\end{proof}

\subsection{Nonexistence of $\kappa$-negative parabolic orbits}\label{non-parabolic}

The same as in the curved $n$-body problem restricted to ${\mathbb H}_\kappa^2$, parabolic relative equilibria do not exist in ${\mathbb H}_\kappa^3$. The idea of the proof is similar to the one we used in \cite{Diacu1}: it exploits the basic fact that a relative equilibrium of parabolic type would violate the conservation law of the angular momentum. Here are a formal statement and a proof of this result.

\begin{proposition}{\bf (Nonexistence of $\kappa$-negative parabolic relative equilibria)}
Consider the point particles of masses $m_1,m_2,\dots,m_n>0, n\ge 2$, in ${\mathbb H}_\kappa^3$, i.e.\ for $\kappa<0$. Then system \eqref{second} does not admit solutions of the form \eqref{para}, which means that $\kappa$-negative parabolic relative equilibria do not exist in the 3-dimensional 
curved $n$-body problem.
\end{proposition}
\begin{proof}
Checking a solution of the form \eqref{para} into the last integral of 
\eqref{angularmomentum}, we obtain that
$$c_{34}=\sum_{i=1}^nm_i(y_i\dot{z}_i-\dot{y}_iz_i)
=\sum_{i=1}^nm_i\left[\gamma_i+\beta_it+(\delta_i-\gamma_i)\frac{t^2}{2}\right][\beta_i+
(\delta_i-\gamma_i)t]$$
$$-\sum_{i=1}^nm_i\left[\delta_i+\beta_it+(\delta_i-\gamma_i)\frac{t^2}{2}\right][\beta_i+
(\delta_i-\gamma_i)t]$$
$$=\sum_{i=1}^nm_i\beta_i(\gamma_i-\delta_i)-
\sum_{i=1}^nm_i(\gamma_i-\delta_i)^2t,\ i\in\{1,2,\dots,n\}.$$
Since $c_{34}$ is constant, it follows that $\gamma_i=\delta_i,\ i\in\{1,2,\dots,n\}$.
But from \eqref{para} we obtain that $\alpha_i^2+\beta_i^2=\kappa^{-1}<0$, a
contradiction which proves that parabolic relative equilibria don't exist. This remark
completes the proof.
\end{proof}

\section{Qualitative behaviour of the relative equilibria
in ${\mathbb S}_\kappa^3$}

In this section we will describe some qualitative dynamical properties for the $\kappa$-positive elliptic and $\kappa$-positive elliptic-elliptic relative equilibria, under the assumption that they exist. (Examples of such solutions will be given in Sections 17 and 18 for various values of $n$ and of the masses $m_1, m_2,\dots, m_n>0$.) For this purpose we start with some geometric-topologic considerations about ${\mathbb S}_\kappa^3$. 

\subsection{Some geometric topology in ${\mathbb S}_\kappa^3$}

Consider the circle of radius $r$, in the $wx$ plane of $\mathbb R^4$ and the circle of radius $\rho$, in the $yz$ plane of $\mathbb R^4$, with $r^2+\rho^2=\kappa^{-1}$. Then ${\bf T}_{r\rho}^2$ is the cartesian product of these two circles, i.e.\ a 2-dimensional surface of genus 1, called a Clifford torus. Since these two circles are submanifolds embedded in $\mathbb R^2$, ${\bf T}_{r\rho}^2$ is embedded in $\mathbb R^4$. But ${\bf T}_{r\rho}^2$ also lies on the sphere $\mathbb S_\kappa^3$, which has radius $R=\kappa^{-1/2}$. Indeed, we can represent this torus as
\begin{equation}
{\bf T}_{r\rho}^2=\{(w,x,y,z)\ |\ r^2+\rho^2=\kappa^{-1}, 0\le\theta,\phi<2\pi\},
\label{clifford-torus}
\end{equation}
where $w=r\cos\theta,\ x=r\sin\theta,\ y=\rho\cos\phi$, and $z=\rho\sin\phi$, so the distance from the origin of the coordinate system to any point of the Clifford torus is
$$
(r^2\cos^2\theta+r^2\sin^2\theta+\rho^2\cos^2\phi+\rho^2\sin^2\phi)^{1/2}=
(r^2+\rho^2)^{1/2}=\kappa^{-1/2}=R.
$$
When $r$ (and, consequently, $\rho$) takes all the values between $0$ and 
$R$, the family of Clifford tori such defined foliates ${\mathbb S}_\kappa^3$ (see Figure \ref{cliff}). Each Clifford torus splits ${\mathbb S}_\kappa^3$ into two solid tori and forms the boundary between them. The two solid tori are congruent when $r=\rho=R/\sqrt{2}$. For the sphere ${\mathbb S}_\kappa^3$, this is the standard Heegaard splitting\footnote{A Heegaard splitting, named after the Danish mathematician Poul Heegaard (1871-1943), is a decomposition of a compact, connected, oriented 3-dimensional manifold along its boundary into two manifolds having the same genus $g$, with $g=0,1,2,\dots$} of genus 1.

\begin{figure}[htbp] 
   \centering
   \includegraphics[width=1.8in]{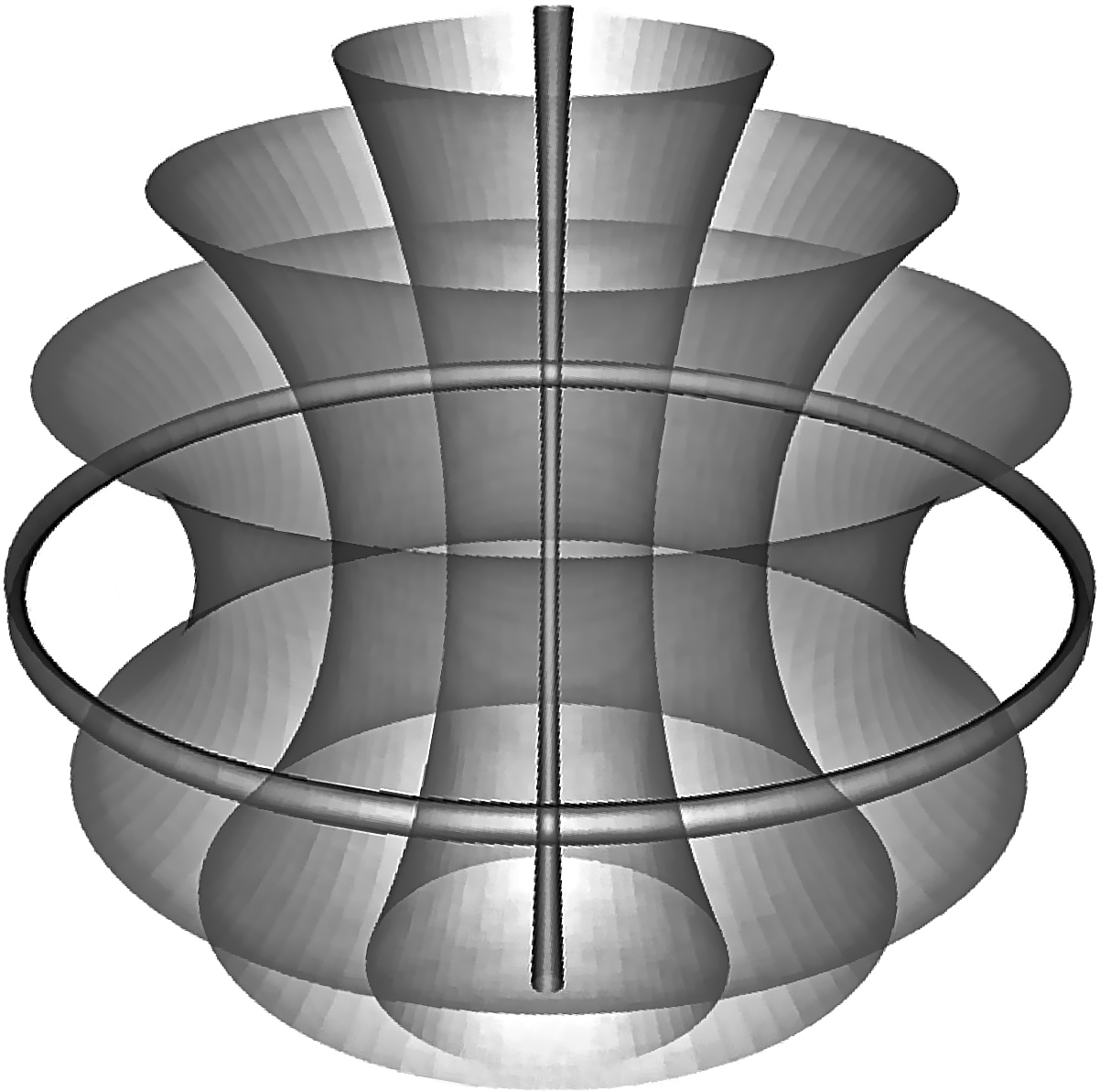}
   \caption{A 3-dimensional projection of a 4-dimensional foliation of the sphere ${\mathbb S}_\kappa^3$ into Clifford tori.}
   \label{cliff}
\end{figure}

Unlike regular tori embedded in $\mathbb R^3$, Clifford tori have zero Gaussian curvature at every point. Their flatness is due to the existence of an additional dimension in $\mathbb R^4$. Indeed, cylinders, obtained by pasting two opposite sides of a square, are flat surfaces both in $\mathbb R^3$ and $\mathbb R^4$. But to form a torus by pasting the other two sides of the square, cylinders must be stretched in $\mathbb R^3$. In $\mathbb R^4$, the extra dimension allows pasting without stretching.

\subsection{Qualitative properties of the relative equilibria in ${\mathbb S}_\kappa^3$}\label{theorem1}

The above considerations allow us to state and prove the following result, under the assumption that $\kappa$-positive elliptic and $\kappa$-positive elliptic-elliptic relative equilibria exist in ${\mathbb S}_\kappa^3$.

\begin{theorem}{\bf (Qualitative behaviour of the relative equilibria in ${\mathbb S}_\kappa^3$)} Assume that, in the 3-dimensional curved $n$-body problem, $n\ge 3$, of curvature $\kappa>0$, with bodies of masses $m_1, m_2,\dots, m_n>0$, $\kappa$-positive elliptic and $\kappa$-positive elliptic-elliptic relative equilibria exist.
Then the corresponding solution ${\bf q}$ may have one of the following dynamical behaviours:

(i) If  ${\bf q}$ is given by \eqref{elliptic1-summarized}, the orbit is a (simply rotating)
$\kappa$-positive elliptic relative equilibrium, with
the body of mass $m_i$ moving on a (not necessarily geodesic) circle ${\mathcal C}_i, i=1,2,\dots,n,$ of a 2-dimensional sphere in ${\mathbb S}_\kappa^3$; in the hyperplanes $wxy$ and $wxz$, the circles ${\mathcal C}_i$ are parallel with the plane $wx$; another possibility is that some bodies rotate on a great circle of a great sphere, while the other bodies stay fixed on a complementary great circle of another great sphere.

(ii)  If ${\bf q}$ is given by \eqref{elliptic-elliptic-summarized}, the orbit is a (doubly rotating) $\kappa$-positive elliptic-elliptic relative equilibrium, with 
some bodies rotating on a great circle of a great sphere and the other bodies rotating on a complementary great circle of another great sphere; another possibility is that each body $m_i$ is moving on the Clifford torus ${\bf T}^2_{r_i\rho_i}, i=1,2,\dots,n$. 
\label{circle-Clifford}
\end{theorem}
\begin{proof}
(i) The bodies move on circles ${\mathcal C}_i, \ i=1,2,\dots,n$, because, by
\eqref{elliptic1-summarized}, the analytic expression of the orbit is given by 
$$
{\bf q}=({\bf q}_1,{\bf q}_2,\dots,{\bf q}_n), \ {\bf q}_i=(w_i,x_i,y_i,z_i),\ i=1,2,\dots,n,
$$
$$
w_i(t)=r_i\cos(\alpha t+a_i), x_i(t)=r_i\sin(\alpha t+a_i), y_i(t)=y_i,
z_i(t)=z_i,
$$
with $w_i^2+x_i^2=r_i^2, r_i^2+y_i^2+z_i^2=\kappa^{-1}$, and $y_i,z_i$ constant, $i=1,2,\dots,n$. This proves the first part of (i), except for the statements about parallelism. 

In particular, if some bodies lie on the circle 
$$
{\bf S}_{\kappa,wx}^1=\{(0,0,y,z)\ |\ y^2+z^2=\kappa^{-1}\},
$$
with $y_i(t)=y_i=$ constant and $z_i(t)=z_i=$ constant, 
then the elliptic rotation, which changes the coordinates $w$ and $x$, does not act on them, therefore the bodies don't move. This remark proves the second part of statement (i). 

To prove the parallelism statement from the first part of (i), let us first remark that, as the concept of two parallel lines makes sense only if the lines are contained in the same plane, the concept of two parallel planes has meaning only if the planes are contained in the same 3-dimensional space. This explains our formulation of the statement. Towards proving it, notice first that 
$$
c_{wx}=\sum_{i=1}^n m_i(w_i\dot{x}_i-\dot{w}_ix_i)=\alpha\sum_{i=1}^n m_i r_i^2
$$
and
$$
c_{yz}=\sum_{i=1}^n m_i(y_i\dot{z}_i-\dot{y}_iz_i)=0.
$$
These constants are independent of the bodies' position, a fact that confirms that they result from first integrals. To determine the values of the constants $c_{wy}, c_{wz}, c_{xy},$ and $c_{xz}$, we first compute that
$$
c_{wy}=\sum_{i=1}^n m_i(w_i\dot{y}_i-\dot{w}_iy_i)=\alpha\sum_{i=1}^n m_ir_iy_i\sin(\alpha t+a_i),
$$
$$
c_{wz}=\sum_{i=1}^n m_i(w_i\dot{z}_i-\dot{w}_iz_i)=\alpha\sum_{i=1}^n m_ir_iz_i\sin(\alpha t+a_i),
$$
$$
c_{xy}=\sum_{i=1}^n m_i(x_i\dot{y}_i-\dot{x}_iy_i)=\alpha\sum_{i=1}^n m_ir_iy_i\cos(\alpha t+a_i),
$$
$$
c_{xz}=\sum_{i=1}^n m_i(x_i\dot{z}_i-\dot{x}_iz_i)=\alpha\sum_{i=1}^n m_ir_iz_i\cos(\alpha t+a_i).
$$
Since they are constant, the first integrals must take the same value for the arguments $t=0$ and $t=\pi/\alpha$. But at $t=0$, we obtain
$$
c_{wy}=\alpha\sum_{i=1}^n m_ir_iy_i\sin a_i,\ \ \ c_{wy}=\alpha\sum_{i=1}^n m_ir_iz_i\sin a_i,
$$
$$
c_{xy}=\alpha\sum_{i=1}^n m_ir_iy_i\cos a_i,\ \ \ c_{xz}=\alpha\sum_{i=1}^n m_ir_iz_i\cos a_i,
$$
whereas at $t=\pi/\alpha$, we obtain
$$
c_{wy}=-\alpha\sum_{i=1}^n m_ir_iy_i\sin a_i,\ \ \ c_{wy}=-\alpha\sum_{i=1}^n m_ir_iz_i\sin a_i,
$$
$$
c_{xy}=-\alpha\sum_{i=1}^n m_ir_iy_i\cos a_i,\ \ \ c_{xz}=-\alpha\sum_{i=1}^n m_ir_iz_i\cos a_i.
$$
Consequently, $c_{wy}=c_{wz}=c_{xy}=c_{xz}=0$. Since, as we already showed, $c_{yz}=0$, it follows that the only nonzero constant of the total angular momentum is $c_{wz}$. This means that the particle system has nonzero total rotation with respect to the origin only in the $wx$ plane.   

To prove that the circles ${\mathcal C}_i, i=1,2,\dots,n,$ are parallel with the plane $wx$ in the hyperplanes $wxy$ and $wxz$, assume that one circle, say ${\mathcal C}_1$, does not satisfy this property. Then some orthogonal projection of ${\mathcal C}_1$ (within
either of the hyperplanes $wxy$ and $wxz$) in at least one of the other base planes, say $xy$, is an ellipse, not a segment---as it would be otherwise. Then the angular momentum of the body of mass $m_1$ relative to the plane $xy$ is nonzero. Should other circles have an
elliptic projection in the plane $xy$, the angular momentum of the corresponding
bodies would be nonzero as well. Moreover, all angular momenta would have the same sign because
all bodies move in the same direction on the original circles. Consequently $c_{xy}\ne 0$, in contradiction with our previous findings. Therefore the circles ${\mathcal C}_i, i=1,2,\dots,n$, must be parallel, as stated. 

(ii) When a $\kappa$-positive elliptic-elliptic (double) rotation acts on a system, if some bodies are on a great circle of a great sphere of ${\mathbb S}_\kappa^3$, while other are on a complementary great circle of another great sphere, then the former bodies move only because of one rotation, while the latter bodies move only because of the other rotation. The special geometric properties of complementary circles leads to this kind of qualitative behaviour.

To prove the other kind of qualitative behaviour, namely that the body of mass $m_i$ of the doubly rotating $\kappa$-positive elliptic-elliptic relative equilibrium moves on the Clifford torus ${\bf T}^2_{r_i\rho_i}, i=1,2,\dots,n$, it is enough to compare the form of the orbit given in \eqref{elliptic-elliptic-summarized} with the characterization \eqref{clifford-torus} of a Clifford torus.
This remark completes the proof.
\end{proof}

\subsection{Qualitative properties of the relative equilibria generated from fixed-point configurations in ${\mathbb S}_\kappa^3$} \label{theorem2}

We will further outline the dynamical consequences of Criterion \ref{fixed-simply} and Criterion \ref{fixed-doubly}, under the assumption that $\kappa$-positive elliptic and $\kappa$-positive elliptic-elliptic relative equilibria, both generated from fixed-point configurations, exist in ${\mathbb S}_\kappa^3$. This theorem deals with a subclass of the orbits whose qualitative behaviour we have just described.

\begin{theorem}{\bf (Qualitative behaviour of the relative equilibria generated from fixed-point configurations in ${\mathbb S}_\kappa^3$)} Consider in ${\mathbb S}_\kappa^3$, i.e.\ for $\kappa>0$, the point particles of masses $m_1, m_2, \dots, m_n>0, n\ge 2$. Then a relative equilibrium ${\bf q}$ generated from a fixed point configuration may have one of the following characteristics:

(i) $\bf q$ is a (simply rotating) $\kappa$-positive elliptic orbit for which all bodies rotate on the same great circle of a great sphere of ${\mathbb S}_\kappa^3$;

(ii) $\bf q$ is a (simply rotating) $\kappa$-positive elliptic orbit for which some bodies rotate on a great circle of a great sphere, while the other bodies are fixed on a complementary great circle of a different great sphere;

(iii) $\bf q$ is a (doubly rotating) $\kappa$-positive elliptic-elliptic orbit for which some bodies rotate with frequency $\alpha\ne 0$ on a great circle of a great sphere, while the other bodies rotate with frequency $\beta\ne 0$ on a complementary great circle of a different sphere; the frequencies may be different in size, i.e.\ $|\alpha|\ne|\beta|$;

(iv) $\bf q$ is a (doubly rotating) $\kappa$-positive elliptic-elliptic orbit with frequencies
$\alpha,\beta\ne 0$ equal in size, i.e.\ $|\alpha|=|\beta|$.
\label{theorem-fixed}
\end{theorem}
\begin{proof}
(i) From conclusion (i) of Criterion \ref{fixed-simply}, a (simply rotating) $\kappa$-positive elliptic relative equilibrium of the form 
$$
{\bf q}=({\bf q}_1,{\bf q}_2,\dots,{\bf q}_n),
{\bf q}_i=(w_i,x_i,y_i,z_i), \ i=1,2,\dots, n,
$$
$$
w_i=r_i\cos(\alpha t+a_i), x_i(t)=r_i\sin(\alpha t+a_i), y_i(t)=y_i, z_i(t)=z_i,
$$
with $w_i^2+x_i^2=r_i^2, r_i^2+y_i^2+z_i^2=\kappa^{-1}$ and $y_i,z_i$ constant, $i=1,2,\dots,n$, generated from a fixed-point configuration, must satisfy one of two additional conditions (besides the initial $4n$ equations), the first of which translates into 
$$
r_i=\kappa^{-1/2}, \ i=1,2,\dots, n.
$$
This property implies that $y_i=z_i=0,\ i=1,2,\dots, n$, so all bodies rotate along the same great circle of radius $\kappa^{-1/2}$, namely ${\bf S}_{\kappa,yz}^1$ thus proving the statement in this case.

(ii) From conclusion (ii) of Criterion \ref{fixed-simply}, there is a proper subset ${\mathcal I}\subset\{1,2,\dots,n\}$ such that $r_i=0$ for all $i\in{\mathcal I}$ and $r_j=\kappa^{-1/2}$ for all $j\in\{1,2,\dots,n\}\setminus{\mathcal I}$. 

The bodies for which $r_i=0$ must have $w_i=x_i=0$ and $y_i^2+z_i^2=\kappa^{-1}$, so they are fixed on the great circle ${\bf S}_{\kappa,wx}^1$, since $y_i$ and $z_i$ are constant, $i\in{\mathcal I}$, and no rotation acts on the coordinates $w$ and $x$.

As in the proof of (i) above, it follows that the bodies with $r_j=\kappa^{-1/2},\ j\in\{1,2,\dots,n\}\setminus{\mathcal I}$ rotate on the circle ${\bf S}_{\kappa,yz}^1$, which is complementary to ${\bf S}_{\kappa,wx}^1$ so statement (ii) is also proved. 

(iii) From conclusion (i) of Criterion \ref{fixed-doubly}, a (doubly rotating) $\kappa$-positive elliptic-elliptic relative equilibrium of the form
$$
{\bf q}=({\bf q}_1,{\bf q}_2,\dots, {\bf q}_n), {\bf q}_i=(w_i,x_i,y_i,z_i),\ i=1,2,\dots,n,
$$
\begin{align*}
w_i(t)&=r_i\cos(\alpha t+a_i), & x_i(t)&=r_i\sin(\alpha t+a_i),\\
y_i(t)&=\rho_i\cos(\beta t+b_i), & z_i(t)&=\rho_i\sin(\beta t+b_i),
\end{align*}
with $w_i^2+x_i^2=r_i^2,\ y_i^2+z_i^2=\rho_i^2,\ r_i^2+\rho_i^2=\kappa^{-1}, \ i=1,2,\dots,n$, 
generated from a fixed-point configuration, must satisfy one of two additional conditions (besides the initial $4n$ equations), the first of which says that there is a proper subset ${\mathcal J}\subset\{1,2,\dots,n\}$ such that $r_i=0$ for all $i\in{\mathcal J}$ and $\rho_j=0$ for all $j\in\{1,2,\dots,n\}\setminus{\mathcal J}$. But this means that the bodies $m_i$ with $i\in\mathcal J$ have $w_i=x_i=0$ and $y_i^2+z_i^2=\rho_i^2$, so one rotation acts along the great circle ${\bf S}_{\kappa,wx}^1$, while the bodies with $m_i, i\in\{1,2,\dots,n\}\setminus{\mathcal J}$ satisfy the conditions $w_i^2+x_i^2=r_i^2$ and $y_i=z_i=0$, so the other rotation acts on them along the great circle  ${\bf S}_{\kappa,yz}^1$, which is complementary to ${\bf S}_{\kappa,wx}^1$. Moreover, since the bodies are distributed on two complementary circles, there are no constraints on the frequencies $\alpha,\beta\ne 0$, so they can be independent of each other, a remark that proves the statement.

(iv) From statement (d) in the proof of Criterion \ref{fixed-doubly}, a (doubly rotating) $\kappa$-positive elliptic-elliptic relative equilibrium may exist also when the bodies are not necessarily on complementary circles but the frequencies satisfy the condition $|\alpha|=|\beta|$, a case that concludes the last statement of this result. 
\end{proof}

\section{Qualitative behaviour of the relative equilibria
in ${\mathbb H}_\kappa^3$}

In this section we will describe some qualitative dynamical properties of the $\kappa$-negative elliptic, hyperbolic, and elliptic-hyperbolic relative equilibria, under the assumption that they exist. (Examples of such solutions, which prove their existence, will be given in Sections 19, 20, and 21.) For this purpose we start with some geometric-topologic considerations about ${\mathbb H}_\kappa^3$.

\subsection{Some geometric topology in ${\mathbb H}_\kappa^3$}

Usually, compact higher-dimensional manifolds have a richer geometry than non-compact manifolds of the same dimension. This is also true about ${\mathbb S}_\kappa^3$ if compared to ${\mathbb H}_\kappa^3$. Nevertheless, we will be able to characterize the relative equilibria of ${\mathbb H}_\kappa^3$ in geometric-topologic terms.

The surface we are introducing in this section, which will play for our dynamical analysis in ${\mathbb H}_\kappa^3$ the same role as the Clifford torus in ${\mathbb S}_\kappa^3$, is homoeomorphic to a cylinder.
Consider a circle of radius $r$ in the $wx$ plane of $\mathbb R^4$ and the upper branch of the hyperbola $r^2-\eta^2=\kappa^{-1}$ in the $yz$ plane of $\mathbb R^4$. Then we will call the surface ${\bf C}^2_{r\eta}$ obtained by taking the cartesian product between the circle and the hyperbola a hyperbolic cylinder since it surrounds equidistantly a branch of a geodesic hyperbola in ${\mathbb H}_\kappa^3$. Indeed, we can represent this cylinder as
\begin{equation}
{\bf C}^2_{r\eta}=\{(w, x, y, z)\ |\
r^2-\eta^2=\kappa^{-1}, \ 0\le\theta<2\pi, \ \xi\in\mathbb R\},
\label{hyp-cylinder}
\end{equation}
where $w=r\cos\theta,\ x=r\sin\theta,\ y=\eta\sinh\xi,\ z=\eta\cosh\xi$. 
But the hyperbolic cylinder ${\bf C}^2_{r\eta}$ also lies in ${\mathbb H}_\kappa^3$ because the coordinates $w, x, y, z$, endowed with the Lorentz metric, satisfy the equation
$$
w^2+x^2+y^2-z^2=r^2-\eta^2=\kappa^{-1}.
$$
As in the case of ${\mathbb S}_\kappa^3$, which is foliated by a family of Clifford tori,
${\mathbb H}_\kappa^3$ can be foliated by a family of hyperbolic cylinders. The foliation is, of course, not unique. But unlike the Clifford tori of $\mathbb R^4$, the hyperbolic cylinders of $\mathbb R^{3,1}$ are not flat surfaces. In general, they have constant positive curvature, which varies with the size of the cylinder, becoming zero only when the cylinder degenerates into a geodesic hyperbola.

\subsection{Qualitative properties of the relative equilibria in ${\mathbb H}_\kappa^3$}

The above considerations allow us to state and prove the following result, under the assumption that $\kappa$-negative elliptic, hyperbolic, and elliptic-hyperbolic relative equilibria exist. Notice that, on one hand, due to the absence of complementary circles, and, on the other hand, the absence of fixed points in ${\mathbb H}_\kappa^3$, the dynamical behaviour of relative equilibria is less complicated than in ${\mathbb S}_\kappa^3$.

\begin{theorem}{\bf (Qualitative behaviour of the relative equilibria in ${\mathbb H}_\kappa^3$)} In the 3-dimensional curved $n$-body problem, $n\ge 3$, of curvature 
$\kappa<0$, with bodies of masses $m_1, m_2,\dots, m_n>0$, every relative equilibrium $\bf q$ has one of the following potential behaviours:

(i) if $\bf q$ is given by \eqref{elliptic2-summarized}, the orbit is a (simply rotating) $\kappa$-negative elliptic relative equilibrium, with the body of mass $m_i$ moving on a circle ${\mathcal C}^i,\  i=1,2,\dots,n$, of a 2-dimensional hyperboloid in ${\mathbb H}_\kappa^3$; in  the hyperplanes $wxy$ and $wxz$, the planes of the circles ${\mathcal C}^i$ are parallel with the plane $wx$;

(ii) if $\bf q$ is given by \eqref{hyperbolic-summarized}, the orbit is a (simply rotating) $\kappa$-negative hyperbolic relative equilibrium, with the body of mass $m_i$ moving on some (not necessarily geodesic) hyperbola $\mathcal H_i$ of a 2-dimensional hyperboloid in ${\mathbb H}_\kappa^3,\ i=1,2,\dots,n$; in  the hyperplanes $wyz$ and $xyz$, the planes of the hyperbolas ${\mathcal C}^i$ are parallel with the plane $yz$;

(iii) if $\bf q$ is given by \eqref{elliptic-hyperbolic-summarized}, the orbit is a (doubly rotating) $\kappa$-negative elliptic-hyperbolic relative equilibrium, with the body of mass $m_i$ moving on the  hyperbolic cylinder ${\bf C}^2_{r_i\rho_i},\ i=1,2,\dots,n$.
\label{theorem3}
\end{theorem}
\begin{proof}
(i) The bodies move on circles, $\mathcal C_i,\ i=1,2,\dots, n,$ because, by \eqref{elliptic2-summarized}, the analytic expression of the orbit is given by
$$
{\bf q}=({\bf q}_1,{\bf q}_2,\dots, {\bf q}_n), {\bf q}_i=(w_i,x_i,y_i,z_i), i=1,2,\dots,n,
$$
$$
w_i(t)=r_i\cos(\alpha t+a_i),\ x_i(t)=r_i\sin(\alpha t+a_i),\ y_i(t)=y_i,\
z_i(t)=z_i,
$$
with $w_i^2+x_i^2=r_i^2,\ r_i^2+y_i^2-z_i^2=\kappa^{-1}$, and $y_i,z_i$ constant, $i=1,2,\dots,n$. The parallelism of the planes of the circles in the hyperplanes $wxy$ and $wxz$ follows exactly as in the proof of part (i) of Theorem \ref{circle-Clifford}, using the integrals of the total angular momenta.

(ii) The bodies move on hyperbolas, ${\mathcal H}_i,\ i=1,2,\dots,n$, because, by \eqref{hyperbolic-summarized}, the analytic expression of the orbit is given by
$$
{\bf q}=({\bf q}_1,{\bf q}_2,\dots,{\bf q}_n), {\bf q}_i=(w_i,x_i,y_i,z_i),\ i=1,2,\dots,n,
$$
\begin{align*}
w_i(t)&=w_i\ {\rm(constant)}, &  x_i(t)&=x_i\ {\rm(constant)},\\
y_i(t)&=\eta_i\sinh(\beta t+b_i), &
z_i(t)&=\eta_i\cosh(\beta t+b_i),
\end{align*}
with $y_i^2-z_i^2=-\eta_i^2,\ w_i^2+x_i^2-\eta_i^2=\kappa^{-1},
\ i=1,2,\dots,n$.

Let us now prove the parallelism statement for the planes containing the hyperbolas $\mathcal H_i$. For this purpose, notice that 
$$
c_{wx}=\sum_{i=1}^n m_i(w_i\dot{x}_i-\dot{w}_ix_i)=0
$$
and
$$
c_{yz}=\sum_{i=1}^n m_i(y_i\dot{z}_i-\dot{y}_iz_i)=-\beta\sum_{i=1}^nm_i\eta_i^2.
$$
These constants are independent of the bodies' position, a fact that confirms that they result from first integrals. To determine the values of the constants $c_{wy}, c_{wz}, c_{xy},$ and $c_{xz}$, we first compute that
$$
c_{wy}=\sum_{i=1}^n m_i(w_i\dot{y}_i-\dot{w}_iy_i)=\beta\sum_{i=1}^n m_iw_i\eta_i\cosh(\beta t+b_i),
$$
$$
c_{wz}=\sum_{i=1}^n m_i(w_i\dot{z}_i-\dot{w}_iz_i)=\beta\sum_{i=1}^n m_iw_i\eta_i\sinh(\beta t+b_i),
$$
$$
c_{xy}=\sum_{i=1}^n m_i(x_i\dot{y}_i-\dot{x}_iy_i)=\beta\sum_{i=1}^n m_ix_i\eta_i\cosh(\beta t+b_i),
$$
$$
c_{xz}=\sum_{i=1}^n m_i(x_i\dot{z}_i-\dot{x}_iz_i)=\beta\sum_{i=1}^n m_ix_i\eta_i\sinh(\beta t+b_i).
$$

We next show that $c_{wy}=0$. For this, notice first that,
using the formula $\cosh(\beta t+b_i)=\cosh b_i\cosh\beta t+\sinh b_i\sinh\beta t$, we can write 
\begin{equation}
c_{wy}=\beta[A(t)+B(t)],
\label{t}
\end{equation}
where
$$
A(t)=\sum_{i=1}^nm_iw_i\eta_i\cosh b_i\cosh\beta t, 
$$
and
$$
B(t)=\sum_{i=1}^nm_iw_i\eta_i\sinh b_i\sinh\beta t.
$$
But the function $\cosh$ is even, whereas $\sinh$ is odd. Therefore $A$ is even and $B$ is odd. Since $c_{wy}$ is constant, we also have
\begin{equation}
c_{wy}=\beta[A(-t)+B(-t)]=\beta[A(t)-B(t)].
\label{-t}
\end{equation}
From \eqref{t} and \eqref{-t} and the fact that $\beta\ne 0$, we can conclude that 
$$
c_{wy}=\beta A(t)\ \ {\rm and}\ \ B(t)=0,
$$
so $\frac{d}{dt}B(t)=0$. However, $\frac{d}{dt}B(t)=\beta A(t)$, which proves that $c_{wy}=0$.

The fact that $c_{xy}=0$ can be proved exactly the same way. The only difference when proving that $c_{wz}=0$ and $c_{xz}=0$ is the use of the corresponding hyperbolic formula, $\sinh(\beta t+b_i)=\sinh\beta t\cosh b_i+\cosh\beta t\sinh b_i$. In conclusion, 
$$
c_{wx}=c_{wy}=c_{wz}=c_{xy}=c_{xz}=0 \ \ {\rm and}\ \ c_{yz}\ne 0,
$$ 
which means that the hyperbolic rotation takes place relative to the origin of the coordinate system solely with respect to the plane $yz$.

Using a similar reasoning as in the proof of (i) for Theorem \ref{circle-Clifford}, it can be shown that the above conclusion proves the parallelism of the planes that contain the hyperbolas $\mathcal H_i$ in the 3-dimensional hyperplanes $wyz$ and $xyz$.

(iii) To prove that (doubly rotating) $\kappa$-negative elliptic-hyperbolic solutions move on hyperbolic cylinders, it is enough to compare the form of the orbit given in \eqref{hyperbolic-summarized} with the characterization \eqref{hyp-cylinder} of a hyperbolic cylinder.
\end{proof}

\newpage

\part{\rm EXAMPLES}

Since Theorems \ref{circle-Clifford}, \ref{theorem-fixed}, and \ref{theorem3} provide us with the qualitative behaviour of all the five classes of the relative equilibria that we expect to find in ${\mathbb S}_\kappa^3$ and ${\mathbb H}_\kappa^3$, we know what kind of rigid-body-type orbits to look for in the curved $n$-body problem for various values of $n\ge 3$. Ideal, of course, would be to find them all, but this problem appears to be very difficult, and it might never be completely solved. As a first step towards this (perhaps unreachable) goal, we will show that each type of orbit described in the above criteria and theorems exists for some values of $n$ and $m_1,m_2,\dots,m_n>0$.

To appreciate the difficulty of the above mentioned problem, we remark that its Euclidean analogue is that of finding all central configurations for the Newtonian potential. The notoreity of this problem has been recognized for at least seven decades, \cite{Win}. In fact, we don't even know whether, for some given masses $m_1, m_2,\dots, m_n>0$, with $n\ge 5$, the number of classes of central configurations (after we factorize the central configurations by size and rotation) is finite or not\footnote{Alain Albouy and Vadim Kaloshin have recently proved the finiteness of central configurations in the planar 5-body problem, except for a negligible set of masses, \cite{Albouy}.} and, if it is infinite, whether the set of classes of central configurations is discrete or contains a continuum. The finiteness of the number of classes of central configurations is Problem 6 on Steven Smale's list of mathematics problems for the 21st century, \cite{smale}. Its analogue in our case would be that of deciding whether, for given masses, $m_1,m_2,\dots, m_n>0$, the number of classes of relative equilibria of the 3-dimensional curved $n$-body problem is finite or not.

\section{Examples of $\kappa$-positive elliptic relative equilibria}

In this section, we will provide specific examples of $\kappa$-positive elliptic relative equilibria, i.e.\ orbits on the sphere ${\mathbb S}_\kappa^3$ that have a single rotation. The first example is that of a 3-body problem in which the 3 equal masses are at the vertices of an equilateral triangle that rotates along a not necessarily great circle of a great or non-great sphere. The second example is that of a 3-body problem in which the 3 non-equal masses move at the vertices of an acute scalene triangle along a great circle of a great sphere. The third example is that of a relative equilibrium generated from a fixed-point configuration in a 6-body problem of equal masses for which 3 equal masses move along a great circle of a great sphere at the vertices of an equilateral triangle, while the other 3 masses, which are equal to the others, are fixed on a complementary great circle of another great sphere at the vertices of an equilateral triangle. The fourth, and last example of the section, generalizes the third example to the case of acute scalene, not necessarily congruent, triangles and non-equal masses.

\subsection{A  class of $\kappa$-positive elliptic relative equilibrium of equal masses moving on a 2-dimensional sphere}\label{simply-Lag}

In the light of Remark \ref{z=0-remark}, we expect to find solutions
in ${\mathbb S}_\kappa^3$ that move on 2-dimensional spheres. A simple example is that of the Lagrangian solutions (i.e.\ equilateral triangles) of equal masses in the curved 3-body problem. Their existence, and the fact that they occur only when the masses are equal, was first proved in \cite{Diacu1} and \cite{Diacu001}. So, in this case, we have $n=3$ and $m_1=m_2=m_3=:m$. The solution of the corresponding system \eqref{second} we are seeking is of the form
\begin{equation}
{\bf q}=({\bf q}_1,{\bf q}_2,{\bf q}_3), \ {\bf q}_i=(w_i,x_i,y_i,z_i),\ i=1,2,3,
\label{0elliptic1}
\end{equation}
\begin{align*}
w_1(t)&=r\cos\omega t,& x_1(t)&=r\sin\omega t,\\
 y_1(t)&=y\ ({\rm constant}),& z_1(t)&=z\ ({\rm constant}),\\
w_2(t)&=r\cos(\omega t+2\pi/3),& x_2(t)&=r\sin(\omega t+2\pi/3),\\
y_2(t)&=y\ ({\rm constant}),& z_2(t)&=z\ ({\rm constant}),\\
w_3(t)&=r\cos(\omega t+4\pi/3),& x_3(t)&=r\sin(\omega t+4\pi/3),\\
y_3(t)&=y\ ({\rm constant}),& z_3(t)&=z\ ({\rm constant}),
\end{align*}
with $r^2+y^2+z^2=\kappa^{-1}$. Consequently, for the equations occurring in Criterion \ref{cri-elliptic1}, we have
$$
r_1=r_2=r_3=:r, \ \ a_1=0, a_2=2\pi/3, a_3=4\pi/3,
$$
$$
y_1=y_2=y_3=:y\ ({\rm constant}),\ \ z_1=z_2=z_3=:z\ ({\rm constant}).
$$
Substituting these values into the equations \eqref{w0}, \eqref{x0}, \eqref{y0}, \eqref{z0}, we obtain either identities or the equation
$$
\alpha^2=\frac{8m}{\sqrt{3}r^3(4-3\kappa r^2)^{3/2}}.
$$
Consequently, given $m>0,\ r\in(0,\kappa^{-1/2})$, and  $y, z$ with $r^2+y^2+z^2=\kappa^{-1}$, we can always find two frequencies, 
$$
\alpha_1=\frac{2}{r}\sqrt{\frac{2m}{\sqrt{3}r(4-3\kappa r^2)^{3/2}}}\ \ {\rm and}\ \  \alpha_2=-\frac{2}{r}\sqrt{\frac{2m}{\sqrt{3}r(4-3\kappa r^2)^{3/2}}},
$$
such that system \eqref{second} has a solution of the form \eqref{0elliptic1}. The positive frequency corresponds to one sense of rotation, whereas the negative frequency corresponds to the opposite sense.

Notice that if $r=\kappa^{-1/2}$, i.e.\ when the bodies move along a great circle of a great sphere, equations  \eqref{w0}, \eqref{x0}, \eqref{y0}, \eqref{z0} are identities for any $\alpha\in\mathbb R$, so any frequency leads to a solution. This phenomenon happens because, under those circumstances, the motion is generated from a fixed-point configuration, a case in which we can apply Criterion \ref{fixed-simply}, whose statement is independent of the frequency.

The bodies move on the great circle ${\bf S}_{\kappa,yz}^1$ of a great sphere only if $y=z=0$. Otherwise they move on non-great circles of great or non-great spheres. So we can also interpret this example as existing in the light of Remark \ref{z=z_0-remark}, which says that there are $\kappa$-positive elliptic rotations that leave non-great spheres invariant.

The constants of the angular momentum are 
$$
c_{wx}=3m\kappa^{-1}\omega\ne 0 \ \ {\rm and}\ \ c_{wy}=c_{wz}=c_{xy}=c_{xz}=c_{yz}=0,
$$
which means that the rotation takes place around the origin of the coordinate system only relative to the plane $wx$.

\subsection{Classes of $\kappa$-positive elliptic relative equilibria of non-equal masses moving on a 2-dimensional sphere}\label{nonequalmasses}

It is natural to ask whether solutions such as the one in the previous example also exist for non-equal masses. The answer is positive, and it was first answered in \cite{DiacuPol}, where we proved that, given a 2-dimensional sphere and any triangle inscribed in a great circle of it (for instance inside the equator $z=0$), there are masses, $m_1,m_2,m_3>0$, such that the bodies form a relative equilibrium that rotates around the $z$-axis.

We will consider a similar solution here in ${\mathbb S}_\kappa^3$, which moves on the great circle ${\bf S}_{\kappa, yz}^1$ of the great sphere ${\bf S}_{\kappa, y}^2$ or ${\bf S}_{\kappa,z}^2$. The expected analytic expression of the solution depends on the shape of the triangle, i.e.\ it has the form
\begin{equation}
{\bf q}=({\bf q}_1,{\bf q}_2,{\bf q}_3), \ {\bf q}_i=(w_i,x_i,y_i,z_i),\ i=1,2,3,
\label{equiattriangle}
\end{equation}
\begin{align*}
w_1(t)&=\kappa^{-1/2}\cos(\omega t+a_1),& x_1(t)&=\kappa^{-1/2}\sin(\omega t+a_1),\\
 y_1(t)&=0,& z_1(t)&=0,\\
w_2(t)&=\kappa^{-1/2}\cos(\omega t+a_2),& x_2(t)&=\kappa^{-1/2}\sin(\omega t+a_2),\\
y_2(t)&=0,& z_2(t)&=0,\\
w_3(t)&=\kappa^{-1/2}\cos(\omega t+a_3),& x_3(t)&=\kappa^{-1/2}\sin(\omega t+a_3),\\
y_3(t)&=0,& z_3(t)&=0,
\end{align*}
where the constants $a_1,a_2,a_3$, with $0\le a_1<a_2<a_3<2\pi$, determine the triangle's shape. The other constants involved in the description of this orbit are 
\begin{equation}
r_1=r_2=r_3=\kappa^{-1/2}.
\label{suppl-cond}
\end{equation}
We can use now Criterion \ref{fixed-simply} to prove that \eqref{equiattriangle} is a $\kappa$-positive elliptic relative equilibrium for any frequency $\omega\ne 0$. Indeed, we know from \cite{DiacuPol} that, for any shape of the triangle, there exist masses that yield a fixed point on the great circle ${\bf S}_{\kappa, yz}^1$, so the corresponding equations \eqref{w00}, \eqref{x00}, \eqref{y00}, \eqref{z00} are satisfied. Since conditions \eqref{suppl-cond} are also satisfied, the proof that \eqref{equiattriangle} is a solution of
\eqref{second} is complete.

The constants of the angular momentum integrals are
$$
c_{wx}=(m_1+m_2+m_3)\kappa^{-1}\omega\ne 0 \ \ {\rm and}\ \ c_{wy}=c_{wz}=c_{xy}=c_{xz}=c_{yz}=0,
$$
which means that the bodies rotate in $\mathbb R^4$ around the origin of the coordinate system only relative to the plane $xy$. 

\subsection{A class of $\kappa$-positive elliptic relative equilibria not contained on any 2-dimensional sphere}\label{non-2D}

The following example of a (simply rotating) $\kappa$-positive elliptic relative equilibrium in the curved 6-body problem corresponds to the second type of orbit described in part (i) of Theorem \ref{circle-Clifford}, and it is interesting from two points of view. First, it is an orbit that exists only in ${\mathbb S}_\kappa^3$, but cannot exist on any 2-dimensional sphere. Second, 3  bodies of equal masses move on a great circle of a great sphere at the vertices of an equilateral triangle, while the other 3 bodies of masses equal to the first stay fixed on a complementary great circle of another great sphere at the vertices of an equilateral triangle.

So consider the equal masses 
$$m_1=m_2=m_3=m_4=m_5=m_6=:m>0.$$ 
Then a solution as described above has the form
$$
{\bf q}=({\bf q}_1,{\bf q}_2, {\bf q}_3, {\bf q}_4, {\bf q}_5, {\bf q}_6),\ {\bf q}_i=(w_i,x_i,y_i,z_i),\ i=1,2,\dots,6,
$$
\begin{align*}
w_1&=\kappa^{-\frac{1}{2}}\cos\alpha t, & x_1&=\kappa^{-\frac{1}{2}}\sin\alpha t,\displaybreak[0]\\ 
y_1&=0, & z_1&=0,\displaybreak[0]\\
w_2&=\kappa^{-\frac{1}{2}}\cos(\alpha t+2\pi/3), & x_2&=\kappa^{-\frac{1}{2}}\sin(\alpha t +2\pi/3),\\ y_2&=0, & z_2&=0,\displaybreak[0]\\
w_3&=\kappa^{-\frac{1}{2}}\cos(\alpha t+4\pi/3), & x_3&=\kappa^{-\frac{1}{2}}\sin(\alpha t +4\pi/3),\\ y_3&=0, & z_3&=0,\displaybreak[0]\\
w_4&=0, & x_4&=0,\\ y_4&=\kappa^{-\frac{1}{2}}, & z_4&=0,\displaybreak[0]\\
w_5&=0, & x_5&=0,\\ y_5&=-\frac{\kappa^{-\frac{1}{2}}}{2}, & z_5&=\frac{\sqrt{3}\kappa^{-\frac{1}{2}}}{2},\displaybreak[0]\\
w_6&=0, & x_6&=0,\\ y_6&=-\frac{\kappa^{-\frac{1}{2}}}{2}, & z_6&=-\frac{\sqrt{3}\kappa^{-\frac{1}{2}}}{2}.
\end{align*}
A straightforward computation shows that this attempted orbit, which is generated from a fixed-point configuration, satisfies Criterion \ref{fixed-simply}, therefore it is indeed a solution of system \eqref{second} for $n=6$ and for any frequency $\alpha\ne 0$.

The constants of the angular momentum are 
$$
c_{wx}=3m\kappa^{-1}\alpha\ne 0\ \ {\rm and}\ \ 
c_{wy}=c_{wz}=c_{xy}=c_{xz}=c_{yz}=0,
$$
which implies that the rotation takes place around the origin of the coordinate system only relative to the $wx$ plane.

\subsection{Classes of $\kappa$-positive elliptic relative equilibria with non-equal masses not contained on any 2-dimensional sphere}
\label{2-ell-non-equal}

To prove the existence of $\kappa$-positive elliptic relative equilibria with non-equal masses not contained in any 2-dimensional sphere, we have only to combine the ideas of Subsections \ref{nonequalmasses} and \ref{non-2D}. More precisely, we consider a 6-body problem in which 3 bodies of non-equal masses, $m_1,m_2,m_3>0$, rotate on a great circle (lying, say, in the plane $wx$) of a great sphere at the vertices of an acute scalene triangle, while the other 3 bodies of non-equal masses, $m_4,m_5,m_6>0$, are fixed on a complementary great circle of another great sphere (lying, as a consequence of our choice of the previous circle, in the plane $yz$) at the vertices of another acute scalene triangle, which is not necessarily congruent with the first. 

Notice that, as shown in \cite{DiacuPol}, we must first choose the shapes of the triangles and then determine the masses that correspond to them, not the other way around. The reason for proceeding in this order is that not any 3 positive masses can rotate along a great circle of a great sphere. Like the solution offered in Subsection \ref{non-2D}, this orbit exists only in ${\mathbb S}_\kappa^3$, but cannot exist on any 2-dimensional sphere. Its analytic expression depends on the shapes of the two triangles, i.e.\
\begin{equation}
{\bf q}=({\bf q}_1,{\bf q}_2, {\bf q}_3, {\bf q}_4, {\bf q}_5, {\bf q}_6),\ {\bf q}_i=(w_i,x_i,y_i,z_i),\ i=1,2,\dots,6,
\label{rot-fixed}
\end{equation}
\begin{align*}
w_1&=\kappa^{-\frac{1}{2}}\cos(\alpha t+a_1), & x_1&=\kappa^{-\frac{1}{2}}\sin(\alpha t+a_1),\displaybreak[0]\\ 
y_1&=0, & z_1&=0,\displaybreak[0]\\
w_2&=\kappa^{-\frac{1}{2}}\cos(\alpha t+a_2), & x_2&=\kappa^{-\frac{1}{2}}\sin(\alpha t +a_2),\\ y_2&=0, & z_2&=0,\displaybreak[0]\\
w_3&=\kappa^{-\frac{1}{2}}\cos(\alpha t+a_3), & x_3&=\kappa^{-\frac{1}{2}}\sin(\alpha t +a_3),\\ y_3&=0, & z_3&=0,\displaybreak[0]\\
w_4&=0, & x_4&=0,\\ y_4&=\kappa^{-\frac{1}{2}}\cos b_4, & z_4&=\kappa^{-\frac{1}{2}}\sin b_4,\displaybreak[0]\\
w_5&=0, & x_5&=0,\\ y_5&=\kappa^{-\frac{1}{2}}\cos b_5, & z_5&=\kappa^{-\frac{1}{2}}\sin b_5,\displaybreak[0]\\
w_6&=0, & x_6&=0,\displaybreak[0]\\ 
y_6&=\kappa^{-\frac{1}{2}}\cos b_6, & z_6&=\kappa^{-\frac{1}{2}}\sin b_6 .
\end{align*}
where the constants $a_1,a_2$, and $a_3$, with $0\le a_1<a_2<a_3<2\pi$, and $b_4,b_5$, and $b_6$, with $0\le b_4<b_5<b_6<2\pi$, determine the shape of the first and second triangle, respectively.

For $t=0$, we obtain the configuration given by the coordinates
\begin{align*}
w_1&=\kappa^{-\frac{1}{2}}\cos a_1,& x_1&=\kappa^{-\frac{1}{2}}\sin a_1,& y_1&=0,& z_1&=0,\displaybreak[0]\\
w_2&=\kappa^{-\frac{1}{2}}\cos a_2,& x_2&=\kappa^{-\frac{1}{2}}\sin a_2,& y_2&=0,& z_2&=0,\displaybreak[0]\\
w_3&=\kappa^{-\frac{1}{2}}\cos a_3,& x_3&=\kappa^{-\frac{1}{2}}\sin a_3,& y_3&=0,& z_3&=0,\displaybreak[0]\\
w_4&=0,& x_4&=0,& y_4&=\kappa^{-\frac{1}{2}}\cos b_4,& z_4&=\kappa^{-\frac{1}{2}}\sin b_4,\displaybreak[0]\\
w_5&=0,& x_5&=0,& y_5&=\kappa^{-\frac{1}{2}}\cos b_5,& z_5&=\kappa^{-\frac{1}{2}}\sin b_5,\displaybreak[0]\\
w_6&=0,& x_6&=0,& y_6&=\kappa^{-\frac{1}{2}}\cos b_6,& z_6&=\kappa^{-\frac{1}{2}}\sin b_6.\\
\end{align*}
We prove next that this is a fixed-point configuration. For this purpose, let us first compute that
$$
\nu_{12}=\nu_{21}=\cos(a_1-a_2), 
$$
$$
\nu_{13}=\nu_{31}=\cos(a_1-a_3),
$$
$$
\nu_{23}=\nu_{32}=\cos(a_2-a_3),
$$
$$
\nu_{14}=\nu_{41}=\nu_{15}=\nu_{51}=\nu_{16}=\nu_{61}=0,
$$
$$
\nu_{24}=\nu_{42}=\nu_{25}=\nu_{52}=\nu_{26}=\nu_{62}=0,
$$
$$
\nu_{34}=\nu_{43}=\nu_{35}=\nu_{53}=\nu_{36}=\nu_{63}=0,
$$
$$
\nu_{45}=\nu_{54}=\cos(b_4-a_5), 
$$
$$
\nu_{46}=\nu_{64}=\cos(b_4-b_6),
$$
$$
\nu_{56}=\nu_{65}=\cos(b_5-b_6).
$$
Since $y_1=y_2=y_3=z_1=z_2=z_3=0$, it follows that, at $t=0$, the equations involving the force for the coordinates $w_1, x_1, w_2, x_2, w_3, x_3$ involve only the constants $m_1, m_2, m_3, a_1, a_2, a_3$. In other words, for these coordinates, the forces acting on the masses $m_1, m_2,$ and $m_3$ do not involve the masses $m_4, m_5,$ and $m_6$. But the bodies $m_1, m_2,$ and $m_3$ lie on the great circle ${\bf S}_{\kappa,yz}^1$, which can be seen as lying on the great sphere ${\bf S}_{\kappa,z}^2$. This means that, by applying the result of \cite{DiacuPol}, the bodies $m_1,m_2,$ and $m_3$ form an independent fixed-point configuration. Similarly, we can show that the masses $m_4, m_5,$ and $m_6$ form an independent fixed-point configuration. Therefore all 6 bodies form a fixed-point configuration.

Consequently, we can use Criterion \ref{fixed-simply} to check whether $\bf q$ given by \eqref{rot-fixed} is a $\kappa$-positive elliptic relative equilibrium generated from a fixed-point configuration. We can approach this problem in two ways. One is computational, and it consists of using the fact that the positions at $t=0$ form a fixed-point configuration to determine the relationships between the constants $m_1, m_2, m_3, a_1, a_2,$ and $a_3$, on one hand, and the constants $m_4,m_5,m_6, b_4, b_5,$ and $b_6$, on the other hand. It turns out that they reduce to conditions \eqref{w00}, \eqref{x00}, \eqref{y00}, and \eqref{z00}. Then we only need to remark that 
$$
r_1=r_2=r_3=\kappa^{-1/2}\ \ {\rm and}\ \ r_4=r_5=r_6=0,
$$
which means that condition (ii) of Criterion \ref{fixed-simply} is satisfied.
The other approach is to invoke again the result of \cite{DiacuPol} and a reasoning similar to the one we used to show that the position at $t=0$ is a fixed-point configuration. Both help us conclude that $\bf q$ given by \eqref{rot-fixed} is a solution of system \eqref{second} for any $\alpha\ne 0$.

The constants of the angular momentum are
$$
c_{wx}=(m_1+m_2+m_3)\kappa^{-1}\alpha\ne 0 \ \ {\rm and}\ \ c_{wy}=c_{wz}=c_{xy}=c_{xz}=c_{yz}=0,
$$
which means, as expected, that the bodies rotate in $\mathbb R^4$ around the origin of the coordinate system only relative to the plane $wx$.

\section{Examples of $\kappa$-positive elliptic-elliptic relative equilibria}

In this section we construct examples of $\kappa$-positive elliptic-elliptic relative equilibria, i.e.\ orbits with two elliptic rotations on the sphere ${\mathbb S}_\kappa^3$.
The first example is that of a 3-body problem in which 3 bodies of equal masses are at the vertices of an equilateral triangle, which has two rotations of the same frequency. The second example is that of a 4-body problem in which 4 equal masses are at the vertices of a regular tetrahedron, which has two rotations of the same frequency. The third example is that of 5 equal masses lying at the vertices of a pentatope with double rotation. This is the only case of the a regular polytope that allows relative equilibria, because the 5 other existing regular polytopes of $\mathbb R^4$ have antipodal vertices, so they introduce singularities. As in the previous example, this motion cannot take place on any 2-dimensional sphere. The fourth example is that of a 6-body problem, with 3 bodies of equal masses rotating at the vertices of an equilateral triangle along a great circle of a great sphere, while the other 3 bodies, of the same masses as the others, rotate at the vertices of an equilateral triangle along a complementary great circle of another great sphere. The frequencies of the two rotations are distinct, in general. The fifth example generalizes the fourth example in the sense that the triangles are scalene, acute, not necessarily congruent, and the masses as well as the frequencies of the rotations are distinct, in general.

\subsection{Equilateral triangle with equal masses moving with equal frequencies} 

The example we will now construct is that of a (doubly rotating) $\kappa$-positive elliptic-elliptic equilateral triangle of equal masses in the curved 3-body problem in ${\mathbb S}_\kappa^3$ for which the rotations have the same frequency. Such solutions cannot be found on 2-dimensional spheres. So we consider $\kappa>0$ and the masses $m_1=m_2=m_3=:m>0$. Then the solution we check for system \eqref{second} with $n=3$ has the form: 
\begin{equation}
{\bf q}=({\bf q}_1, {\bf q}_2, {\bf q}_3), \ {\bf q}_i=(w_i,x_i,y_i,z_i), \ i=1,2,3,
\label{ellip-3}
\end{equation}
\begin{align*}
w_1&=r\cos\alpha t, & x_1&=r\sin\alpha t,\\ 
y_1&=\rho\cos\alpha t, & z_1&=\rho\sin\alpha t,\displaybreak[0]\\
w_2&=r\cos\left(\alpha t+2\pi/3\right),& x_2&=r\sin\left(\alpha t+ 2\pi/3\right),\\ 
y_2&=\rho\cos(\alpha t+ 2\pi/3),& z_2&=\rho\sin(\alpha t + 2\pi/3),\displaybreak[0]\\
w_3&=r\cos\left(\alpha t+4\pi/3\right),& x_3&=r\sin\left(\alpha t+ 4\pi/3\right),\\
 y_3&=\rho\cos(\alpha t+4\pi/3),& z_3&=\rho\sin(\alpha t + 4\pi/3),
\end{align*}
with $r^2+\rho^2=\kappa^{-1}$. For $t=0$, the above attempted solution gives for the 3 bodies the coordinates
\begin{align*}
w_1&=r, & x_1&=0, & y_1&=\rho, & z_1&=0,\\
w_2&=-\frac{r}{2}, & x_2&=\frac{r\sqrt{3}}{2}, & y_2&=-\frac{\rho}{2}, & z_2&=\frac{\rho\sqrt{3}}{2},\\
w_3&=-\frac{r}{2}, & x_3&=-\frac{r\sqrt{3}}{2}, & y_3&=-\frac{\rho}{2}, & z_3&=-\frac{\rho\sqrt{3}}{2},
\end{align*}
which is a fixed-point configuration, since the bodies have equal masses and are at the vertices of an equilateral triangle inscribed in a great circle of a great sphere. Consequently, we can use Criterion \ref{fixed-doubly} to check whether a solution of the form \eqref{ellip-3} satisfies system \eqref{second} for any $\alpha\ne 0$. A straightforward computation shows that the first $4n$ conditions are satisfied. Moreover, since the two rotations have the same frequency, it follows that condition (ii) of Criterion \ref{fixed-doubly} is verified, therefore \eqref{ellip-3} is indeed a solution of system \eqref{second} for any $\alpha\ne 0$.

The angular momentum constants are 
$$
c_{wx}=3m\alpha r^2,\ \ c_{wy}=0,\ \ c_{wz}=3m\alpha r\rho,
$$
$$ 
c_{xy}=-3m\alpha r\rho,\ \ c_{xz}=0,\ \ c_{yz}=3m\alpha\rho^2,
$$
which show that rotations around the origin of the coordinate system take place relative to 4 planes: $wx, wz, xy$, and $yz$. Consequently  the bodies don't move on the same circle, but on the same Clifford torus, namely ${\bf T}_{r,\rho}^2$,
a case that agrees with the qualitative result described in part (ii) of Theorem
\ref{circle-Clifford}.

Notice that for $r=\kappa^{-1}$ and $\rho=0$, the orbit becomes a (simply rotating) $\kappa$-positive elliptic relative equilibrium that rotates along a great circle of a great sphere in ${\mathbb S}_\kappa^3$, i.e.\ an orbit such as the one we described in Subsection \ref{simply-Lag}.

\subsection{Regular tetrahedron with equal masses moving with e\-qual-size frequencies} 

We will further construct a (doubly rotating) $\kappa$-positive elliptic-elliptic solution of the 4-body problem in ${\mathbb S}_\kappa^3$, in which 4 equal masses are at the vertices of a regular tetrahedron that has rotations of equal frequencies. So let us fix $\kappa>0$ and $m_1=m_2=m_3=m_4=:m>0$ and consider the initial position of the 4 bodies to be given as in the first example of Subsection \ref{specific-fixed}, i.e.\ by the coordinates
\begin{align*}
w_1^0&=0, & x_1^0&=0, & y_1^0&=0, & z_1^0&=\kappa^{-1/2},\\
w_2^0&=0 ,& x_2^0&=0, & y_2^0&=\frac{2\sqrt{2}}{3}\kappa^{-1/2}, & z_2^0&=-\frac{1}{3}\kappa^{-1/2},\displaybreak[0]\\
w_3^0&=0, & x_3^0&=-\frac{\sqrt{6}}{3}\kappa^{-1/2},& y_3^0&=-\frac{\sqrt{2}}{3}\kappa^{-1/2},& z_3^0&=-\frac{1}{3}\kappa^{-1/2},\\
w_4^0&=0, & x_4^0&=\frac{\sqrt{6}}{3}\kappa^{-1/2},& y_4^0&=-\frac{\sqrt{2}}{3}\kappa^{-1/2},& z_4^0&=-\frac{1}{3}\kappa^{-1/2},
\end{align*}
which is a fixed-point configuration. Indeed, the masses are equal and the bodies are at the vertices of a regular tetrahedron inscribed in a great sphere
of ${\mathbb S}_\kappa^3$.

For this choice of initial positions, we can compute that
\begin{align*}
r_1&=r_2=0,&  \rho_1&=\rho_2=\kappa^{-1/2},\\
r_3&=r_4=\frac{\sqrt{6}}{3}\kappa^{-1/2}, &
\rho_3&=\rho_4=\frac{\sqrt{3}}{3}\kappa^{-1/2},
\end{align*}
which means that we expect that 2 bodies (corresponding to $m_1$ and $m_2$) move on the Clifford torus with $r=0$ and $\rho=\kappa^{-1/2}$ (i.e.\ the only Clifford torus, in the class of a given foliation of ${\mathbb S}_\kappa^3$, which is also a great circle of ${\mathbb S}_\kappa^3$, see Figure \ref{cliff}), while we expect the other 2 bodies to move on the Clifford torus with  $r=\frac{\sqrt{6}}{3}\kappa^{-1/2}$ and
$\rho=\frac{\sqrt{3}}{3}\kappa^{-1/2}$.

These considerations allow us to obtain the constants that determine the angles. Indeed, $a_1$ and $a_2$ can take any values, 
$$a_3=3\pi/2,\ a_4=\pi/2,\ b_1=\pi/2,$$
and $b_2, b_3, b_4$ are such that 
$$\sin b_2=-\frac{1}{3},\ \ \cos b_2=\frac{2\sqrt{2}}{3},$$
$$\cos b_3=-\frac{\sqrt{6}}{3},\ \  \sin b_3=-\frac{\sqrt{3}}{3}$$
$$\cos b_4=-\frac{\sqrt{6}}{3},\ \  \sin b_4=-\frac{\sqrt{3}}{3},$$
which means that $b_3=b_4$.

We can now compute the form of the candidate for a solution generated from the above fixed-point configuration. Using the above values of $r_i, \rho_i, a_i$ and $b_i$, $i=1,2,3,4$, we obtain from the equations
$$
w_i^0=r_i\cos a_i, \ x_i^0=r_i\sin a_i, \ y_i^0=\rho_i\cos b_i, \ z_i^0=\rho_i\sin b_i, \ i=1,2,3,4,
$$ 
that the candidate for a solution is given by
\begin{equation}
{\bf q}=({\bf q}_1, {\bf q}_2, {\bf q}_3, {\bf q}_4),\ {\bf q}_i=(w_i,x_i,y_i,z_i),\ i=1,2,3,4,
\label{ellip-ellip-tetra}
\end{equation}
\begin{align*}
w_1&=0,& x_1&=0,\\ 
y_1&=\kappa^{-1/2}\cos(\alpha t+\pi/2),& z_1&=\kappa^{-1/2}\sin(\alpha t +\pi/2),\displaybreak[0]\\
w_2&=0,& x_2&=0,\\ 
y_2&=\kappa^{-1/2}\cos(\alpha t+b_2),& z_2&=\kappa^{-1/2}\sin(\alpha t +b_2),\displaybreak[0]\\
w_3&=\frac{\sqrt{6}}{3}\kappa^{-1/2}\cos(\beta t+3\pi/2),& x_3&=\frac{\sqrt{6}}{3}\kappa^{-1/2}\sin(\beta t+3\pi/2),\displaybreak[0]\\
y_3&=\frac{\sqrt{3}}{3}\kappa^{-1/2}\cos(\beta t+b_3),& 
z_3&=\frac{\sqrt{3}}{3}\kappa^{-1/2}\sin(\beta t+b_3),\displaybreak[0]\\
w_4&=\frac{\sqrt{6}}{3}\kappa^{-1/2}\cos(\beta t+\pi/2),& x_4&=\frac{\sqrt{6}}{3}\kappa^{-1/2}\sin(\beta t+\pi/2),\displaybreak[0]\\
y_4&=\frac{\sqrt{3}}{3}\kappa^{-1/2}\cos(\beta t+b_4),& 
z_4&=\frac{\sqrt{3}}{3}\kappa^{-1/2}\sin(\beta t+b_4).
\end{align*}
If we invoke Criterion \ref{fixed-doubly}, do a straightforward computation, and use the fact that the frequencies of the two rotations have the same size, i.e.\ are equal in absolute value, we can conclude that $\bf q$, given by \eqref{ellip-ellip-tetra}, satisfies system \eqref{second}, so it is indeed a solution of the curved 4-body problem for $\kappa>0$. 

Straightforward computations lead us to the following values of the angular momentum constants:
$$
c_{wx}=\frac{4}{3}m\alpha\kappa^{-1},\ c_{wy}=c_{wz}=c_{xy}=c_{xz}=0,\ c_{yz}=\frac{8}{3}m\alpha\kappa^{-1},
$$
for $\beta=\alpha$ and 
$$
c_{wx}=\frac{4}{3}m\alpha\kappa^{-1},\ c_{wy}=c_{wz}=c_{xy}=c_{xz}=0,\ c_{yz}=-\frac{8}{3}m\alpha\kappa^{-1}
$$
for $\beta=-\alpha$, a fact which shows that rotations around the origin of the coordinate system takes place only relative to the planes $wx$ and $yz$.

\subsection{Regular pentatope with equal masses and equal-size frequencies}

We will next construct a (doubly rotating) $\kappa$-positive elliptic-elliptic solution of the 5-body problem in ${\mathbb S}_\kappa^3$, in which 5 equal masses are at the vertices of a regular pentatope that has two rotations of equal-size frequencies. So let us fix $\kappa>0$ and $m_1=m_2=m_3=m_4=m_5=:m>0$ and consider the initial position of the 5 bodies to be given as in the second example of Subsection \ref{specific-fixed}, i.e.\ by the coordinates
\begin{align*}
w_1^0&=\kappa^{-1/2},& x_1^0&=0,& y_1^0&=0,& z_1^0&=0,\displaybreak[0]\\
w_2^0&=-\frac{\kappa^{-1/2}}{4},& x_2^0&=\frac{\sqrt{15}\kappa^{-1/2}}{4},& y_2^0&=0,& z_2^0&=0,\displaybreak[0]\\
w_3^0&=-\frac{\kappa^{-1/2}}{4},& x_3^0&=-\frac{\sqrt{5}\kappa^{-1/2}}{4\sqrt{3}},& y_3^0&=\frac{\sqrt{5}\kappa^{-1/2}}{\sqrt{6}},& z_3^0&=0,\displaybreak[0]\\
w_4^0&=-\frac{\kappa^{-1/2}}{4},& x_4^0&=-\frac{\sqrt{5}\kappa^{-1/2}}{4\sqrt{3}},& y_4^0&=-\frac{\sqrt{5}\kappa^{-1/2}}{2\sqrt{6}},& z_4^0&=\frac{\sqrt{5}\kappa^{-1/2}}{2\sqrt{2}},\displaybreak[0]\\
w_5^0&=-\frac{\kappa^{-1/2}}{4},& x_5^0&=-\frac{\sqrt{5}\kappa^{-1/2}}{4\sqrt{3}},& y_5^0&=-\frac{\sqrt{5}\kappa^{-1/2}}{2\sqrt{6}},& z_5^0&=-\frac{\sqrt{5}\kappa^{-1/2}}{2\sqrt{2}},
\end{align*}
which is a fixed-point configuration because the masses are equal and the bodies are at the vertices of a regular pentatope inscribed in ${\mathbb S}_\kappa^3$.

For this choice of initial positions, we can compute that
\begin{align*}
r_1&=r_2=\kappa^{-1/2}, & \rho_1&=\rho_2=0,\\
r_3&=r_4=r_5=\frac{\kappa^{-1/2}}{\sqrt{6}}, & \rho_3&=\rho_4=\rho_5=\frac{\sqrt{5}\kappa^{-1/2}}{\sqrt{6}},
\end{align*}
which means that we expect that 2 bodies (corresponding to $m_1$ and $m_2$) move on the Clifford torus with $r=\kappa^{-1/2}$ and $\rho=0$ (i.e.\ the only Clifford torus, in a class of a given foliation of ${\mathbb S}_\kappa^3$, which is also a great circle of ${\mathbb S}_\kappa^3$), while we expect the other 3 bodies to move on the Clifford torus with $r=\frac{1}{\sqrt{6}}\kappa^{-1/2}$ and $\rho=\frac{\sqrt{5}}{\sqrt{6}}\kappa^{-1/2}$.

These considerations allow us to obtain the constants that determine the angles. We obtain that 
$$
a_1=0,
$$
$a_2$ is such that
$$
\cos a_2=-\frac{1}{4}, \ \ \sin a_2=-\frac{\sqrt{15}}{4}
$$
and $a_3, a_4, a_5$ are such that
$$
\cos a_3=-\frac{\sqrt{6}}{4}, \ \ \sin a_3=-\frac{\sqrt{10}}{4},
$$
$$
\cos a_4=-\frac{\sqrt{6}}{4}, \ \ \sin a_4=-\frac{\sqrt{10}}{4},
$$
$$
\cos a_5=-\frac{\sqrt{6}}{4}, \ \ \sin a_5=-\frac{\sqrt{10}}{4},
$$
which means that $a_3=a_4=a_5$. We further obtain that, since $\rho_1=\rho_2=0$,
the constants $b_1$ and $b_2$ can be anything, in particular 0. Further computations lead us to the conclusion that
$$
b_1=b_2=b_3=0,\ \ b_4=2\pi/3,\ \ b_5=4\pi/3.
$$

We can now compute the form of the candidate for a solution generated from the above fixed-point configuration. Using the above values of $r_i, \rho_i, a_i$ and $b_i$, $i=1,2,3,4,5$, we obtain from the equations
$$
w_i^0=r_i\cos a_i, \ x_i^0=r_i\sin a_i, \ y_i^0=\rho_i\cos b_i, \ z_i^0=\rho_i\sin b_i, \ i=1,2,3,4,5
$$ 
that the candidate for a solution is given by
\begin{equation}
{\bf q}=({\bf q}_1, {\bf q}_2, {\bf q}_3, {\bf q}_4),\ {\bf q}_i=(w_i,x_i,y_i,z_i),\ i=1,2,3,4,5
\label{ellip-ellip-penta}
\end{equation}
\begin{align*}
w_1&=\kappa^{-1/2}\cos\alpha t,& x_1&=\kappa^{-1/2}\sin\alpha t,\\ 
y_1&=0,& z_1&=0,\displaybreak[0]\\
w_2&=\kappa^{-1/2}\cos(\alpha t+a_2),& x_2&=\kappa^{-1/2}\sin(\alpha t+a_2),\\ 
y_2&=0,& z_2&=0,\displaybreak[0]\\
w_3&=\frac{1}{\sqrt{6}}\kappa^{-1/2}\cos(\alpha t+a_3),& x_3&=\frac{1}{\sqrt{6}}\kappa^{-1/2}\sin(\alpha t+a_3),\displaybreak[0]\\
y_3&=\frac{\sqrt{5}}{\sqrt{6}}\kappa^{-1/2}\cos\beta t,& 
z_3&=\frac{\sqrt{5}}{\sqrt{6}}\kappa^{-1/2}\sin\beta t,\displaybreak[0]\\
w_4&=\frac{1}{\sqrt{6}}\kappa^{-1/2}\cos(\alpha t+a_4),& x_4&=\frac{1}{\sqrt{6}}\kappa^{-1/2}\sin(\alpha t+a_4),\displaybreak[0]\\
y_4&=\frac{\sqrt{5}}{\sqrt{6}}\kappa^{-1/2}\cos(\beta t+2\pi/3),& 
z_4&=\frac{\sqrt{5}}{\sqrt{6}}\kappa^{-1/2}\sin(\beta t+2\pi/3), \displaybreak[0]\\
w_5&=\frac{1}{\sqrt{6}}\kappa^{-1/2}\cos(\alpha t+a_5),& x_5&=\frac{1}{\sqrt{6}}\kappa^{-1/2}\sin(\alpha t+a_5),\displaybreak[0]\\
y_5&=\frac{\sqrt{5}}{\sqrt{6}}\kappa^{-1/2}\cos(\beta t+4\pi/3),& 
z_5&=\frac{\sqrt{5}}{\sqrt{6}}\kappa^{-1/2}\sin(\beta t+4\pi/3),
\end{align*}
If we invoke Criterion \ref{fixed-doubly}, do a straightforward computation, and use the fact that the frequencies of the two rotations have the same size, i.e.\ $|\alpha|=|\beta|$, we can conclude that $\bf q$, given by \eqref{ellip-ellip-penta}, satisfies system \eqref{w10}, \eqref{x10}, \eqref{y10}, \eqref{z10} and condition (ii), so it is indeed a (doubly rotating) $\kappa$-positive elliptic-elliptic relative equilibrium generated from a fixed point, i.e.\ a solution of the curved 5-body problem for $\kappa>0$ and any value of $\alpha$ and $\beta$ with $ |\alpha|=|\beta|\ne 0$. 

A straightforward computation shows that the constants of the angular momentum are
$$
c_{wx}=\frac{5}{2}m\kappa^{-1}\alpha, \ c_{wy}=c_{wz}=c_{xy}=c_{xz}=0, \ c_{yz}=\frac{5}{2}m\kappa^{-1}\alpha
$$
for $\beta=\alpha$ and
$$
c_{wx}=\frac{5}{2}m\kappa^{-1}\alpha, \ c_{wy}=c_{wz}=c_{xy}=c_{xz}=0,\ c_{yz}=-\frac{5}{2}m\kappa^{-1}\alpha
$$
for $\beta=-\alpha$, which means that the bodies rotate around the origin of the coordinate system only relative to the planes $wx$ and $yz$.

\subsection{Pair of equilateral triangles with equal masses moving with distinct-size frequencies}\label{pair-equal-masses}

We now construct an example in the 6-body problem in which 3 bodies of equal masses move along a great circle at the vertices of an equilateral triangle, while the other 3 bodies of masses equal to those of the previous bodies move along a complementary circle of another great sphere, also at the vertices of an equilateral triangle. So consider $\kappa>0$, the equal masses 
$m_1=m_2=m_3=m_4=m_5=m_6=:m>0$, and the frequencies $\alpha, \beta$, which, in general, we can take as distinct, $\alpha\ne\beta$.
Then a candidate for a solution as described above has the form
\begin{equation}
{\bf q}=({\bf q}_1,{\bf q}_2, {\bf q}_3, {\bf q}_4, {\bf q}_5, {\bf q}_6),\ {\bf q}_i=(w_i,x_i,y_i,z_i),\ i=1,2,\dots,6,
\label{2-noneq-compl}
\end{equation}
\begin{align*}
w_1&=\kappa^{-\frac{1}{2}}\cos\alpha t, & x_1&=\kappa^{-\frac{1}{2}}\sin\alpha t,\\ 
y_1&=0, & z_1&=0,\displaybreak[0]\\\
w_2&=\kappa^{-\frac{1}{2}}\cos(\alpha t+2\pi/3), & x_2&=\kappa^{-\frac{1}{2}}\sin(\alpha t +2\pi/3),\\ y_2&=0, & z_2&=0,\displaybreak[0]\\\
w_3&=\kappa^{-\frac{1}{2}}\cos(\alpha t+4\pi/3), & x_3&=\kappa^{-\frac{1}{2}}\sin(\alpha t +4\pi/3),\\ y_3&=0, & z_3&=0,\displaybreak[0]\\
w_4&=0, & x_4&=0,\\ y_4&=\kappa^{-\frac{1}{2}}\cos\beta t, & z_4&=\kappa^{-\frac{1}{2}}\sin\beta t,\displaybreak[0]\\\
w_5&=0, & x_5&=0,\\ y_5&=\kappa^{-\frac{1}{2}}\cos(\beta t+2\pi/3), & z_5&=\kappa^{-\frac{1}{2}}\sin(\beta t+2\pi/3),\displaybreak[0]\\\
w_6&=0, & x_6&=0,\\ y_6&=\kappa^{-\frac{1}{2}}\cos(\beta t+4\pi/3), & z_6&=\kappa^{-\frac{1}{2}}\sin(\beta t+4\pi/3).
\end{align*}
For $t=0$, we obtain the fixed-point configuration specific to ${\mathbb S}_\kappa^3$, similar to the one constructed in Subsection \ref{specific-fixed}:
\begin{align*}
w_1&=0,& x_1&=\kappa^{-1/2},& y_1&=0,& z_1&=0,\displaybreak[0]\\\
w_2&=-\frac{\kappa^{-1/2}}{2},& x_2&=\frac{\sqrt{3}\kappa^{-1/2}}{2},& y_2&=0,& z_2&=0,\displaybreak[0]\\
w_3&=-\frac{\kappa^{-1/2}}{2},& x_3&=-\frac{\sqrt{3}\kappa^{-1/2}}{2},& y_3&=0,& z_3&=0,\displaybreak[0]\\
w_4&=0,& x_4&=0,& y_4&=\kappa^{-1/2},& z_4&=0,\displaybreak[0]\\
w_5&=0,& x_5&=0,& y_5&=-\frac{\kappa^{-1/2}}{2},& z_5&=\frac{\sqrt{3}\kappa^{-1/2}}{2},\displaybreak[0]\\
w_6&=0,& x_6&=0,& y_6&=-\frac{\kappa^{-1/2}}{2},& z_6&=-\frac{\sqrt{3}\kappa^{-1/2}}{2}.
\end{align*}
To prove that $\bf q$ given by \eqref{2-noneq-compl} is a solution of system \eqref{second}, we can therefore apply Criterion \ref{fixed-doubly}. A straightforward computation shows that the $4n$ conditions \eqref{w10}, \eqref{x10}, \eqref{y10}, \eqref{z10} are satisfied, and then we can observe that condition (i) is also verified because 
\begin{align*}
r_1&=r_2=r_3=\kappa^{-1/2}, & \rho_1&=\rho_2=\rho_3=0,\\
r_4&=r_5=r_6=0, & \rho_4&=\rho_5=\rho_6=\kappa^{-1/2}.
\end{align*}
Consequently $\bf q$ given by \eqref{2-noneq-compl} is a $\kappa$-positive elliptic-elliptic relative equilibrium of the 6-body problem given by system \eqref{second} with $n=6$ for any $\alpha,\beta\ne 0$. If $\alpha/\beta$ is rational, a case that corresponds to a set of frequency pairs that has measure zero in $\mathbb R^2$, the orbits are periodic. In general, however, $\alpha/\beta$ is irrational, so the orbits are quasiperiodic. Though quasiperiodic relative equilibria were already discovered
for the classical equations in $\mathbb R^4$, \cite{Al-Chen}, \cite{Chen-Perez},
this is the first example of such an orbit in a $3$-dimensional space.

A straightforward computation shows that the constants of the angular momentum integrals are 
$$
c_{wx}=3m\kappa^{-1}\alpha\ne 0,\ \ c_{yz}=3m\kappa^{-1}\beta\ne 0, \ \
c_{wy}=c_{wz}=c_{xy}=c_{xz}=0,
$$
which means that the rotation takes place around the origin of the coordinate system only relative to the planes $wx$ and $yz$.

Notice that, in the light of \cite{DiacuPol}, the kind of example constructed here in the 6-body problem can be easily generalized to any $(n+m)$-body problem, $n,m\ge 3$, of equal masses, in which $n$ bodies rotate along a great circle of a great sphere at the vertices of a regular $n$-gon, while the other $m$ bodies rotate along a complementary great circle of another great sphere at the vertices of a regular $m$-gon. The same as in the 6-body problem discussed here, the rotation takes place around the origin of the coordinate system only relative to 2 out of 6 reference planes.

\subsection{Pair of scalene triangles with non-equal masses moving with distinct frequencies}\label{alpha-beta}

We will now extend the example constructed in Subsection \ref{pair-equal-masses} to non-equal masses. The idea is the same as the one we used in Subsection \ref{2-ell-non-equal}, based on the results proved in \cite{DiacuPol}, according to which, given an acute scalene triangle inscribed in a great circle of a great sphere, we can find 3 masses such that this configuration forms a fixed point.
The difference is that we don't keep the configuration fixed here by assigning zero initial velocities, but make it rotate uniformly, thus leading to a relative equilibrium. In fact, since we are in a 6-body problem, 3 bodies of non-equal masses will rotate along a great circle of a great sphere at the vertices of an acute scalene triangle, while the other 3 bodies will rotate along a complementary great circle of another great sphere at the vertices of another acute scalene triangle, not necessarily congruent with the other one.

So consider the masses $m_1, m_2, m_3, m_4, m_5, m_6>0,$ which, in general, are not equal. Then a candidate for a solution as described above has the form
\begin{equation}
{\bf q}=({\bf q}_1,{\bf q}_2, {\bf q}_3, {\bf q}_4, {\bf q}_5, {\bf q}_6),\ {\bf q}_i=(w_i,x_i,y_i,z_i),\ i=1,2,\dots,6,
\label{2-triangles}
\end{equation}
\begin{align*}
w_1&=\kappa^{-\frac{1}{2}}\cos(\alpha t+a_1), & x_1&=\kappa^{-\frac{1}{2}}\sin(\alpha t+a_1),\displaybreak[0]\\\ 
y_1&=0, & z_1&=0,\displaybreak[0]\\\
w_2&=\kappa^{-\frac{1}{2}}\cos(\alpha t+a_2), & x_2&=\kappa^{-\frac{1}{2}}\sin(\alpha t +a_2),\\ y_2&=0, & z_2&=0,\displaybreak[0]\\\
w_3&=\kappa^{-\frac{1}{2}}\cos(\alpha t+a_3), & x_3&=\kappa^{-\frac{1}{2}}\sin(\alpha t +a_3),\\ y_3&=0, & z_3&=0,\displaybreak[0]\\
w_4&=0, & x_4&=0,\\ y_4&=\kappa^{-\frac{1}{2}}\cos(\beta t+b_4), & z_4&=\kappa^{-\frac{1}{2}}\sin(\beta t+b_4),\displaybreak[0]\\\
w_5&=0, & x_5&=0,\\ y_5&=\kappa^{-\frac{1}{2}}\cos(\beta t+b_5), & z_5&=\kappa^{-\frac{1}{2}}\sin(\beta t+b_5),\displaybreak[0]\\\
w_6&=0, & x_6&=0,\\ y_6&=\kappa^{-\frac{1}{2}}\cos(\beta t+b_6), & z_6&=\kappa^{-\frac{1}{2}}\sin(\beta t+b_6),
\end{align*}
where the constants $a_1,a_2,$ and $a_3$, with $0\le a_1<a_2<a_3<2\pi$, and $b_1,b_2,$ and $b_3$, with $0\le b_4<b_5<b_6<2\pi$, determine the shape of the first and second triangle, respectively.

Notice that for $t=0$, the position of the bodies is the fixed-point configuration described and proved to be as such in Subsection \ref{2-ell-non-equal}. Therefore we can apply Criterion \ref{fixed-doubly} to check whether $\bf q$ given in \eqref{2-triangles} is a $\kappa$-positive elliptic-elliptic relative equilibrium.  Again, as in Subsection \ref{2-ell-non-equal}, we can approach this problem in two ways. One is computational, and it consists of using the fact that the positions at $t=0$ form a fixed-point configuration to determine the relationships between the constants $m_1, m_2, m_3, a_1, a_2,$ and $a_3$, on one hand, and the constants $m_4,m_5,m_6, b_4, b_5,$ and $b_6$, on the other hand. It turns out that they reduce to conditions \eqref{w10}, \eqref{x10}, \eqref{y10}, and \eqref{z10}. Then we only need to remark that 
$$
r_1=r_2=r_3=\kappa^{-1/2}\ \ {\rm and}\ \ r_4=r_5=r_6=0,
$$
which means that condition (i) of Criterion \ref{fixed-doubly} is satisfied.
The other approach is to invoke again the result of \cite{DiacuPol} and a reasoning similar to the one we used to show that the position at $t=0$ is a fixed-point configuration. Both help us conclude that $\bf q$ given by \eqref{2-triangles} is a solution of system \eqref{second} for any $\alpha, \beta\ne 0$. 

Again, when $\alpha/\beta$ is rational, a case that corresponds to a negligible set of frequency pairs, the solutions are periodic. In the generic case, when $\alpha/\beta$ is irrational, the solutions are quasiperiodic.

A straightforward computation shows that the constants of the angular momentum integrals are 
$$
c_{wx}=(m_1+m_2+m_3)\kappa^{-1}\alpha\ne 0,\ \ c_{yz}=(m_1+m_2+m_3)\kappa^{-1}\beta\ne 0,
$$
$$
c_{wy}=c_{wz}=c_{xy}=c_{xz}=0,
$$
which means that the rotation takes place around the origin of the coordinate system only relative to the planes $wx$ and $yz$.

\section{Examples of $\kappa$-negative elliptic relative equilibria}

The class of examples we construct here is the analogue of the one presented in Subsection \ref{simply-Lag} in the case of the sphere, namely Lagrangian solutions (i.e.\ equilateral triangles) of equal masses for $\kappa<0$. In the
light of Remark \ref{remark-ellip0}, we expect that the bodies move on a 2-dimensional hyperboloid, whose curvature is not necessarily the same as the one of ${\mathbb H}_\kappa^3$. The existence of these orbits, and the fact that they occur only when the masses are equal, was first proved in \cite{Diacu1} and \cite{Diacu001}. So, in this case, we have $n=3$ and $m_1=m_2=m_3=:m>0$. The solution of the corresponding system \eqref{second} we are seeking is of the form
\begin{equation}
{\bf q}=({\bf q}_1,{\bf q}_2,{\bf q}_3), \ {\bf q}_i=(w_i,x_i,y_i,z_i),\ i=1,2,3,
\label{0elliptic2}
\end{equation}
\begin{align*}
w_1(t)&=r\cos\omega t,& x_1(t)&=r\sin\omega t,\\
 y_1(t)&=y\ ({\rm constant}),& z_1(t)&=z\ ({\rm constant}),\\
w_2(t)&=r\cos(\omega t+2\pi/3),& x_2(t)&=r\sin(\omega t+2\pi/3),\\
y_2(t)&=y\ ({\rm constant}),& z_2(t)&=z\ ({\rm constant}),\\
w_3(t)&=r\cos(\omega t+4\pi/3),& x_3(t)&=r\sin(\omega t+4\pi/3),\\
y_3(t)&=y\ ({\rm constant}),& z_3(t)&=z\ ({\rm constant}),
\end{align*}
with $r^2+y^2-z^2=\kappa^{-1}$. Consequently, for the equations occurring in Criterion \ref{cri-elliptic2}, we have
$$
r_1=r_2=r_3=:r, \ \ a_1=0, a_2=2\pi/3, a_3=4\pi/3,
$$
$$
y_1=y_2=y_3=:y\ ({\rm constant}),\ \ z_1=z_2=z_3=:z\ ({\rm constant}).
$$
Substituting these values into the equations \eqref{w3}, \eqref{x3}, \eqref{y3}, \eqref{z3}, we obtain either identities or the same equation as in Subsection \ref{simply-Lag}, namely
$$
\alpha^2=\frac{8m}{\sqrt{3}r^3(4-3\kappa r^2)^{3/2}}.
$$
Consequently, given $m>0, r>0$, and  $y, z$ with $r^2+y^2-z^2=\kappa^{-1}$ and $z>|\kappa|^{-1/2}$, we can always find two frequencies, 
$$
\alpha_1=\frac{2}{r}\sqrt{\frac{2m}{\sqrt{3}r(4-3\kappa r^2)^{3/2}}}\ \ {\rm and}\ \  \alpha_2=-\frac{2}{r}\sqrt{\frac{2m}{\sqrt{3}r(4-3\kappa r^2)^{3/2}}},
$$
such that system \eqref{second} has a solution of the form \eqref{0elliptic2}. The positive frequency corresponds to one sense of rotation, whereas the negative frequency corresponds to the opposite sense.

Notice that the bodies move on the 2-dimensional hyperboloid
$$
{\bf H}_{\kappa_0,y_0}^2=\{(w,x,y_0,z)\ \! | \ \! w^2+x^2-z^2=\kappa^{-1}-y_0^2,\ y_0={\rm constant}\},
$$
which has curvature $\kappa_0=-(y_0^2-\kappa^{-1})^{-1/2}$. When $y_0=0$, we have a great 2-dimensional hyperboloid, i.e.\ its curvature is $\kappa$, so the motion is in agreement with Remark \ref{remark-ellip}.

A straightforward computation shows that the constants of the angular momentum are
$$
c_{wx}=3m\kappa^{-1}\alpha\ne 0\ \ {\rm and}\ \ c_{wy}=c_{wz}=c_{xy}=c_{xz}=c_{yz}=0,
$$
which means that the rotation takes place around the origin of the coordinate system only relative to the plane $wx$.

\section{Examples of $\kappa$-negative hyperbolic relative equilibria}\label{3hyp}

In this section we will construct a class of $\kappa$-negative hyperbolic relative equilibria, for which, in agreement with Remark \ref{remark-hyp}, the
bodies rotate on a 2-dimensional hyperboloid of the same curvature as
${\mathbb H}_\kappa^3$. In the 2-dimensional case, the existence of a similar 
orbit was already pointed out in \cite{Diacu4}, where we have also proved it 
to be unstable. So let us check a solution of the form
\begin{equation}
{\bf q}=({\bf q}_1, {\bf q}_2, {\bf q}_3), \ {\bf q}_i=(w_i,x_i,y_i,z_i), \ i=1,2,3,
\label{hyp-4}
\end{equation}
\begin{align*}
w_1&=0,& x_1&=0,& y_1&=\frac{\sinh\beta t}{|\kappa|^{1/2}},& z_1&=\frac{\cosh\beta t}{|\kappa|^{1/2}},\\
w_2&=0,& x_2&=x\ {\rm (constant)},& y_2&=\eta\sinh\beta t,& z_2&=\eta\cosh\beta t,\\
w_3&=0,& x_3&=-x\ {\rm (constant)},& y_3&=\eta\sinh\beta t,& z_3&=\eta\cosh\beta t,
\end{align*}
with $x^2-\eta^2=\kappa^{-1}$. Consequently
$$
\eta_1=|\kappa|^{-1/2},\ \eta_2=\eta_3=:\eta \ {\rm (constant)},\ b_1=b_2=b_3=0.
$$
We then compute that
$$
\mu_{12}=\mu_{21}=\mu_{13}=\mu_{23}=-|\kappa|^{-1/2}\eta,\ \mu_{23}=\mu_{32}=-\kappa^{-1}-2\eta^2.
$$
We can now use Criterion \ref{crit-hyp} to determine whether a candidate of the form $\bf q$ given by \eqref{hyp-4} is a (simply rotating) $\kappa$-negative hyperbolic relative equilibrium. Straightforward computations lead us from equations \eqref{w4},
\eqref{x4}, \eqref{y4}, and \eqref{z4} either to identities or to the equation
$$
\beta^2=\frac{1-4\kappa\eta^2}{4\eta^3(|\kappa|\eta^2-1)^{3/2}}.
$$
Therefore, given $\kappa<0$ and $m, x,\eta>0$ with $x^2-\eta^2=\kappa^{-1}$,
there exist two non-zero frequencies, 
$$
\beta_1=\frac{1}{2\eta}\sqrt{\frac{1-4\kappa\eta^2}{\eta(|\kappa|\eta^2-1)^{3/2}}}\ \ {\rm and}\ \ \beta_2=-\frac{1}{2\eta}\sqrt{\frac{1-4\kappa\eta^2}{\eta(|\kappa|\eta^2-1)^{3/2}}},
$$
such that $\bf q$ given by \eqref{hyp-4} is a (simply rotating) $\kappa$-positive hyperbolic relative equilibrium. Notice that the motion takes place on the 2-dimensional hyperboloid
$$
{\bf H}_{\kappa,w}^2=\{(0,x,y,z)\ \! |\ \! x^2+y^2-z^2=\kappa^{-1}\}.
$$
A straightforward computation shows that the constants of the angular momentum are
$$
c_{wx}=c_{wy}=c_{wz}=c_{xy}=c_{xz}=0,\ c_{yz}=-m\beta(\kappa^{-1}+2\eta^2),
$$
which means that the rotation takes place relative to the origin of the coordinate system only relative to the plane $yz$.

\section{Examples of $\kappa$-negative elliptic-hyperbolic relative equilibria}

In this section we will construct a class of (doubly rotating) $\kappa$-negative elliptic-hyperbolic relative equilibria. In the light of Remark \ref{remark-ellip-hyp}, we expect that the motion cannot take place on any 2-dimensional hyperboloid of ${\mathbb H}_\kappa^3$. In fact, as we know from Theorem \ref{theorem3}, relative equilibria of this type may rotate on hyperbolic cylinders, which is also the case with the solution we introduce here.

Consider $\kappa<0$ and the masses $m_1=m_2=m_3=:m>0$. We will check a solution of the form  
\begin{equation}
{\bf q}=({\bf q}_1, {\bf q}_2, {\bf q}_3), \ {\bf q}_i=(w_i,x_i,y_i,z_i), \ i=1,2,3,
\label{hyp-3}
\end{equation}
\begin{align*}
w_1&=0,& x_1&=0,& y_1&=\frac{\sinh\beta t}{|\kappa|^{1/2}},& z_1&=\frac{\cosh\beta t}{|\kappa|^{1/2}},\\
w_2&=r\cos\alpha t,& x_2&=r\sin\alpha t,& y_2&=\eta\sinh\beta t,& z_2&=\eta\cosh\beta t,\\
w_3&=-r\cos\alpha t,& x_3&=-r\sin\alpha t,& y_3&=\eta\sinh\beta t,& z_3&=\eta\cosh\beta t.
\end{align*}

In terms of the form \eqref{elliptic-hyperbolic-summarized} of a hyperbolic solution, \eqref{hyp-3} is realized when
$$
r_1=0, r_2=r_3=:r,\ \eta_1=|\kappa|^{-1/2}, \eta_2=\eta_3=:\eta, 
$$
$$
a_1=a_2=0, a_3=\pi,\ {\rm and}\  b_1=b_2=b_3=0.
$$
Substituting these values into the equations \eqref{w5}, \eqref{x5}, \eqref{y5}, \eqref{z5} of Criterion \ref{crit-elliptic-hyp} and using the fact that $r^2-\eta^2=\kappa^{-1}$, we obtain the equation
$$
\alpha^2+\beta^2=\frac{m(4|\kappa|\eta^2+1)}{4\eta^3(-\kappa\eta^2-1)^{3/2}},
$$
for which there are infinitely many values of $\alpha$ and $\beta$ that satisfy it. Therefore, for any $\kappa<0$, masses $m_1=m_2=m_3=:m>0$, and $r,\eta$ with $r^2-\eta^2=\kappa^{-1}$, there are infinitely many frequencies $\alpha$ and $\beta$ that correspond to a $\kappa$-negative elliptic-hyperbolic relative equilibrium of the form \eqref{hyp-3}.

The bodies $m_2$ and $m_3$ move on the same hyperbolic cylinder, namely
${\bf C}_{r\eta}^2$, which has constant positive curvature, while $m_1$ moves on the degenerate hyperbolic cylinder ${\bf C}_{0|\kappa|^{-1/2}}^2$, which is a geodesic hyperbola, therefore has zero curvature. The motion is neither periodic nor quasiperiodic.

A straightforward computation shows that the constants of the angular momentum are 
$$c_{wx}=2m\alpha r^2, c_{yz}=-|\kappa|^{-1}-2\beta\eta^2,
c_{wy}=c_{wz}=c_{xy}=c_{xz}=0,
$$
which means that the rotation takes place around the origin of the coordinate system only relative to the $wx$ and $yz$ planes.

\newpage

\part{\rm CONCLUSIONS}

In this final part of the paper, we aim to emphasize 3 aspects related to the curved $n$-body problem. We will first analyze the issue of stability of orbits and then present some open problems that arise from the results we obtained here in the direction of relative equilibria and the stability of periodic and quasiperiodic solutions.

\section{Stability}\label{stab}

An important problem to explore in celestial mechanics is that of the stability of solutions of the $n$-body problem. Unless the celestial orbits proved to exist have some stability properties, they will not be found in the universe. Of course, stability is only a necessary, but not a sufficient, condition for the existence of certain orbits in nature. If the solutions require additional properties, such as equal masses, then they are unlikely to show up in the universe either. Therefore some of the orbits mathematically proved to be stable can be observed in our planetary system, such as the Lagrangian relative equilibria formed by the Sun, Jupiter, and the Trojan asteroids, when one mass is very small. Other orbits, such as the figure eight solutions, \cite{Chenciner}, which require equal masses, have not been discovered in our solar system, although we have strong numerical evidence for their stability. No unstable orbit has ever been observed.

From the mathematical point of view, it is important to distinguish between the several levels of stability that appear in the literature. Liapunov stability, for instance, is an unlikely property in celestial mechanics. No expert in the field expects to find Liapunov stable orbits too often. The reason for this scepticism is that not even the classical elliptic orbits of the Euclidean Kepler problem posses this property. Indeed, think of two elliptic orbits that are close to each other, each of them accommodating a body of the same mass. Then the body on the larger ellipse moves slower than the other, so, after enough time elapses, the body moving on the inner orbit will get well ahead of the body moving on the outer orbit, thus proving that Liapunov stability cannot take place. Therefore the most we usually hope for in celestial mechanics from this point of view is orbital stability. However, most of the time we are happy to find linearly stable solutions.

In October 2010, I asked Carles Sim\'o whether he would like to investigate numerically the stability of the Lagrangian solutions of the curved 3-body problem. He accepted the challenge in January 2011, and came up with much more than some numerical insight. Together with Regina Mart\'inez, he studied the case of 3 equal masses, $m_1=m_2=m_3=1$, for the sphere ${\mathbb S}^2$ of radius 1. In the end, they found an analytic proof of the following result.

\begin{theorem} {\bf (Mart\'inez-Sim\'o)}\label{martinez-simo}
Consider the Lagrangian solution, with masses $m_1=m_2=m_3=1$ of the 2-dimensional curved 3-body problem given by system \eqref{second} with $n=3$ and $\kappa=1$, i.e.\ on the sphere ${\mathbb S}^2$ of radius 1, embedded in $\mathbb R^3$, with a system of coordinates $x,y,z$. Let the motion take place
on a non-geodesic circle of radius $r$ in the plane $z=$ constant, with $0<z=\sqrt{1-r^2}<1$. Then the orbits are linearly stable (or totally elliptic) for $r\in(r_1,r_2)\cup(r_3,1)$ and linearly unstable for $r\in(0,r_1)\cup(r_2, r_3)$, where, with a 35-digit approximation,
$$
r_1 = 0.55778526844099498188467226566148375,
$$
$$
r_2 = 0.68145469725865414807206661241888645,
$$
$$
r_3 = 0.92893280143637470996280353121615412.
$$
In the unstable domains, the local behaviour around the orbits consists of two elliptic planes and a complex saddle. At the values $r_1, r_2$, and $r_3$, there occur Hamiltonian-Hopf bifurcations.
\end{theorem}

In the proof of the above theorem, the only numerical computation is that of $r_1, r_2$, and $r_3$, as roots of a certain characteristic polynomial. But the existence of these roots is easy to demonstrate, so the proof is entirely analytic, \cite{simo}. 

As expected, when $r$ is close to 0, i.e.\ when approaching the Euclidean plane, the Lagrangian orbit is unstable, as it happens with the classical Lagrangian solution of equal masses. In the Euclidean case, the Lagrangian orbits are stable only when one mass is very small. It is therefore surprising to discover two zones of stability, which occur in the intervals $(r_1,r_2)$ and $(r_3,1)$, in the curved problem. This important result implies that the curvature of the space has decisive influence over the stability of orbits, a fact that was not previously known.

\section{Future perspectives}

More than 170 years after Bolyai and Lobachevsky suggested the study of gravitational motions in hyperbolic space, the recent emergence of the $n$-body problem in spaces of constant curvature, given by a system of differential equations set in a context that unifies the positive and the negative case, provides a new and widely open area of research. The results obtained in this paper raise several interesting questions, which I will outline below. 

\subsection{Relative equilibria} 

As shown here, the curved $n$-body problem has 5 natural types of relative equilibria: 2 in ${\mathbb S}_\kappa^3$ and 3 in ${\mathbb H}_\kappa^3$. Criteria 1 through 7 provide the necessary conditions for the existence of these relative equilibria. Let us consider the masses $m_1,m_2,\dots, m_n>0,\ n\ge 3$. We have seen in the above examples that if the geometric configuration of a relative equilibrium has a certain size, and the orbit is not generated from a fixed-point configuration, then we can find 2 frequencies (equal in size and of opposite signs, $\alpha$ and $-\alpha$) in the case of orbits with a single rotation. If such a relative equilibrium is generated from a fixed-point configuration, any non-zero frequency yields a solution. When a double rotation takes place, we can find two distinct frequencies, $\alpha$ and $\beta$. Let us factorize the set of these relative equilibria by geometric size and by the value of the frequencies. In other words, all we care about is the shape of the relative equilibrium. While in the Euclidean case there is only one way to look at the shape of a configuration, the concept is a bit more complicated here since we discovered orbits, such as those of the curved 6-body problem when 3 bodies move along a great circle of a great sphere of ${\mathbb S}_\kappa^3$ and the other 3 bodies move along a complementary great circle of another great sphere, where there is not only one shape to take into consideration. Nevertheless, a decomposition of the orbit into a finite number of shapes solves the problem.

We can now ask how many classes of relative equilibria exist in various contexts, such as for each of the 5 types of relative equilibria proved to exist or for $n=3,4, 5,\dots$ If we restrict to simply asking whether the number of such classes is finite or infinite, and if it is infinite, whether the set is discrete or contains a continuum, then the question extends Smale's Problem 6, discussed at the beginning of Part 5, to spaces of constant curvature. But even finding new classes of relative equilibria for various values of $n$ and $m_1,m_2,\dots,m_n>0$ is a problem worth pursuing.

\subsection{Stability of periodic and quasiperiodic orbits}

We referred in Section \ref{stab} to the stability of Lagrangian solutions in the case
of curvature $\kappa=1$. Since Theorem \ref{martinez-simo} is proved only for that particular case, it would be interesting to solve the general problem in order to understand how the stability of these orbits changes for various values of the curvature $\kappa$, using a method similar to the one Mart\'inez and Sim\'o developed in \cite{simo}. The same method can be further applied to all the periodic orbits we constructed in Part 5 of this paper. Moreover, Mart\'inez and Sim\'o analyzed in \cite{simo}, for $\kappa=1$, the stability of a Lagrangian homographic orbit, which turns out not to be periodic, but quasiperiodic, as shown in \cite{Diacu4}. Again their method could be further applied to other quasiperiodic orbits, such as the relative equilibria constructed in Subsections \ref{pair-equal-masses} and \ref{alpha-beta}, towards getting a better understanding of how curvature affects the stability of orbits. In
the mean time, it has been used to determine the stability of tetrahedral orbits in
$\mathbb S^2$, where one body, of mass $M$, is fixed at the north pole, while
the other bodies, all of mass $m$, lie at the vertices of an equilateral triangle that rotates in a plane parallel with the plane of the equator, \cite{DiacuMartinez}. This paper also opens new avenues, such as the study of the orbit's stability in $\mathbb S^3$.

\bigskip 

\noindent{\bf Acknowledgment.} The research presented above was supported in part by a Discovery Grant from NSERC of Canada. I am indebted to the anonymous referee for his/her thorough report, which helped me improve the content of this paper.

\newpage


\end{document}